\documentclass[letterpaper,10pt]{article}

\usepackage[letterpaper,left=1in,right=1in,top=1.5in,bottom=1.5in]{geometry}
\usepackage{bbm}
\newcommand{\myauthor}{}
\newcommand{\mytitle}{Cartier modules and cyclotomic spectra}

\title{\mytitle}

\author{Benjamin Antieau and Thomas Nikolaus}
\date{\today}
 
\usepackage[pdfstartview=FitH,
            pdfauthor={\myauthor},
            pdftitle={\mytitle},
            colorlinks,
            linkcolor=reference,
            citecolor=citation,
            urlcolor=e-mail,
            backref]{hyperref}

\usepackage{amsmath}
\usepackage{amscd}
\usepackage{amsbsy}
\usepackage{amssymb}
\usepackage{verbatim}
\usepackage{eufrak}
\usepackage{eucal}
\usepackage{microtype}
\usepackage{hyperref}
\usepackage{mathrsfs}
\usepackage{amsthm}
\usepackage{stmaryrd}
\usepackage[all,cmtip]{xy}
\usepackage{tikz}
\usetikzlibrary{matrix,arrows}
\usepackage[abbrev,lite,alphabetic]{amsrefs}
\usepackage{dsfont}

\usepackage{fancyhdr}
\usepackage{makeidx}
\makeindex

\usepackage{color}
\definecolor{todo}{rgb}{1,0,0}
\definecolor{conditional}{rgb}{0,1,0}
\definecolor{e-mail}{rgb}{0,.40,.80}
\definecolor{reference}{rgb}{.20,.60,.22}
\definecolor{mrnumber}{rgb}{.80,.40,0}
\definecolor{citation}{rgb}{0,.40,.80}

\let\oldmarginpar\marginpar
\renewcommand\marginpar[1]{\-\oldmarginpar[\raggedleft\footnotesize #1]%
{\raggedright\footnotesize #1}}

\newcommand{\Cscr}{\mathcal{C}}
\newcommand{\Dscr}{\mathcal{D}}

\newcommand{\Fscr}{\mathcal{F}}

\newcommand{\Hscr}{\mathcal{H}}

\newcommand{\Nscr}{\mathcal{N}}
\newcommand{\Oscr}{\mathcal{O}}

\newcommand{\Sscr}{\mathcal{S}}

\newcommand{\Xscr}{\mathcal{X}}

\newcommand{\B}{\mathrm{B}}
\newcommand{\C}{\mathrm{C}}
\newcommand{\D}{\mathrm{D}}
\newcommand{\E}{\mathrm{E}}
\newcommand{\F}{\mathrm{F}}
\renewcommand{\H}{\mathrm{H}}
\newcommand{\K}{\mathrm{K}}
\renewcommand{\L}{\mathrm{L}}
\newcommand{\M}{\mathrm{M}}

\renewcommand{\P}{\mathrm{P}}
\newcommand{\R}{\mathrm{R}}
\renewcommand{\S}{\mathrm{S}}
\newcommand{\T}{\mathrm{T}}

\newcommand{\W}{\mathrm{W}}
\newcommand{\Z}{\mathrm{Z}}

\renewcommand{\1}{\mathbbm{1}}
\renewcommand{\AA}{\mathds{A}}

\newcommand{\EE}{\mathds{E}}
\newcommand{\FF}{\mathds{F}}
\newcommand{\GG}{\mathds{G}}

\newcommand{\NN}{\mathds{N}}

\newcommand{\QQ}{\mathds{Q}}

\renewcommand{\SS}{\mathds{S}}
\newcommand{\TT}{\mathds{T}}
\newcommand{\ZZ}{\mathds{Z}}

\renewcommand{\mathds}[1]{\mathbb{#1}}

\newcommand{\BMS}{\mathrm{BMS}}
\newcommand{\gr}{\mathrm{gr}}
\newcommand{\Fr}{\mathrm{Fr}}
\newcommand{\crys}{\mathrm{crys}}
\newcommand{\dual}{\mathrm{dual}}

\newcommand{\triv}{\mathrm{triv}}

\newcommand{\cyc}{\mathrm{cyc}}

\newcommand{\op}{\mathrm{op}}

\newcommand{\Sp}{\mathrm{Sp}}
\newcommand{\CycSp}{\mathbf{CycSp}}
\newcommand{\LEq}{\mathrm{LEq}}
\newcommand{\Eq}{\mathrm{Eq}}

\newcommand{\HKR}{\mathrm{HKR}}
\newcommand{\fib}{\mathrm{fib}}
\newcommand{\cofib}{\mathrm{cofib}}

\newcommand{\gen}{\mathrm{gen}}
\newcommand{\bT}{\mathbb{T}}

\newcommand{\Nm}{\mathrm{Nm}}

\newcommand{\xto}{\xrightarrow}

\newcommand{\can}{\mathrm{can}}

\newcommand{\heart}{\heartsuit}

\newcommand{\ev}{\mathrm{ev}}

\DeclareMathOperator{\id}{id}

\newcommand{\coker}{\mathrm{coker}}

\renewcommand{\geq}{\geqslant}
\renewcommand{\leq}{\leqslant}

\newcommand{\Cat}{\mathrm{Cat}}
\newcommand{\Ab}{\mathrm{Ab}}

\newcommand{\TCart}{\mathbf{TCart}}
\newcommand{\Cart}{\mathbf{Cart}}

\DeclareMathOperator{\Tor}{Tor}

\newcommand{\THH}{\mathrm{THH}}

\newcommand{\TP}{\mathrm{TP}}
\newcommand{\TC}{\mathrm{TC}}

\newcommand{\TR}{\mathrm{TR}}
\newcommand{\HH}{\mathrm{HH}}

\newcommand{\KU}{\mathrm{KU}}

\DeclareMathOperator{\Br}{Br}

\DeclareMathOperator{\Fun}{Fun}

\DeclareMathOperator{\Hom}{Hom}
\newcommand{\Map}{\mathrm{Map}}
\newcommand{\MapSp}{\mathbf{Map}}
\newcommand{\EndSp}{\mathbf{End}}

\DeclareMathOperator{\Shv}{Shv}

\newcommand{\Mod}{\mathrm{Mod}}

\newcommand{\Alg}{\mathrm{Alg}}
\newcommand{\CoAlg}{\mathrm{CoAlg}}
\newcommand{\CAlg}{\mathrm{CAlg}}

\newcommand{\Ga}{\mathds{G}_{a}}

\DeclareMathOperator*{\colim}{colim}

\newcommand{\proet}{\mathrm{pro\acute{e}t}}

\DeclareMathOperator{\Spec}{Spec}

\newcommand{\we}{\simeq}
\newcommand{\iso}{\cong}

\theoremstyle{plain}
\newtheorem{theorem}{Theorem}[section]
\newtheorem*{theorem*}{Theorem}
\newtheorem{lemma}[theorem]{Lemma}

\newtheorem{proposition}[theorem]{Proposition}

\newtheorem{corollary}[theorem]{Corollary}
\newtheorem*{corollary*}{Corollary}

\theoremstyle{plain}
\newtheorem{maintheorem}{Theorem}

\newtheorem{maincorollary}[maintheorem]{Corollary}

\theoremstyle{definition}

\newtheorem{mainexample}[maintheorem]{Example}
\newtheorem{mainremark}[maintheorem]{Remark}

\newtheoremstyle{named}{}{}{\itshape}{}{\bfseries}{.}{.5em}{#1 \thmnote{#3}}
\theoremstyle{named}

\theoremstyle{definition}
\newtheorem{definition}[theorem]{Definition}
\newtheorem{warning}[theorem]{Warning}
\newtheorem{variant}[theorem]{Variant}

\newtheorem{example}[theorem]{Example}
\newtheorem*{example*}{Example}

\newtheorem*{question*}{Question}
\newtheorem{construction}[theorem]{Construction}
\newtheorem{remark}[theorem]{Remark}
\newtheorem{remarks}[theorem]{Remarks}

\begin{document}

\maketitle

\begin{abstract}
    \noindent
    We construct and study a $t$-structure on $p$-typical cyclotomic spectra
    and explain how to recover crystalline cohomology of smooth schemes over
    perfect fields using this $t$-structure. Our main tool
    is a new approach to $p$-typical cyclotomic spectra via objects we call $p$-typical topological
    Cartier modules. Using these, we prove that the heart of the cyclotomic $t$-structure
    is the full subcategory of derived $V$-complete objects in the abelian category of
    $p$-typical Cartier modules.


    \paragraph{Mathematics Subject Classification 2010.}
    \href{http://www.ams.org/mathscinet/msc/msc2010.html?t=14Fxx&btn=Current}{14F30},
    \href{http://www.ams.org/mathscinet/msc/msc2010.html?t=14Lxx&btn=Current}{14L05},
    \href{http://www.ams.org/mathscinet/msc/msc2010.html?t=13Dxx&btn=Current}{13D03}.
\end{abstract}

\tableofcontents

\section{Introduction}\label{sec:intro}

The notion of a cyclotomic spectrum was introduced
in~\cite{bokstedt-hsiang-madsen}. The importance of cyclotomic spectra in arithmetic
contexts was understood in the work of Hesselholt and his collaborators (see for
example~\cite{hesselholt-ptypical,geisser-hesselholt-1,hesselholt-madsen-local,hesselholt-local,hesselholt-tp}).
The homotopy theory of cyclotomic spectra is more recent and was developed by
Kaledin~\cite{kaledin-motivic,kaledin-cyclotomic},
Blumberg--Mandell~\cite{blumberg-mandell-cyclotomic}, and
Barwick--Glasman~\cite{barwick-glasman-cyclonic}.
Using cyclotomic spectra,
Bhatt, Morrow, and Scholze give in~\cite{bms2} a topological construction of
(completed) prismatic cohomology theories, which generalize
crystalline and $A\Omega$ cohomology. Their work relies
on work of the second author with P. Scholze~\cite{nikolaus-scholze}, which provides a simple description of bounded
below cyclotomic spectra. In this paper, we give another way to understand
cyclotomic spectra, which is better suited to answering the question: what are
the building blocks of a cyclotomic spectrum? 

\subsection{Statement of results}

Fix a prime number $p$. A {\bf $p$-typical Cartier module} is an abelian group
$M$ equipped with endomorphisms $V$ and $F$ such that $FV=p$. The building blocks of $p$-typical cyclotomic spectra
are, in a precise sense, certain $p$-typical Cartier modules. 

\begin{maintheorem}[see Theorems~\ref{thm:cyclotomict} and~\ref{thm:heartid}]\label{mt:t}
    The $\infty$-category $(\CycSp_p)_{\geq 0}$ of connective $p$-typical cyclotomic spectra is the connective part of
    a $t$-structure on $\CycSp_p$, the $\infty$-category
    of $p$-typical cyclotomic spectra. The heart
    $\CycSp_p^\heart$ is equivalent to the abelian category of derived $V$-complete Cartier
    modules.
\end{maintheorem}

We say that a $p$-typical Cartier module $M$ is
derived $V$-complete if the natural map $M\rightarrow\lim_n M/V^n$ is an
equivalence in $\Dscr(\ZZ)$.\footnote{Unless otherwise specified, all quotients $M/V^n$ are computed in
the derived sense (and hence are given as the cofiber of $M\xrightarrow{V^n}M$ in the derived category)
as are all limits.}

The existence and uniqueness of such a $t$-structure is a formal consequence of the fact that
$(\CycSp_p)_{\geq 0}$ is presentable and is closed under colimits and
extensions in $\CycSp_p$. The difficult part of the theorem is the identification of the heart.

Recall B\"okstedt's theorem, which says that $\pi_*\THH(\FF_p)= \FF_p[b]$, a
polynomial ring on a degree $2$ generator. More generally, using the vanishing of the cotangent
complex, one deduces that $\pi_*\THH(k)=  k[b]$ for any perfect ring $k$.
Our interest in the cyclotomic $t$-structure was piqued by the discovery of
the next result.

\begin{maintheorem}[see Theorem~\ref{thm:discrete}]\label{mt:discrete}
    If $k$ is a perfect ring of characteristic $p$, then
    $\THH(k)\in\CycSp^\heart_p$.
\end{maintheorem}

\newcommand{\Gmf}{\widehat{\GG}_m}

Despite the higher homotopy groups, $\THH(k)$ is discrete as a cyclotomic spectrum.
On the Cartier module side of the story, when $k$ is a perfect ring of
characteristic $p$, $\THH(k)$ corresponds to $W(k)$,
the ring of $p$-typical Witt vectors over $k$, with its Witt vector Verschiebung
and Frobenius operations. The fact that $\THH(k)$ is in $\CycSp^\heart_p$ is
consistent with the fact, due to
Hesselholt--Madsen~\cite{hesselholt-madsen-1}*{Theorem~B} for perfect
fields of characteristic $p$, that $\pi_i\TC(k)=0$ for $i>0$. However, the theorem
is much stronger. It says that for any cyclotomic spectrum $X$ with $\pi_iX=0$
for $i<0$ one has $\Hom_{\CycSp_p}(X[i],\THH(k))=0$ for $i>0$.
To reconcile the fact that $\THH(k)$ is not at all discrete as a
spectrum or even as a spectrum with $S^1$-action with the fact that it is
discrete as a $p$-typical cyclotomic spectrum, observe that the
$S^1$-equivariant map
$b\colon\THH(k)[2]\rightarrow\THH(k)$ is not a map of cyclotomic spectra.

For any $p$-typical cyclotomic spectrum $X$, the homotopy groups with respect
to the $t$-structure of Theorem~\ref{mt:t} are denoted by $\pi_i^\cyc
X$. These are objects of $\CycSp_p^\heart\subseteq\CycSp_p$. Thus, they can be
considered either as $p$-typical cyclotomic spectra, with underlying spectrum
with $S^1$-action and Frobenius $\varphi\colon\pi_i^\cyc
X\rightarrow(\pi_i^\cyc X)^{tC_p}$, or as derived $V$-complete $p$-typical Cartier modules under the equivalence of Theorem \ref{mt:t}.
We will typically not distinguish notationally between these two points of view.
Write $\CycSp_{[m,n]}=(\CycSp_p)_{\geq
m}\cap(\CycSp_p)_{\leq n}$. The objects of $\CycSp_{[m,n]}$ are $p$-typical
cyclotomic spectra $X$, which are bounded with respect to the cyclotomic
$t$-structure, and such that $\pi_i^\cyc X\we 0$ for $i\notin[m,n]$.

Now, we discuss some of the consequences for schemes, especially in
characteristic $p$.
Recall the classical Hochschild--Kostant--Rosenberg theorem~\cite{hkr}, which states that when $k$ is a commutative ring and $R$ is
a smooth $k$-algebra, there is a natural isomorphism $\HH_i(R/k)\iso\Omega^i_{R/k}$.
Say that a commutative $k$-algebra $R$ is ind-smooth if it is a filtered colimit of smooth
$k$-algebras. If $R$ is ind-smooth, then $\L_{R/k}$ is a flat $R$-algebra and the HKR
theorem continues to hold. The next result is a reinterpretation of a theorem of
Hesselholt~\cite{hesselholt-ptypical}*{Theorem~C} in the context of the cyclotomic
$t$-structure.

\begin{maintheorem}[see Theorem~\ref{thm:derhamwitt}]\label{mt:derhamwitt}
    Let $k$ be a perfect ring of characteristic $p$.
    If $R$ is an ind-smooth commutative $k$-algebra, then for each $i$ there is a natural isomorphism
    $\pi_i^\cyc\THH(R)\iso\W\Omega^i_{R}$ of Cartier
    modules, where $\W\Omega^i_{R}$ is the $i$th term in the de Rham--Witt complex
    of $R$.
\end{maintheorem}

In particular, if $R$ is smooth and has relative dimension $d$ over a perfect
ring of characteristic $p$, then $\W\Omega^i_{R}\iso 0$ for $i>d$ and hence
$\THH(R)\in\CycSp_{[0,d]}$.

Let $k$ be a commutative ring. Then,
$\CycSp_{\THH(k)}=\Mod_{\THH(k)}(\CycSp_p)$ admits a $t$-structure with
$(\CycSp_{\THH(k)})_{\geq 0}\we\Mod_{\THH(k)}((\CycSp_p)_{\geq 0})$. We let
$\CycSp_{\THH(k)}^\heart$ denote the heart, which we identify in
Corollary~\ref{mt:pi0}.

\begin{maincorollary}[see Proposition~\ref{prop:descentss}
    and Theorem~\ref{thm:derhamwitt}]\label{mt:descentss}
    If $X$ is smooth and quasi-compact over a perfect field $k$ of
    characteristic $p$, then there
    is a convergent spectral sequence
    $$\E_2^{s,t}\iso\H^{-s}(X,\W\Omega^t_{\Oscr_X})\Rightarrow\pi_{s+t}^\cyc\THH(X)$$
    in the abelian category $\CycSp_{\THH(k)}^\heart$.
\end{maincorollary}

\begin{maincorollary}\label{mt:smoothbounded}
    Let $\Cscr$ be a smooth and proper dg category over a commutative ring $R$ which itself is
    smooth of relative dimension $d$ over a perfect ring $k$ of characteristic
    $p$. Then, $\THH(\Cscr)$ is bounded.
    Specifically, if
    $\THH(\Cscr)\in(\CycSp_p)_{\geq -e}$, then it is in
    $\CycSp_{[-e,d+e]}$.
\end{maincorollary}

\begin{proof}
    Note that if $\Cscr$ is smooth and proper over $R$, then $\THH(\Cscr)$ is a perfect
    $\THH(R)$-module spectrum and hence bounded below as a spectrum.
    It suffices to show that if $X\in(\CycSp_p)_{\geq d+e+1}$, then the mapping space
    $\Map_{\CycSp_p}(X,\THH(\Cscr))\we 0$. However, this is equivalent to
    $\Map_{\CycSp_{\THH(R)}}(X\otimes\THH(R),\THH(\Cscr))$ and hence to
    $\Map_{\CycSp_{\THH(R)}}(X\otimes\THH(\Cscr),\THH(R))$ using that $\THH(\Cscr)$ is self
    dual over $\THH(R)$. Now, $X\otimes\THH(\Cscr)$ is
    contained in $(\CycSp_p)_{\geq d+1}$, so the result follows from the fact that
    $\THH(R)\in(\CycSp_p)_{\leq d}$.
\end{proof}

\newcommand{\QSyn}{\mathrm{QSyn}}

Let $R$ be a quasisyntomic ring in the sense of~\cite{bms2}*{Definition~4.9},
meaning that $R$ is $p$-complete with bounded $p^\infty$-torsion and the
cotangent complex $\L_{R/\ZZ_p}$ has $p$-complete Tor-amplitude contained in
$[0,1]$.
Note that a characteristic $p$ ring $R$ is quasisyntomic if and only if $\L_{R/\FF_p}$ has
Tor-amplitude contained in $[0,1]$.

For $R$ quasisyntomic, Bhatt--Morrow--Scholze~\cite{bms2} use syntomic descent to construct
a filtration $\F^\star_{\BMS}\TP(R)$ on $\TP(R)$ whose graded pieces are given by an
absolute prismatic cohomology theory, suitably completed. In the special case
where $R$ is smooth over a perfect field of characteristic $p$, the graded
pieces are given by crystalline cohomology $\R\Gamma_{\crys}(R/W)$ by~\cite{bms2}*{Theorem~1.10}.

Again in the case of an ind-smooth algebra $R$ over a perfect field $k$ of
characteristic $p$, we can directly
construct a filtration on $\TP(R)$ whose graded pieces are given canonically by the de
Rham--Witt complex of de Rham and Illusie (see~\cite{illusie-derham-witt}) in the smooth affine case.

\begin{maincorollary}[see Theorem~\ref{thm:drw}]\label{mt:crystalline}
    Let $k$ be a perfect field of characteristic $p$ and let $R$ be an ind-smooth $k$-algebra and let
    $\TP(\tau_{\geq\star}^\cyc\THH(R))$ be the filtration on $\TP(R)$ induced by the
    cyclotomic Whitehead tower. The induced Whitehead tower with respect to the
    Beilinson $t$-structure on filtered spectra defines
    a natural complete exhaustive multiplicative decreasing
    $\ZZ$-indexed filtration
    $\F^\star_\B\TP(R)$ in filtered spectra with graded pieces given by
    $$\gr^i_\B\TP(R)\we\W\Omega_R^\bullet[2i]$$ for all $i$. Moreover, this
    filtration agrees with the BMS filtration on $\TP(R)$ after forgetting the
    secondary filtration.
\end{maincorollary}

We call the filtration here the Beilinson filtration as it is constructed using the
Beilinson $t$-structure following an idea
of~\cite{antieau-derham}. Both the BMS filtration $\F^\star_{\BMS}\TP(R)$ and
the Beilinson filtration $\F^\star_\B\TP(R)$ are equipped with secondary
filtrations. For the BMS filtration, one obtains the Nygaard filtration
$\Nscr^{\geq\star}\R\Gamma(R/W)$ as the
residual filtration on the graded pieces, while for the Beilinson filtration,
one obtains the Hodge filtration $\W\Omega_R^{\geq\star}$ on the
graded pieces. Forgetting these secondary filtrations, the Beilinson and BMS
filtrations agree on $\TP(R)$. By remembering the secondary filtration in
$\F^\star_\B\TP(R)$, we recover a specific complex, the de Rham--Witt complex, which computes
crystalline cohomology.

\begin{mainexample}\label{mainexample}
    Let $X$ be a K3 surface over a perfect field $k$ of characteristic $p$, meaning a smooth and proper surface
    over $k$ such that the canonical bundle $\Omega_{X/k}^2$ is trivial and $\H^1(X,\Oscr_X)=0$.
    The formal Brauer group of $X$ is, by the work of
    Artin and Mazur~\cite{artin-mazur}, a commutative formal Lie group. The nature of the
    formal Brauer group stratifies K3 surfaces over $k$ into two types: the {\bf Hodge--Witt}
    K3 surfaces, for which $\widehat{\Br}_X$ is a formal $p$-divisible group of height $h$
    where $1\leq h\leq 10$, and the {\bf supersingular} K3 surfaces, where $\widehat{\Br}_X$
    is unipotent (and even isomorphic to $\widehat{\Ga}$, the formal completion of the
    additive group). When $X$ is supersingular, the $p$-typical Cartier module
    $\H^2(X,\W\Omega^0_X)$  associated to $\widehat{\Br}_X$
    is isomorphic to $\FF_p\llbracket x\rrbracket$, the $p$-typical Cartier module with $F=0$,
    $Vx^n=x^{n+1}$, and $ax=xa^p$ (see~\cite{illusie-derham-witt}*{7.2}). The
    spectral sequence of Corollary~\ref{mt:descentss} implies that
    $\pi_{-2}^\cyc\THH(X)\iso\H^2(X,\W\Omega^0_X)$ since there can be no
    differentials in or out. Now, the spectrum underlying $\pi_{-2}^\cyc(\THH(X))$ has homotopy groups $k$ in degree $0$ and $k\llbracket
    x\rrbracket$ in degrees $\geq 2$ by Figure~\ref{fig:homotopygroups} on
    page~\pageref{fig:homotopygroups}. In particular, it is not compact and
    hence not dualizable as a $\THH(k)$-module spectrum.
    Since the forgetful functor $\CycSp_{\THH(k)}\rightarrow\Dscr(\THH(k))$ is
    symmetric monoidal, it follows that $\pi_{-2}^\cyc\THH(X)$ is not dualizable as a cyclotomic
    spectrum when $X$ is a supersingular K3 surface.
    Thus, we see that the $t$-structure on $\CycSp_{\THH(k)}$ does not restrict to a
    $t$-structure on $\CycSp_{\THH(k)}^\dual$, the full subcategory of dualizable objects in
    $\CycSp_{\THH(k)}$. Moreover, as this example shows, $\THH(X)$ cannot be
    perfect in $\CycSp_{\THH(k)}$. Indeed, if it were perfect, then each
    cyclotomic homotopy group $\pi_i^\cyc\THH(X)$ would be a finitely presented
    $W(k)$-module, which this example shows is not the case. This recovers the
    counterexample of~\cite{amn1}.
\end{mainexample}

In general, despite the coincidence of the filtrations on $\TP(R)$ when $R$ is
a smooth ring over a perfect field $k$ of characteristic $p$, the BMS filtration and the cyclotomic
Whitehead tower do not agree on $\THH(R)$ when $R$ is quasisyntomic. In fact,
even when $R=k$ is a perfect field they differ. Indeed, the BMS filtration is
not a filtration by cyclotomic spectra. We have seen by
Theorem~\ref{mt:discrete} that the cyclotomic Whitehead tower for $\THH(k)$ is
concentrated in a single degree. But, the BMS filtration is given by
$\tau_{\geq 2\star}\THH(k)$, i.e., the classical Whitehead tower. The maps $\tau_{\geq
2\star}\THH(k)\rightarrow\THH(k)$ cannot be given the structure of cyclotomic
maps. Nevertheless, we have the following result, which Scholze suggested
to us.

\begin{maintheorem}[see Theorem~\ref{thm:syntomic}]\label{mt:syntomic}
    The BMS filtration and the $t$-structure filtration on $\TC(X)$ agree when $X$
    is a smooth quasi-compact scheme over a perfect field $k$ of characteristic
    $p$.
\end{maintheorem}

\subsection{Idea of proofs: topological Cartier modules}

To prove our theorems, we introduce the stable $\infty$-category $\TCart_p$ of {\bf $p$-typical topological
Cartier modules}.\footnote{In early talks on this project, we called these
topological Dieudonn\'e modules.} These are spectra $M$ equipped with an $S^1$-action together with an
$S^1$-equivariant factorization $$M_{hC_p}\xto{V}M\xto{F} M^{hC_p}$$ of the $C_p$-norm
$$M_{hC_p}\xto{\Nm_{C_p}}M^{hC_p}.$$ Maps are defined as usual to be maps of spectra with
$S^1$-action which commute with the $V$ and $F$ operations and the homotopy witnessing the factorization.

There is a $t$-structure on $\TCart_p$ where an object $M$ is in $(\TCart_p)_{\geq 0}$ (resp.
$(\TCart_p)_{\leq 0}$) if and only if $\pi_iM=0$ for $i<0$ (resp. $\pi_iM=0$ for
$i>0$). The heart of this $t$-structure is the abelian category of $p$-typical Cartier modules
introduced above. Given a $p$-typical topological Cartier module $M$, the cofiber of
$V$, which we write as $M/V$, naturally admits the structure of a $p$-typical
cyclotomic spectrum. On bounded
below objects with respect to this $t$-structure on $\TCart_p$ and the cyclotomic $t$-structure, we obtain the following theorem.

\begin{maintheorem}[see Theorem~\ref{thm:boundedbelow}]\label{mt:td}
    The functor $(-)/V\colon\TCart_p^-\rightarrow\CycSp_p^-$ admits a fully faithful
    $t$-exact right adjoint given
    by $\TR$. The essential image of $\TR$ is the full subcategory of
    bounded below $p$-typical topological Cartier modules $M$ such that $\pi_iM$ is derived
    $V$-complete for all $i$.
\end{maintheorem}

From Theorem~\ref{mt:td}, we can read off two things: first, that the heart of the
$t$-structure on cyclotomic spectra is given by derived $V$-complete $p$-typical topological
Cartier modules; second, that for a cyclotomic spectrum $X$, the cyclotomic homotopy groups
$\pi_i^\cyc X$ are given by $\pi_i\TR(X)$ equipped with the canonical $V$ and $F$
operations induced from transfer and inclusion of fixed points. This gives the proof of
Theorem~\ref{mt:t}. Previous calculations in $\TR$ of commutative rings of
Hesselholt--Madsen~\cite{hesselholt-madsen-1}*{Theorem~5.5} and
Hesselholt~\cite{hesselholt-ptypical}*{Theorem~C} then suffice to establish
Theorems~\ref{mt:discrete} and~\ref{mt:derhamwitt} for smooth algebras over
perfect fields, which we show is enough to prove the general case of each theorem. The next three corollaries are immediate consequences.

\begin{maincorollary}\label{mt:tate}
    If $X$ is a bounded below $p$-typical cyclotomic spectrum, then the
    natural $S^1$-equivariant map $\TR(X)\rightarrow X$ induces a $p$-adic equivalence
    $\TR(X)^{tS^1}\rightarrow X^{tS^1}$.
\end{maincorollary}

\begin{proof}
    By Theorem~\ref{mt:td}, the counit map $\TR(X)/\TR(X)_{hC_p}\rightarrow X$
    is an equivalence. In particular, we have a cofiber sequence
    $$\TR(X)_{hC_p}\rightarrow\TR(X)\rightarrow X.$$ Applying $(-)^{tS^1}$,
    we obtain a cofiber sequence $p$-adically equivalent to
    $$((\TR(X)_{hC_p})^{tC_p})^{hS^1}\rightarrow(\TR(X)^{tC_p})^{hS^1}\rightarrow(X^{tC_p})^{hS^1}$$
    by~\cite{nikolaus-scholze}*{Lemma~II.4.2}.
    But, $(\TR(X)_{hC_p})^{tC_p}\we 0$ by the Tate orbit
    lemma~\cite{nikolaus-scholze}*{Lemma~I.2.1}.
\end{proof}

\begin{maincorollary}\label{mt:conservativetr}
    If $X$ is a bounded below $p$-typical cyclotomic spectrum such that $\TR(X)\we 0$,
    then $X\we 0$.
\end{maincorollary}

\begin{proof}
    Indeed, since $\TR\colon\CycSp_p^-\rightarrow\TCart_p^-$ is fully faithful,
    the adjoint $\TR(X)/V\rightarrow X$ is an equivalence, so $X\we 0$.
\end{proof}

By~\cite{hesselholt-madsen-1}*{Theorem~F}, which says that $\pi_0\TR(k)\iso
W(k)$ when $k$ is a commutative ring, and using Theorem~\ref{mt:td}, we obtain
the following corollary.

\begin{maincorollary}[see Theorem~\ref{thm:pi0}]\label{mt:pi0}
    For any commutative ring $k$, $\pi_0^\cyc\THH(k)\we W(k)$. Moreover,
    $\CycSp_{\THH(k)}^\heart$ is equivalent to the abelian category of derived
    $V$-complete $W(k)$-modules in $p$-typical Cartier modules.
\end{maincorollary}

\begin{mainexample}\label{mt:homotopy groups}
    The cyclotomic homotopy groups $\pi_i^\cyc X$ of a cyclotomic spectrum $X$
    are given by derived $V$-complete $p$-typical  Cartier modules.
    In particular, any derived $V$-complete $p$-typical  Cartier
    module $M$ has an underlying cyclotomic spectrum $M/V$. Its homotopy groups
    are computed using the cofiber sequence $M_{hC_p}\xrightarrow{V} M\rightarrow
    M/V$. They are given in Figure~\ref{fig:homotopygroups}.
\end{mainexample}

\begin{figure}[h]
    \centering
        $$\pi_iM/V\iso\begin{cases}
            \coker(M\xrightarrow{V}M)&\text{if $i=0$,}\\
            \ker(M\xrightarrow{V}M)&\text{if $i=1$,}\\
            \coker(M\xrightarrow{p}M)&\text{if $i\geq 2$ is even, and}\\
            \ker(M\xrightarrow{p}M)&\text{if $i\geq 2$ is odd.}
        \end{cases}$$
        \caption{This table gives the homotopy groups of the cyclotomic spectrum $M/V$ if $M$ is a $p$-typical Cartier module (i.e., an object of
        $\TCart_p^\heart$).}
    \label{fig:homotopygroups}
\end{figure}

\begin{mainexample}\label{mt:pdivisiblegroups}
    The theory of $p$-typical Cartier modules arises in the study of
    commutative formal groups over commutative $\ZZ_{(p)}$-algebras
    (see~\cite{zink}). We describe briefly the connection in characteristic
    $p$. Let $k$ be a perfect field of characteristic $p$. Because $VF=p$ on
    $W(k)$, one finds that $VF=p$ on any $W(k)$-module in
    $\TCart_p^\heart$ (see Remark~\ref{rem:dieudonne}). A Dieudonn\'e module is
    an abelian group $M$ with endomorphisms $F$ and $V$ such that $FV=VF=p$.
    Let $G$ be a finite flat group scheme and let $\M(G)$ denote the covariant
    Dieudonn\'e module (the dual of the construction
    in~\cite{demazure}*{Chapter~III}). We find that if $G$ is infinitesimal
    (like $\alpha_p$) or multiplicative (like $\mu_p$), then $V$ is nilpotent
    and thus $\M(G)$ is derived $V$-complete. Similarly, if $G$ is a
    formal $p$-divisible group, then $\M(G)$ is a derived $V$-complete
    Dieudonn\'e module. In particular, we obtain a functor $$\{\text{formal
    $p$-divisible groups over $k$}\}\rightarrow\CycSp_{\THH(k)}^\heart.$$
    We will see in~\cite{an2} that this functor induces an equivalence between
    the bounded derived $\infty$-category of isogeny
    classes of $p$-divisible groups and isogeny classes of dualizable
    cyclotomic spectra over perfect fields.
\end{mainexample}

The $t$-structures on $\TCart_p$ and $\CycSp_p$ are compatible with the natural
symmetric monoidal structures, so we obtain induced symmetric monoidal structures
on the abelian categories of $p$-typical Cartier modules and derived
$V$-complete $p$-typical Cartier modules, respectively.
When $k$ is a perfect field of characteristic $p$, Goerss~\cite{goerss-hopf} had previously constructed
a symmetric monoidal structure on $W(k)$-modules in $p$-typical Cartier modules. We
prove in Section~\ref{sec:monoidal} that our symmetric monoidal structure
agrees with his in this case. In future work we will study these symmetric
monoidal structures more closely and deduce generalizations of the HKR theorem
(Theorem \ref{mt:derhamwitt}) to non-perfect rings.

\begin{mainremark}
In Section~\ref{sec:genuine}, we show that
$p$-typical topological Cartier modules are to genuine fixed points as genuine
cyclotomic spectra are to geometric fixed points. Let $\bT\Sp_p^\gen$ denote
the $\infty$-category of genuine $S^1$-spectra with respect to the finite
$p$-subgroups of $S^1$. The $\infty$-category $\TCart_p$ is equivalent to
$\mathrm{Fix}_{(-)^{C_p}}(\bT\Sp_p^\gen)$, the $\infty$-category of fixed points for
the endofunctor $(-)^{C_p}$ of $\bT\Sp_p^\gen$. In particular, an equivalent
way of defining an object of $\TCart_p$ is to give a genuine $S^1$-spectrum
$M\in\bT\Sp_p^\gen$ together with an equivalence $M^{C_p}\we M$ of genuine
$S^1$-spectra.
\end{mainremark}

\paragraph{Outline.}In Section~\ref{sec:cyct}, we prove the existence of the $t$-structure of
Theorem~\ref{mt:t}. We prove some basic, but important, properties of the
cyclotomic $t$-structure. Section~\ref{sec:tcm} introduces $p$-typical topological
Cartier modules and establishes Theorem~\ref{mt:td}. In
Section~\ref{sec:monoidal}, we compare the natural symmetric monoidal structure
on $\CycSp_p^\heart$ with the monoidal structure previously constructed by
Goerss~\cite{goerss-hopf} on $W(k)$-modules in $p$-typical Cartier modules when
$k$ is a perfect field of characteristic $p$. We study some genuine
equivariant homotopy-theoretic aspects of the story in Section~\ref{sec:genuine}.
Finally, Section~\ref{sec:schemes} contains our
applications to $\THH$ of rings and schemes. Appendix~\ref{sec:tbackground} gives some
background on $t$-structures.

\paragraph{Conventions.} We will freely use the theory of $\infty$-categories
developed by Lurie in~\cite{htt,ha,sag}. Unless otherwise mentioned, we work with cyclotomic
spectra as studied in~\cite{nikolaus-scholze}. We will make one important
deviation from the notation in {\em op.\ cit}. Namely, a $p$-typical cyclotomic spectrum
will be a spectrum $X\in\Sp^{BS^1}$ with $S^1$-action equipped with an
$S^1$-equivariant map $\varphi_p\colon X\rightarrow
X^{tC_p}$, where $X^{tC_p}$ carries the residual $S^1/C_p\iso S^1$-action. See
Remark~\ref{rem:tc} for more about this choice.
We write $\CycSp_p$ for the stable $\infty$-category of $p$-typical cyclotomic spectra.

\paragraph{Notation.} Let $\Cscr$ be an $\infty$-category with objects $x,y\in\Cscr$. We
will write $\Map_\Cscr(x,y)$ for the $\infty$-groupoid (space) of maps from $x$ to $y$ in
$\Cscr$. If $\Cscr$ is stable, we will write $\MapSp_\Cscr(x,y)$ for the mapping
\emph{spectrum} from $x$ to $y$. Given an $\EE_\infty$-ring spectrum $R$, let
$\Dscr(R)\we\Mod_R(\Sp)$ denote the stable $\infty$-category of $R$-module spectra.
If $R$ is connective, then we equip $\Dscr(R)$ with the canonical Postnikov
$t$-structure, where $\Dscr(R)^\heart\we\Mod_{\pi_0R}$, the abelian category of
$\pi_0R$-modules. Unless specified otherwise, all limits, colimits, and tensor
products are computed in $\Dscr(R)$. This convention holds even for discrete rings; thus,
limits and colimits of abelian groups are computed in $\Dscr(\ZZ)$ as opposed
to $\Mod_\ZZ$: quotients are given by cofibers, limits by derived
limits, and the tensor product by the derived tensor product. For example, if
$M$ is an abelian group, then $M/p$ is the cofiber of $M\xrightarrow{p}M$.
Hence, $M/p$ is an object of $\Dscr(\ZZ)$ with $\pi_iM/p=0$ for $i\neq 0,1$,
$\pi_1M/p\iso\ker(M\xrightarrow{p}M)\iso\Tor_1^\ZZ(M,\ZZ/p)$,
and $\pi_0M/p\iso\coker(M\xrightarrow{p}M)\iso\Tor_0^\ZZ(M,\ZZ/p)$. Note that the objects of
$\Dscr(\ZZ)$ can be modeled by either chain complexes or $\ZZ$-module spectra.
We will typically write $\pi_*M$ for the homotopy groups of $M$ viewed as a
spectrum; these are isomorphic to the homology groups $\H_*M$ when $M$ is
viewed as a chain complex.

\paragraph{Acknowledgments.} We would like to thank Bhargav Bhatt, Alice Hedenlund, Lars
Hesselholt, Marc Hoyois, Achim Krause, Akhil Mathew, Irakli Patchkoria, Nick Rozenblyum, Peter Scholze, and Jay
Shah for their helpful insights about the ideas presented here. We would also like
to thank the Max Planck Institute for Mathematics, the Universit\"at
M\"{u}nster, and the Isaac Newton Institute for Mathematical Sciences for support and for hospitality during various visits.
Benjamin Antieau was supported by NSF Grant DMS-1552766. Finally, we would like
to thank the anonymous referees for their very helpful comments.

\section{The cyclotomic $t$-structure}\label{sec:cyct}

In this section we define the {\bf cyclotomic $t$-structure} for integral
and $p$-typical cyclotomic spectra in their genuine and non-genuine flavors.
With some difficulty, one can prove some basic facts about truncations in the
cyclotomic $t$-structure, including for example the fact that $\THH(\FF_p)$ is
in the heart. By working instead with topological
Cartier modules, introduced in Section~\ref{sec:tcm}, such computations are
more transparent. Hence, in this section, we restrict ourselves to discussing
formal properties of the cyclotomic $t$-structure.

For the necessary background on $t$-structures, see
Appendix~\ref{sec:tbackground}.

\subsection{The cyclotomic $t$-structure}\label{sub:cyct}

There is a $t$-structure on
$\Sp^{BS^1}$, the $\infty$-category of spectra with $S^1$-action, where the connective objects are
the {\em connective} spectra with $S^1$-action. This is reviewed in  Proposition~\ref{prop:xsp}. The heart is the abelian category of
abelian groups since $BS^1$ is simply connected.

We let $\CycSp_p$ be the lax equalizer
$$\mathrm{LEq}(\id,(-)^{tC_p}\colon\Sp^{BS^1}\rightrightarrows\Sp^{BS^1}).$$ (For background
on lax equalizers, see~\cite{nikolaus-scholze}*{Section~II.1}.)
An object of $\CycSp_p$ is a spectrum $X$ with $S^1$-action and an $S^1$-equivariant map
$\varphi\colon X\rightarrow X^{tC_p}$, where $X^{tC_p}$ carries the residual $S^1\iso S^1/C_p$-action.

Let $(\CycSp_p)_{\geq 0}\subseteq\CycSp_p$ denote the full subcategory of $p$-typical cyclotomic
spectra $X$ such that the underlying spectrum is connective.
In this case, the
cyclotomic structure map $\varphi:X\rightarrow X^{tC_p}$ factors canonically through the
connective cover $\tau_{\geq 0}(X^{tC_p})$.

Recall that a $t$-structure on a stable $\infty$-category $\Cscr$ is {\bf compatible with a
symmetric monoidal structure $\Cscr^\otimes$ on $\Cscr$} if $\Cscr_{\geq
0}\subseteq\Cscr$ is closed under tensor products and the unit object of
$\Cscr$ is in $\Cscr_{\geq 0}$. The purpose of this section is to prove
the following theorem.

\begin{theorem}\label{thm:cyclotomict}
    The $\infty$-category $(\CycSp_p)_{\geq 0}$ forms the connective part of an
    accessible, left complete $t$-structure on $\CycSp_p$, which
    is compatible with the symmetric monoidal structure on $\CycSp_p$.
\end{theorem}

\begin{example}
    If $X$ is a connective spectrum, then $X^\triv\in(\CycSp_p)_{\geq 0}$.
    Recall from~\cite[Example II.1.2 and Section IV.4]{nikolaus-scholze} that $X^\triv$ is the cyclotomic
    spectrum with trivial $S^1$-action and cyclotomic Frobenius given by the
    composition $X\rightarrow X^{hC_p}\rightarrow X^{tC_p}$. The cyclotomic
    sphere spectrum $\SS^\triv$ is the unit object for the natural symmetric
    monoidal structure on $\CycSp_p$
    (see~\cite{nikolaus-scholze}*{Section~IV.2}).
\end{example}

\begin{example}
    If $R$ is a connective $\EE_1$-ring spectrum, then
    $\THH(R)\in(\CycSp_p)_{\geq 0}$.
\end{example}

By construction, the forgetful functor $\CycSp_p\rightarrow\Sp^{BS^1}$ is right
$t$-exact. It is not left $t$-exact or even left bounded: as we will see
$\THH(\FF_p)$ is in the heart in $\CycSp_p$ but is not bounded above when
viewed as an object of $\Sp^{BS^1}$.

\begin{variant}\label{var:t}
    The same arguments that we give to prove Theorem \ref{thm:cyclotomict} will work for $\CycSp$, the $\infty$-category of
    global cyclotomic spectra, for $\CycSp_R=\Mod_R(\CycSp)$ where $R$ is a
    connective $\EE_\infty$-algebra in cyclotomic spectra, and for
    $\Mod_R(\CycSp_p)$ when $R$ is a connective $\EE_\infty$-algebra in
    $p$-typical cyclotomic spectra.
    In all of these cases the connective part of the $t$-structure consists of
    the full subcategory of cyclotomic (or $p$-typical cyclotomic) $R$-modules $X$ for
    which the underling spectrum with $S^1$-action is connective. 
\end{variant}

\begin{remark}\label{rem:tc}
    In this paper we shall concentrate on the $p$-typical aspects of the theory.
    The integral case will be pursued in future work. For this reason, we write
    $\TC(X)$ and $\TR(X)$ for $p$-typical versions of $\TC$ and $\TR$.
    Since our definition of $p$-typical cyclotomic spectra differs slightly
    from that of~\cite{nikolaus-scholze}, note that for us, $\TC(-)$ is the theory
    representable by $\SS^\triv$ in $\CycSp_p$; hence, $\TC(X)$ is the
    equalizer of $X^{hS^1}\rightrightarrows (X^{tC_p})^{hS^1}$, where the two
    maps are given by the canonical map $X^{hS^1}\rightarrow
    X^{tS^1}\rightarrow (X^{tC_p})^{hS^1}$ and by $\varphi^{hS^1}$.
    
    We let $\TR^{n+1}(X)$ be defined as the iterated pullback
    \[
    \TR^{n+1}(X) =  X^{hC_{p^n}} \times_{(X^{tC_p})^{hC_{p^{n-1}}}}  ...  \times_{(X^{tC_p})^{hC_p}} X^{hC_p} \times_{X^{tC_p}} X 
    \]
    where the maps to the left are induced by $\varphi$ and the maps to the right are the canonical maps. When $X$ is bounded below, the methods
    of~\cite{nikolaus-scholze}*{Chapter~II} endow $X$ with the structure of an $S^1$-spectrum which is
    genuine with respect to the subgroups $C_{p^n}$ for $n\geq 0$ and $\TR^{n+1}(X)\we
    X^{C_{p^n}}$. There are natural maps
    $R\colon\TR^{n+1}(X)\rightarrow\TR^n(X)$ given by forgetting the first factor in the above pullback and $\TR(X)\we\lim\TR^{n+1}(X)$. In general,
    there is a map $\TC(X)\rightarrow\TR(X)$ which induces a $p$-adic equivalence
    $\TC(X)\rightarrow\fib(\TR(X)\xrightarrow{1-F}\TR(X))$.
\end{remark}

\begin{warning}
    Note that $\TC(X,p)$ is typically defined as $\fib(\TR(X)\xrightarrow{1-F}\TR(X))$, so
    our definition agrees with $\TC(X,p)$ only after $p$-completion. Rather,
    $\fib(\TR(X)\xrightarrow{1-F}\TR(X))$ is equivalent to the equalizer of the canonical
    and Frobenius maps $X^{hC_{p^\infty}}\rightrightarrows (X^{tC_p})^{hC_{p^\infty}}$.
\end{warning}

\begin{definition}
    For an object $X\in\CycSp_p$, we will write $\pi_i^\cyc X\in\CycSp_p^\heart$ for
    the $i$th homotopy object of $X$. Note that this is a cyclotomic spectrum
    and hence has an underlying spectrum with $S^1$-action. The homotopy groups
    of this underlying spectrum are $\pi_*(\pi_i^\cyc X)$.
\end{definition}

\begin{example}
    If $X\in\CycSp_p$, a necessary condition for
    $X\in(\CycSp_p)_{\leq 0}$ is that
    $\pi_i\TC(X)=0$ for $i>0$ since $\SS^\triv\in(\CycSp_p)_{\geq 0}$ and
    $\TC(X)$ is the mapping spectrum from $\SS^\triv$ to $X$. Thus, $\ZZ^\triv$ is not in $\CycSp^\heart$, because
    $\pi_i\TC(\ZZ^\triv)\iso\ZZ_p$ for all positive odd $i$.
\end{example}

\begin{warning}
    This $t$-structure is not compatible with filtered colimits. In particular,
    the heart $\CycSp^\heart_p$ is not closed under filtered colimits in
    $(\CycSp_p)_{\geq 0}$. See Example~\ref{ex:filteredcolimits} for details.
\end{warning}

\begin{lemma}\label{lem:t}
    The following diagram
    \begin{equation*}
        \xymatrix{
        (\Sp^{BS^1})_{\geq 0}\ar@<.5ex>[rrr]^{\id}\ar@<-.5ex>[rrr]_{\tau_{\geq
        0}(-)^{tC_p}}\ar[d]&&&(\Sp^{BS^1})_{\geq 0}\ar[d]\\
        \Sp^{BS^1}\ar@<.5ex>[rrr]^{\id}\ar@<-.5ex>[rrr]_{(-)^{tC_p}}&&&\Sp^{BS^1}
        }
    \end{equation*}
    induces via the natural transformations of functors
    $\tau_{\geq 0}(-)^{tC_p}\rightarrow(-)^{tC_p}$ an equivalence $$\LEq\left(\id,\tau_{\geq
    0}(-)^{tC_p}\colon(\Sp^{BS^1})_{\geq 0}\rightrightarrows(\Sp^{BS^1})_{\geq 0}\right)\we(\CycSp_p)_{\geq 0}.$$
    In particular, $(\CycSp_p)_{\geq 0}$ is presentable and the inclusion functor
    $(\CycSp_p)_{\geq 0}\rightarrow\CycSp_p$ preserves colimits.
\end{lemma}

\begin{proof}
    The equivalence is clear given that for a connective cyclotomic spectrum
    $X$ the cyclotomic structure map determines and is determined by its
    factorization through the connective cover. We
    can appeal to~\cite{nikolaus-scholze}*{II.1.5(3)} which gives
    presentability (since $\tau_{\geq 0}(-)^{tC_p}$ is accessible). Now, both
    forgetful functors $(\CycSp_p)_{\geq 0}\rightarrow(\Sp^{BS^1})_{\geq 0}$
    and $\CycSp_p\rightarrow\Sp^{BS^1}$ preserve and detect colimits. Since
    $(\Sp^{BS^1})_{\geq 0}\rightarrow\Sp^{BS^1}$ is closed under colimits, the
    claim about preservation of colimits follows.
\end{proof}

\begin{proof}[Proof of Theorem~\ref{thm:cyclotomict}]
    By definition, $(\CycSp_p)_{\geq 0}$ is closed under extensions in
    $\CycSp_p$. Combined with Lemma~\ref{lem:t},
    it follows from~\cite{ha}*{1.4.4.11} that
    there exists a unique $t$-structure $((\CycSp_p)_{\geq
    0},(\CycSp_p)_{\leq 0})$ on $\CycSp_p$. The cyclotomic $t$-structure is accessible because $(\CycSp_p)_{\geq
    0}$ is presentable by Lemma~\ref{lem:t}. 

    It is clear that $\CycSp_p$ is left separated: an object in $$\bigcap_{n\in\ZZ}(\CycSp_p)_{\geq
    n}$$ has contractible underlying $S^1$-spectrum. We check that it is left
    complete by showing that $(\CycSp_p)_{\geq 0}$ is closed under countable
    products in $\CycSp_p$ and applying~\cite{ha}*{1.2.1.19}, which says that
    under this hypothesis left completeness is equivalent to left
    separatedness. In fact, we will show that $(\CycSp_p)_{\geq 0}$ is closed under
    all products in $\CycSp_p$.

    We already know that $(\Sp^{BS^1})_{\geq 0}$ is closed under
    products in $\Sp^{BS^1}$ by Proposition~\ref{prop:xsp}.
    Thus, to conclude, it is enough to show that the functor $\CycSp\rightarrow\Sp^{BS^1}$
    commutes with products of connective objects.
    By~\cite{nikolaus-scholze}*{II.1.5(v)}, it is enough to see that for
    a product $\prod_{i\in I}X_i$ of connective spectra with
    $S^1$-action the natural map $$\left(\prod_i X_i\right)^{tC_p}\rightarrow\prod_i
    X_i^{tC_p}$$ is an equivalence. This follows from
    Lemma~\ref{lem:littlelimits} below.
    
    To see that the cyclotomic $t$-structure is compatible with the symmetric
    monoidal structure on $\CycSp_p$ we simply note that
    $\SS^\triv\in(\CycSp_p)_{\geq 0}$ and that $(\CycSp_p)_{\geq 0}$ is
    closed under the tensor product, which follows from the fact that the
    tensor product of two connective spectra with $S^1$-action is again
    connective. This completes the proof.
\end{proof}

We used the following lemma in the proof.

\begin{lemma}\label{lem:littlelimits}
    Suppose that $F\colon I\rightarrow\Sp^{BS^1}$ is an $I$-diagram in spectra
    with $S^1$-action with limit $X=\lim_iF(i)$.
    \begin{enumerate}
        \item[{\rm (a)}] If $I=\ZZ^{\op}$, so that $X$ is the limit of the
            tower $\cdots\rightarrow F(i+1)\rightarrow F(i)\rightarrow\cdots$, and if the
            fiber of $X\rightarrow F(i)$ is $n_i$-connective where
            $n_i\rightarrow\infty$ as $i\rightarrow\infty$, then $X_{hC_p}\we\lim_i\left( F(i)_{hC_p}\right)$ and
            $X^{tC_p}\we\lim_i\left(F(i)^{tC_p}\right)$.
        \item[{\rm (b)}] If  there exists $d$ such that
            $\lim_I\colon\Sp^I\rightarrow\Sp$ sends $\Sp^I_{\geq 0}$ to
            $\Sp_{\geq -d}$ and there exists $N$ such that
            each $F(i)$ is $N$-connective, then $X_{hC_p}\we\lim_i\left(F(i)_{hC_p}\right)$ and
            $X^{tC_p}\we\lim_i\left(F(i)^{tC_p}\right)$.
    \end{enumerate}
\end{lemma}

\begin{example}
    If $I=\ZZ^{\op}$ so that $X$ is a sequential limit, then (b) applies since
    $d=1$. If $I$ is discrete, so that $X$ is a product, then (b) applies since
    $d=0$.
\end{example}

\begin{proof}[Proof of Lemma~\ref{lem:littlelimits}]
    Each statement for the Tate construction follows from the corresponding statement for
    homotopy orbits using the fiber sequence $X_{hC_p}\rightarrow
    X^{hC_p}\rightarrow X^{tC_p}$ and the fact that $X^{hC_p}$ commutes with
    all limits of spectra with $C_p$-action.

    The proof of (a) is the same as the proof
    of~\cite{nikolaus-scholze}*{Lemma~I.2.6}. It is easy enough to repeat here:
    the fiber of $X_{hC_p}\rightarrow F(i)_{hC_p}$ is $n_i$-connective since
    taking homotopy orbits is right $t$-exact. Thus, the limit of the fibers
    vanishes.

    For (b), consider for each $n$ the map $\tau_{\leq
    n}X_{hC_p}\rightarrow\lim_I\tau_{\leq n}F(i)_{hC_p}$. The functor
    $\tau_{\leq n}(-)_{hC_p}$ is computed as a colimit over the skeleton of
    $BC_p$ on uniformly bounded below objects, so it commutes
    with $I$-limits of uniformly bounded below objects. Hence, $\lim_I\tau_{\leq
    n}F(i)_{hC_p}\we(\tau_{\leq n}(\lim_IF(i))_{hC_p})\we\lim\tau_{\leq
    n}X_{hC_p}$. Taking the limit over $n$ on both sides, we obtain the desired
    equivalence.
\end{proof}

\begin{remark}
    The heart $\CycSp_p^\heart$ is by
    definition $(\CycSp_p)_{\geq 0}\cap(\CycSp_p)_{\leq 0}$ and it is an abelian category
    (by~\cite{bbd}*{Th\'eor\`eme~1.3.6}). In fact,
    $\CycSp_p^\heart\subseteq (\CycSp_p)_{\geq 0}$ is the full subcategory of
    $0$-truncated objects (see~\cite{ha}*{1.2.1.9}), so it is
    presentable because $(\CycSp_p)_{\geq 0}$ is presentable. Truncation
    $\tau_{\leq 0}\colon\Cscr_{\geq
    0}\rightarrow\Cscr^\heart$ gives a left adjoint to the
    inclusion. See~\cite{htt}*{5.5.6.21} for details.
\end{remark}

Because $(\CycSp_p)_{\geq 0}$ is closed under the tensor product in
$\CycSp_p$, and since the unit $\SS^\triv$ of $\CycSp_p$ is connective,
$(\CycSp_p)_{\geq 0}$ naturally inherits a symmetric monoidal structure such that the inclusion
$(\CycSp_p)_{\geq 0}\rightarrow\CycSp_p$ is symmetric monoidal.
It is easy to see that $\CycSp_p^\heart$ inherits a symmetric monoidal structure
from $(\CycSp_p)_{\geq 0}$. For details, see~\ref{sub:compatible}.

\begin{corollary}\label{cor_symmheart}
    The abelian category $\CycSp_p^\heart$ inherits a symmetric monoidal
    structure $\otimes^\heart$ from $(\CycSp_p)_{\geq 0}$ such that the localization functor
    $\pi_0^\cyc:(\CycSp_p)_{\geq 0}\rightarrow\CycSp_p^\heart$ is symmetric monoidal.
    Moreover, $\otimes^\heart$ is compatible with colimits in each variable.
\end{corollary}

\begin{proof}
    The first claim follows from Lemma~\ref{lem:monoidalloc}. To prove the
    second, we use that colimits in $\CycSp_p^\heart$ are obtained by computing
    colimits in $(\CycSp_p)_{\geq 0}$ and applying $\pi_0^\cyc$. Since the tensor
    product of connective cyclotomic spectra commutes with colimits in each
    variable and since $\pi_0^\cyc\colon(\CycSp_p)_{\geq
    0}\rightarrow(\CycSp_p)^\heart$ commutes with colimits (being a left adjoint),
    the claim follows.
\end{proof}

\begin{remark}
    Concretely, if $A,B\in\CycSp_p^\heart$, then $A\otimes^\heart
    B\we\pi_0^\cyc(A\otimes B)$, the $0$-truncation of the tensor product of
    $A$ and $B$ when viewed as objects of $(\CycSp_p)_{\geq 0}$.
\end{remark}

\begin{example}
    The fiber $F$ of $\SS^\triv\rightarrow\ZZ^\triv$ in cyclotomic spectra is
    connected. Hence, $\pi_0^\cyc F\we 0$ and it follows that
    $\pi_0^\cyc\SS^\triv\we\pi_0^\cyc\ZZ^\triv$. In particular, every object of
    $\CycSp_p^\heart$ is canonically a $\ZZ^\triv$-module in $\CycSp_p$.
\end{example}

\subsection{The genuine cyclotomic $t$-structure}\label{sub:genuinet}

\newcommand{\eff}{\mathrm{eff}}

\newcommand{\tsets}{\Oscr_{\Fscr}^{\sqcup}(S^1)}
\newcommand{\tpsets}{\Oscr_{\Fscr_p}^{\sqcup}(S^1)}

In this section we introduce the $t$-structure on genuine cyclotomic spectra
after introducing the Mackey $t$-structure on genuine $S^1$-spectra. This will
only become relevant for Section~\ref{sec:genuine} and the reader who does not want to get
involved with the intricacies of genuine homotopy theory can safely  skip this
section.

We will use genuine equivariant homotopy theory as a black box, but remind the
reader that there is a stable presentable $\infty$-category 
\[
\bT\Sp_\Fscr^\gen
\]
of genuine $S^1$-spectra with respect to a family $\Fscr$ of subgroups of
$S^1$. Here we write $\bT$ for $S^1$ to distinguish it from the homotopical
circle. A genuine spectrum $X \in \bT\Sp_\Fscr^\gen$
has fixed points $X^H$ for all closed subgroup $H \subseteq \bT$ that lie in
$\Fscr$. If $\Fscr$ is the family just consisting of the trivial group $1 \subseteq
\bT$ then $\bT\Sp_\Fscr^\gen \we \Sp^{BS^1}$. We will mostly be concerned with
the family consisting of finite $p$-subgroups of $\TT$. In this case we write
$\bT\Sp_p^\gen$ for this $\infty$-category:
\[
\bT\Sp_p^\gen := \bT\Sp^\gen_{\{1,C_p,C_{p^2},...\}}.
\]
There are several equivalent ways of describing $\bT\Sp_p$. A treatment using
equivariant orthogonal spectra is reviewed in~\cite{nikolaus-scholze} based on
lecture notes by Schwede~\cite{schwede-lectures}. An elegant $\infty$-categorical model
using spectral Mackey functors is due to Barwick~\cite{barwick-spectral} and
Barwick--Glasman~\cite{barwick-glasman-cyclonic} based on a model of Guillou--May \cite{MayGuillou}.  We
assume the reader is familiar with basic constructions such as fixed points,
geometric fixed points, classifying spaces for families, the tom Dieck
splitting  and the isotropy separation sequence.
For the reminder of the section we shall work with genuine spectra $X \in 
\bT\Sp_p^\gen$ and for simplicity just refer to them as genuine $S^1$-spectra or
even genuine spectra.

Let $X$ be a genuine $S^1$-spectrum. We say that $X$ is connective if for each
$n \geq 0$ the fixed points spectrum $X^{C_{p^n}}$ is connective. Similar we
say that $X$ is coconnective if for each $n \geq 0$ the spectrum $X^{C_{p^n}}$
is coconnective. This defines two full subcategories
\[
(\bT\Sp^\gen_p)_{\geq 0} \qquad \text{and} \qquad (\bT\Sp^\gen_p)_{\leq 0}
\]
of $\bT\Sp_p^\gen$.

\begin{proposition}\label{prop:mackeyt}
    The pair $((\bT\Sp^\gen_p)_{\geq 0},(\bT\Sp^\gen_p)_{\leq 0})$
    defines a left and right complete, accessible $t$-structure on $\bT\Sp^\gen_p$ which is compatible with filtered
    colimits.
\end{proposition}

Note that by construction the fixed point functor
\[
(-)^{C_{p^n}} : \bT\Sp^\gen_p \to \Sp
\]
is $t$-exact for each $n\geq 0$.

\begin{definition}
    We call this $t$-structure on $\bT\Sp^\gen_p$ the {\bf Mackey
    $t$-structure}.
\end{definition}

The terminology is motivated by the fact that the heart
$(\bT\Sp^\gen_p)^\heart$ is equivalent to the abelian category of Mackey functors. The notion of a Mackey
functor might not be entirely standard in this setting (since $\bT$ is not a
finite group) but Mackey $t$-structures exist for categories of genuine
$G$-spectra where $G$ is finite, in which case the heart is equivalent to the
abelian category of classical Mackey functors on $G$. The construction and
proof is the same as in Proposition~\ref{prop:mackeyt}.

\begin{remark}
    One equivalent description of $\bT\Sp^\gen_p$ following Barwick~\cite{barwick-spectral} is as the $\infty$-category of
    product-preserving functors from the effective Burnside category of the orbit category
    $\{S^1/C_{p^n}\}_{n\geq 0}\subseteq\Sscr^{BS^1}$ to spectra.
    In this language, $(\bT\Sp^\gen_p)_{\geq 0}$ is equivalent to the $\infty$-category of
    product-preserving functors from the Burnside category to $\Sp_{\geq 0}$ and similarly for
    the $\infty$-category of coconnective objects. From this description, the existence of
    the $t$-structure in Proposition~\ref{prop:mackeyt} is clear. We give, however, a
    presentation-independent proof.
\end{remark}

\begin{proof}[Proof of Proposition~\ref{prop:mackeyt}]
Consider the compact generators $\Sigma^\infty_+ \bT/C_{p^n} \in
\bT\Sp^\gen_p$. By \cite[Proposition 1.4.4.11]{ha} the smallest subcategory
$\Cscr_{\geq 0} \subseteq \bT\Sp^\gen_p$ that contains $\Sigma^\infty_+
\bT/C_{p^n} $ and is closed under colimits and extensions is the
$\infty$-category of connective
objects for a $t$-structure $(\Cscr_{\geq 0},\Cscr_{\leq 0})$ on $\bT\Sp^\gen_p$. An object $X \in \bT\Sp^\gen_p$ is then
$(-1)$-truncated (i.e. in $\Cscr_{\leq -1}$) if and only if the mapping spectrum
\[
    \MapSp_{\bT\Sp_p^\gen}(\Sigma^\infty_+ \bT/C_{p^n},X) \simeq X^{C_{p^n}}
\]
is  $(-1)$-truncated. We claim that an object $X$ is connective precisely if
for every $n$ the spectrum $X^{C_{p^n}}$ is connective. If $X$ is
connective, then it follows that $X^{C_{p^n}}$ is connective since this is true
for the generators. Conversely, assume that $X^{C_{p^n}}$ is connective for each $n$.
We consider the truncation $X \to \tau_{\leq -1} X$. The fibre of this map is
given by $\tau_{\geq 0} X \to X$. We get a cofiber sequence
\[
(\tau_{\geq 0} X)^{C_{p^n}} \to X^{C_{p^n}} \to  (\tau_{\leq -1} X)^{C_{p^n}}
\]
for each $n\geq 0$.
The first two terms are connective spectra, thus so is the third. But since
$\tau_{\leq -1} X$ is $(-1)$-truncated the last term is also $(-1)$-truncated as it is a mapping spectrum from a connective object.
Therefore it has to be zero. Since the orbits are generators of
$\bT\Sp_p^\gen$, it follows that $\tau_{\leq -1} X \we 0$. Thus, $X$ is connective.

This establishes the existence of $t$-structure. The other claims now immediately follow
using the conservativity of family $\{(-)^{C_{p^n}}\colon n\geq 0\}$ of
functors $\bT\Sp_p^\gen\rightarrow\Sp$ and the fact that fixed points preserve
limits and colimits.\end{proof}

\begin{lemma}\label{conn}
A genuine  $S^1$-spectrum $X \in \bT\Sp^\gen_p$ is connective in the Mackey $t$-structure precisely if all geometric fixed points
$X^{\Phi C_{p^n}} \in \Sp$ are connective. 
\end{lemma}
\begin{proof}    
 To see this we argue by induction over $n$. Assume that all spectra
   $ X^{C_k}$ and  $ X^{\Phi C_{p^k}}$  for $k = 0, \ldots, n-1$ are connective. Consider the isotropy separation sequence
  \[
  ((E_{\mathcal{P}rop} C_{p^n})_+ \otimes X)^{C_{p^n}} \to X^{C_{p^n}} \to X^{\Phi C_{p^n}}
  \]
  where the left hand side is a colimit of a diagram only involving fixed
  points of proper subgroups of $C_{p^n}$.\footnote{In fact this can even by
  simplified in this case since the subgroup lattice of $C_{p^n}$ is very
  simple but we prefer to write the proof in a form that also works for
  more complicated families.} Thus it is a connective spectrum. It follows that $
  X^{C_{p^n}} $ is connective if and only if $X^{\Phi C_{p^n}}$ is connective. 
\end{proof}

We recall the definition of genuine cyclotomic spectra. 

\begin{definition}\label{def:gencyc}
    The $\infty$-category $\CycSp^\gen_p$ of {\bf genuine
    $p$-typical cyclotomic spectra} is defined to be fixed points for the
    endofunctor $(-)^{\Phi C_p}$:
    \[
        \CycSp^\gen_p := \mathrm{Fix}_{(-)^{\Phi C_p}}(\bT\Sp^\gen_p).
    \]
    In other words, a genuine $p$-typical cyclotomic spectrum is a genuine
    $S^1$-spectrum $X$ equipped with an equivalence $X^{\Phi C_p}\we X$ of genuine
    $S^1$-spectra.
\end{definition}

By Lemma~\ref{conn}, the functor $(-)^{\Phi C_p}$ is right $t$-exact. Using that accessible
$t$-structures are closed under limits of right $t$-exact left adjoint functors,
we obtain the following corollary.

\begin{corollary}\label{cor:gencyct}
    There is an accessible $t$-structure on
    $\CycSp^\gen_p$ where $$(\CycSp_p^\gen)_{\geq 0}\we\mathrm{Fix}_{(-)^{\Phi
    C_p}}((\bT\Sp_p^\gen)_{\geq 0});$$ the forgetful functor to $\bT\Sp_p^\gen$ is right $t$-exact.
\end{corollary}

We call the induced $t$-structure on $p$-typical genuine cyclotomic spectra the
{\bf genuine cyclotomic $t$-structure}.

\begin{remark}
With the same arguments one also gets a Mackey $t$-structure on
$\bT\Sp^\gen_\Fscr$ where $\Fscr$ is the family of finite subgroups. This then
induces also a $t$-structure on the $\infty$-category of global genuine
cyclotomic spectra $\CycSp^\gen$. The latter is defined as the
$\infty$-category of homotopy fixed
points
\[
\CycSp^\gen := (\bT\Sp^\gen_\Fscr)^{h\mathbb{N}_{>0}},
\]
where the multiplicative monoid $\mathbb{N}_{>0}$ acts on $\bT\Sp^\gen_\Fscr$
via $n \mapsto (-)^{\Phi C_n}$. See~\cite{nikolaus-scholze}*{Section~2.3}.
\end{remark}

The next result says that the $t$-structure of Corollary~\ref{cor:gencyct} 
reduces to the cyclotomic $t$-structure of Section~\ref{sub:cyct} when restricted to
bounded below objects.

\begin{theorem}
    The natural functors $\CycSp^\gen\rightarrow\CycSp$ and
    $\CycSp^\gen_p\rightarrow\CycSp_p$ are right $t$-exact and restrict to equivalences
    $\CycSp^{\gen}_{\geq 0}\we\CycSp_{\geq 0}$ and $(\CycSp^{\gen}_p)_{\geq
        0}\we(\CycSp_{p})_{\geq 0}$.
\end{theorem}

\begin{proof}
    This is simply a restatement of
    Theorem~\cite{nikolaus-scholze}*{Theorem~II.3.8}, once we note that an
    object $X$ in $\CycSp^\gen$ or $\CycSp^\gen_p$ is bounded below in the
    cyclotomic $t$-structure if and only if the underlying spectrum of $X$ is
    bounded below. This follows from Lemma \ref{conn} and the fact that
    $\Phi^{C_{p^n}}X\we X$ for all $n$.
\end{proof}

In particular, we see that $\CycSp_p^{\gen,\heart}\we\CycSp_p^\heart$ and similarly
$\CycSp^{\gen,\heart}\we\CycSp^\heart$.

\subsection{The cyclotomic $t$-structure and the generalized Segal conjecture}\label{sub:segal}

Recall that the Segal conjecture for $C_p$ (a theorem of Lin~\cite{lin} and
Gunawardena~\cite{gunawardena})  says that the Tate diagonal
$\SS\rightarrow(\SS^{\otimes p})^{tC_p}\we\SS^{tC_p}$ is $p$-completion. This map is also
equivalent to the trivial map that factors through homotopy fixed points $\SS^{hC_p}$.  We
prove that every bounded
above object in the cyclotomic $t$-structure satisfies an analogue of the Segal
conjecture.

\begin{proposition}\label{prop:segalcondition}
    Suppose that $M\in\CycSp_p$. If $M\in\left(\CycSp_p\right)_{\leq d}$, then the
    cyclotomic Frobenius map
    $M\xrightarrow{\varphi} M^{tC_p}$ is
    $d$-truncated.\footnote{Recall that a map of spectra $X\rightarrow Y$
    is $d$-truncated if it induces isomorphisms $\pi_nX\iso\pi_nY$ for
    $n\geq d+2$ and an injection $\pi_{d+1}X\rightarrow\pi_{d+1}Y$. In other
    words, the fiber of $X\rightarrow Y$ is $d$-truncated.}
\end{proposition}

\begin{proof}
    Fix a point $x\in BS^1$. We endow $x_!\SS[d+1]$ with the structure of a
    $p$-typical cyclotomic spectrum by letting $x_!\SS[d+1]\rightarrow(x_!\SS[d+1])^{tC_p}$ be the
    zero map. The mapping spectrum $\MapSp_{\CycSp_p}(x_!\SS[d+1],M)$ is the
    equalizer of
    $$\MapSp_{S^1}(x_!\SS[d+1],M)\rightrightarrows\MapSp_{S^1}(x_!\SS[d+1],M^{tC_p}),$$ where the two maps are given by
    $$(x_!\SS[d+1]\xrightarrow{f}M)\mapsto
    (x_!\SS[d+1]\xrightarrow{f}M\xrightarrow{\varphi}M^{tC_p})$$ and $$f\mapsto
    x_!\SS[d+1]\xrightarrow{0}(x_!\SS[d+1])^{tC_p}\xrightarrow{f^{tC_p}}M^{tC_p}.$$
    Since
    $\MapSp_{S^1}(x_!\SS[d+1],M)\we M[-d-1]$ and
    $\MapSp_{S^1}(x_!\SS[d+1],M^{tC_p})\we M^{tC_p}[-d-1]$, we find that
    $\MapSp_{\CycSp_p}(x_!\SS[d+1],M)$ is the fiber of
    $M[-d-1]\xrightarrow{\varphi[-d-1]}M^{tC_p}[-d-1]$. 
    Since $M\in\left(\CycSp_p\right)_{\leq d}$, it follows that
    $$\pi_i\MapSp_{\CycSp_p}(x_!\SS[d+1],M)=0$$ for $i\geq 0$.
    Hence, $\pi_iM\xrightarrow{\pi_i(\varphi)}\pi_iM^{tC_p}$ is an injection
    for $i=d+1$ and an isomorphism for $i\geq d+2$. In other words, $\varphi$
    is $d$-truncated.
\end{proof}

\begin{remark}
    The proposition gives a necessary but certainly not sufficient condition for an
    object to be bounded above. For example, $\SS_p^\triv$ (meaning the $p$-complete sphere with the trivial cyclotomic structure) satisfies the conclusion
    of Proposition \ref{prop:segalcondition} by the Segal conjecture but not the hypothesis since $\TC(\SS_p^\triv)$ contains $\SS_p$ as a summand.
\end{remark}

We now give an example of a class of bounded above objects in the cyclotomic
$t$-structure. Many more will appear later, in Section~\ref{sec:schemes}, once we
have access to the $p$-typical topological Cartier module machinery developed in
Section~\ref{sec:tcm}.

\begin{example}
    Let $M\in\Sp_{\leq d}$ be a bounded above spectrum and let $x_!M$ denote the induced
    $S^1$-spectrum $M\otimes S^1_+$. Now, $(x_!M)^{tC_p}\we 0$. To see this, we
    can reduce to the case where $M$ is concentrated in a single degree
    using the easy generalization of~\cite{nikolaus-scholze}*{I.2.6(ii)}
    for weak Postnikov towers and the fact that $x_!$ preserves colimits.
    We make $x_!M$ into a $p$-typical cyclotomic
    spectrum in the only way we can: we let the cyclotomic Frobenius
    $x_!M\rightarrow (x_!M)^{tC_p}\we 0$ be the zero map.
    Now, we show that $x_!M\in(\CycSp_p)_{\leq d+1}$. To see this, fix another
    cyclotomic spectrum $X$. We see that the mapping spectrum
    $\MapSp_{\CycSp_p}(X,x_!M)\we\MapSp_{S^1}(X,x_!M)$ using the equalizer
    formula, since $(x_!M)^{tC_p}\we 0$. Of course, $x_!M\in(\Sp^{BS^1})_{\leq
    d+1}$. Hence, if $X\in(\CycSp_p)_{\geq d+2}$, the mapping space
    $\Map_{\CycSp_p}(X,x_!M)$ vanishes.
\end{example}

\section{Topological Cartier modules}\label{sec:tcm}

In this section we give a new description of $p$-typical cyclotomic spectra based on a
topological version of $p$-typical Cartier modules. Using this description, we find
that the cyclotomic homotopy groups $\pi_i^\cyc X$ of a $p$-typical cyclotomic
spectrum $X$ ``are'' the homotopy groups of
$\TR(X)=\lim_{n,R}X^{C_{p^n}}$ equipped with operations $V$ and $F$ such that
$FV=p$.

As for cyclotomic spectra, $p$-typical topological Cartier modules admit two flavors: a
genuine flavor and a simplistic flavor. Unlike cyclotomic
spectra, where one needs a boundedness assumption to show that the simple
version agrees with the genuine version, these flavors are unconditionally the
same for $p$-typical topological Cartier modules, as we prove in Section~\ref{sec:genuine}.

\subsection{The $\infty$-category of topological Cartier modules}\label{sub:tcm}

Classically, a {\bf $p$-typical Cartier module} is an abelian group $M$ equipped with endomorphisms $V$
and $F$ such that $FV=p$.\footnote{In contrast to the case of Dieudonn\'e modules we do not include the condition that
    $VF=p$. This is only appropriate  when working over Witt vectors of a
    perfect ring of characteristic $p$, see Example \ref{dieudonne} and
Section~\ref{sub:charp}.}

\begin{definition}
    A {\bf $p$-typical topological Cartier module}
    is a spectrum $M$ with an $S^1$-action together with an $S^1$-equivariant factorization of the $C_p$-norm
    \[
    M_{hC_p} \xto{V} M \xto{F} M^{hC_p},
    \]
    where $M_{hC_p}$ and $M^{hC_p}$ carry the residual $S^1 = S^1/C_p$ actions
    and the norm $M_{hC_p} \to M^{hC_p}$ is equivariant for this action. We
    will refer to $V$ as the {\bf Verschiebung} and $F$ as the {\bf Frobenius}.
\end{definition}

Here are some examples and constructions with $p$-typical topological Cartier modules.

\begin{example}\label{example_classical}
    Let $M$ be an abelian group considered as an Eilenberg--MacLane spectrum with (necessarily) trivial $S^1$-action.
    A $p$-typical topological Cartier module structure on $M$ is equivalent to
    a $p$-typical Cartier module structure on $M$ since the maps $V$ and $F$ necessarily
    have to factor through the truncations $M_{hC_p} \to \tau_{\leq
        0}M_{hC_p}\we M$ and $M\we\tau_{\geq 0}M^{hC_p} \to M^{hC_p}$ and the norm factors as $M_{hC_p} \to M \xto{\cdot p} M \to
    M^{hC_p}$.
    \end{example}

\begin{example}\label{exhomotopygroup}
    For every $p$-typical topological Cartier module $M$ the homotopy groups
    $\pi_nM$ are $p$-typical Cartier modules: the maps $V,F: \pi_nM \to \pi_n M$ are induced by $\pi_n$ of the compositions 
    \[
        M \to M_{hC_p} \xto{V} M \qquad \text{and} \qquad M \xto{F} M^{hC_p} \to M.
    \]
    To see that the composition $FV$ is $p$, note that the composition is
    equivalent to $M\to M_{hC_p}\xto{\Nm_{C_p}} M^{hC_p}\to M$ and we can
    reduce to the case where $M$ is discrete, where the claim follows from the 
    description of the norm from group homology to group cohomology.
    For more about the structure of $\pi_*M$ when $M$ is a $p$-typical
    topological Cartier module, see Section~\ref{sub:cc}.
\end{example}

\begin{example}\label{TrDieu}
    Let $X$ be a $p$-typical cyclotomic spectrum. The $S^1$-spectrum
    \begin{align*}
    \TR(X) & \we \left( \ldots \times_{(X^{tC_p})^{hC_{p^2}}}  X^{hC_{p^2}} \times_{(X^{tC_p})^{hC_p}} X^{hC_p} \times_{X^{tC_p}} X \right) \\
    & \we \underleftarrow{\lim}_{n,R}\left(  X^{hC_{p^n}} \times_{(X^{tC_p})^{hC_{p^{n-1}}}}  ...  \times_{(X^{tC_p})^{hC_p}} X^{hC_p} \times_{X^{tC_p}} X \right) \\ 
    & \we \underleftarrow{\lim}_{n,R} (\TR^{n+1}(X))
    \end{align*}
    canonically carries the structure of a topological Cartier module.
    The map $V\colon \TR(X)_{hC_p} \to \TR(X)$ is given by the canonical map 
    \begin{equation}\label{mapuniform}
      \TR(X)_{hC_p} \we \left(\underleftarrow{\lim}_{n,R} \TR^n(X)\right)_{hC_p} \to \underleftarrow{\lim}_{n,R}(\TR^n(X)_{hC_p})
    \end{equation}
    followed by the levelwise norm
    \begin{equation}\label{def_V}
    \xymatrix{
    \ldots (X^{hC_{p}})_{hC_p} \times_{(X^{tC_p})_{hC_p}} X_{hC_p} \times_{0\phantom{dd}} 0 \phantom{dd}\ar[d]_V \\
    \ldots \phantom{dd}X^{hC_{p^2}}\phantom{dd} \times_{(X^{tC_p})^{hC_p}} X^{hC_p} \times_{X^{tC_p}} X.
    }
    \end{equation}
    The map $F\colon \TR(X) \to \TR(X)^{hC_p}$ is given by the projection
    \begin{equation}\label{def_F}
    \xymatrix{
    \ldots \times_{(X^{tC_p})^{hC_{p^2}}} X^{hC_{p^2}} \times_{(X^{tC_p})^{hC_p}} X^{hC_p} \times_{X^{tC_p}} X \ar[d]_F \\
    \ldots \times_{(X^{tC_p})^{hC_{p^2}}} X^{hC_{p^2}} \times_{(X^{tC_p})^{hC_p}} X^{hC_p}. \phantom{\times_{X^{tC_p}} X}
    }
    \end{equation}
    The composition is evidently the norm; this will also follow from a more
    detailed analysis in Construction~\ref{const:tr}.
\end{example}

\begin{example}\label{ex:tcm2cyc}
    Let $M$ be a $p$-typical topological Cartier module. We consider the $S^1$-spectrum 
    \[X = M/V = \mathrm{cofib}\left(M_{hC_p} \xto{V} M\right)
    \]
    together with the `quotient' map $\rho:M\rightarrow X$, which is also
    $S^1$-equivariant.  The $S^1$-spectrum $X$ admits a canonical $S^1$-equivariant map $X \to M^{tC_p}$ induced from the commutative square
    \[
    \xymatrix{
    M_{hC_p} \ar[d]^V \ar[r]^\id & M_{hC_p} \ar[d]^{\Nm_{C_p}} \\
    M \ar[r]^F & M^{hC_p}  
    }
    \]
    by taking vertical cofibers. The composition $X \to M^{tC_p}
    \xto{\rho^{tC_p}} X^{tC_p}$ endows $X$ with the structure of a cyclotomic
    spectrum. Note that the map $M^{tC_p}\to X^{tC_p}$ is an equivalence if $M$ is
    bounded below, since the fiber is given by $(M_{hC_p})^{tC_p}$ which vanishes by the Tate orbit lemma~\cite[Lemma I.2.1]{nikolaus-scholze}. 
    When $M$ is a $p$-typical Cartier module, viewed as a $p$-typical
    topological Cartier module via Example~\ref{example_classical}, the homotopy groups of
    $M/V$ are given in Figure~\ref{fig:homotopygroups} on page~\pageref{fig:homotopygroups}.
\end{example}

Now, we give a rigorous construction of the $\infty$-category of $p$-typical
topological Cartier modules.

\begin{definition}\label{defTCart}
    The $\infty$-category $\TCart_p$ of $p$-typical topological Cartier modules is the pullback
    \begin{equation}\label{eq:tcart}
        \xymatrix{
        \TCart_p\ar[rr]\ar[d]  &&
        \left(\Sp^{BS^1}\right)^{\Delta^2}\ar[d]^{(\ev_1, \partial^1)}\\
        \Sp^{BS^1}\ar[rr]^-{(\id,\Nm_{C_p})}&&\Sp^{BS^1} \times \left(\Sp^{BS^1}\right)^{\Delta^1},
        }
    \end{equation}
    where the bottom arrow sends an object $M$ to the pair
    $(M,M_{hC_p}\xrightarrow{\Nm_{C_p}}M^{hC_p})$ and the right-hand vertical arrow sends a
    $2$-simplex
    $$\xymatrix{
        &M_1\ar[dr]^f&\\
        M_0\ar[rr]_n\ar[ur]^v&&M_2,}$$
    expressing that $n\we f\circ v$, to the pair $(M_1,M_0\xrightarrow{n}M_2)$.
    We will write $M=(M,V_M,F_M,\sigma_M)$ for an object of $\TCart_p$, where $M$ is a spectrum
    with $S^1$-action, $V_M:M_{hC_p}\rightarrow M$, $F_M:M\rightarrow
    M^{hC_p}$, and $\sigma_M$ is a $2$-simplex expressing an equivalence
    $\Nm_{C_p}\we F_M\circ V_M$. 
\end{definition}

For the rest of the section we will establish some facts about the $\infty$-category 
$\TCart_p$ including the fact that it is presentable and a formula for the mapping spaces.
Since this is a bit technical the reader might want to skip the rest of this section on a
first reading.

We recall the $\infty$-category $\CycSp^\Fr_p$ of cyclotomic spectra with Frobenius lifts. The objects are spectra $X$ with $S^1$-action equipped
with an $S^1$-equivariant map $\psi_p\colon X\rightarrow X^{hC_p}$. As an $\infty$-category
$\CycSp^\Fr_p$  is defined as the pullback
\[
\xymatrix{
\CycSp^\Fr_p \ar[rr]\ar[d] && \left(\Sp^{BS^1}\right)^{\Delta^1} \ar[d]^{(\ev_0, \ev_1)} \\
\Sp^{BS^1} \ar[rr]^-{(\id, (-)^{hC_p})} && \Sp^{BS^1} \times \Sp^{BS^1},
}
\]
or as $\LEq(\id, (-)^{hC_p})$ in the language of \cite{nikolaus-scholze}. 

\begin{lemma}\label{pullbackTCart}
    There is a pullback square of stable $\infty$-categories
    \[
    \xymatrix{
    \TCart_p \ar[r] \ar[d] & \Sp^{BS^1} \times \Sp^{BS^1}\ar[d] \\
    \CycSp^\Fr_p \ar[r] & \left(\Sp^{BS^1} \right)^{\Delta^1}
    }
    \]
    where
    \begin{enumerate}
    \item[{\rm (1)}]
    the left vertical functor sends $(M, V, F,\sigma)$ to $(M, F)$;
    \item[{\rm (2)}] the upper horizontal map sends $(M,V,F, \sigma)$ to the pair
    $(M_{hC_p},M^{hC_p}/F)$, where $M^{hC_p}/F$ denotes the cofiber of
    $M\xrightarrow{F}M^{hC_p}$;
    \item[{\rm (3)}] the lower horizontal functor sends $(M,F)$ to the composition $M_{hC_p}
    \xto{\Nm_{C_p}} M^{hC_p}\rightarrow M^{hC_p}/F$; 
     \item[{\rm (4)}]
     the right vertical map sends $(X,Y)$ to the zero map $X \xto{0} Y$. 
    \end{enumerate}
\end{lemma}

\begin{proof}
    We consider the diagram of stable $\infty$-categories
    \begin{equation}\label{bigsquarew}
    \xymatrix{
    \TCart_p \ar[rr] \ar[dd] &&   \left(\Sp^{BS^1}\right)^{\Delta^2} \ar[d]^{i} \\
    &&  \left(\Sp^{BS^1}\right)^{\Lambda_2^2} \ar[d]^{(\ev_1,\partial_1)} \ar[rr]^{\partial_0} && \left(\Sp^{BS^1}\right)^{\Delta^1}\ar[d]^{(\ev_0, \ev_1)} \\
     \Sp^{BS^1} \ar[rr]^-{(\id, \Nm_{C_p})}  && \Sp^{BS^1}\times  \left(\Sp^{BS^1}\right)^{\Delta^1}  \ar[rr]^{\id \times \ev_1 } && \Sp^{BS^1} \times \Sp^{BS^1}
    }
    \end{equation}
    where the left hand square is the defining pullback for $\TCart_p$ as in
    Definition \ref{defTCart} and the right hand side is induced from the
    diagram of simplicial sets\footnote{Note that we use $\ev_i$ for evaluation
    at the vertex $i$ in $\Delta^n$. For example for $\Delta^1$ we have that
$\ev_0$ corresponds to $\partial^1$ and $\ev_1$ to $\partial^0$.}
    \[
    \xymatrix{
    \Delta^2 \\
    \Lambda_2^2 \ar[u]^{i} && \Delta^1 \ar[ll]_{\partial^0} \\
   \Delta^0  \sqcup \Delta^1 \ar[u]^{\{1\}\sqcup \partial^1} && \Delta^0 \sqcup \Delta^0
    \ar[u]_{\partial^1\sqcup \partial^0} \ar[ll]_{\id \sqcup \partial^0}
    }
    \]
    with $i\colon \Lambda_2^2 \to \Delta^2$ the `inclusion' and $\{1\}$ denotes the map that hits the
    object $\{1\} \in \Lambda_2^2$. The square in this diagram of simplicial sets is a pushout
    of simplicial sets and the lower horizontal map is a monomorphism. Therefore it is a pushout
    in $\Cat_\infty$ and thus the right hand square in~\eqref{bigsquarew}  is 
    a pullback of stable $\infty$-categories. We can insert a further pullback and obtain a diagram

    \begin{equation}\label{finalbig}
    \xymatrix{
    \TCart_p \ar[rr] \ar[d] &&   \left(\Sp^{BS^1}\right)^{\Delta^2} \ar[d]^{i} \\
    \CycSp^\Fr_p \ar[rr]\ar[d] &&  \left(\Sp^{BS^1}\right)^{\Lambda_2^2} \ar[d]^{(\ev_1,\partial_1)} \ar[rr]^{\partial_0} && \left(\Sp^{BS^1}\right)^{\Delta^1}\ar[d]^{(\ev_0, \ev_1)} \\
     \Sp^{BS^1} \ar[rr]^-{(\id , \Nm_{C_p})}  &&  \Sp^{BS^1}\times \left(\Sp^{BS^1}\right)^{\Delta^1}  \ar[rr]^{\id \times \ev_1 } && \Sp^{BS^1} \times \Sp^{BS^1}
    }
    \end{equation}
    where the new term is equivalent to $\CycSp^\Fr_p$ by pasting of pullback
    squares and the left upper square is a pullback for the same reason.
    Finally we use  that for every stable $\infty$-category $\Dscr$ there is a
    pullback square of the form
    \begin{equation}\label{stablepullback}
    \xymatrix{
    \Dscr^{\Delta^2} \ar[d]^i\ar[rr]^{(\ev_0, \cofib(\partial_1))} && \Dscr \times \Dscr \ar[d] \\
    \Dscr^{\Lambda_2^2} \ar[rr] && \Dscr^{\Delta^1},
    }
    \end{equation}
    where the lower horizontal map sends a diagram
    \begin{equation}\label{horn}
    \xymatrix{
    & B\ar[dr] & \\
    A\ar[rr] & & C   }
    \end{equation}
    to the composition $A \to C \to C/B$ and the right hand vertical map sends
    $(X,Y)$ to the zero morphism $X \xto{0} Y$. This pullback is just a
    manifestation of the fact that filling a diagram as \eqref{horn} is
    equivalent to choosing a nullhomotopy of the composition $A \to B \to
    C/B$.

    To finish the proof we paste together the upper square in diagram \eqref{finalbig} with the square \eqref{stablepullback} for $\Dscr = \Sp^{BS^1}$. 
\end{proof}

Now, we can give the desired formula for the mapping spectra in $\TCart_p$. Let
$M = (M, V_M, F_M, \sigma_M)$ and  $N = (N, V_N, F_N, \sigma_N)$ be topological
Cartier modules. There is a map
\[
\vartheta: \MapSp_{\CycSp^\Fr_p}(M, N) \to \MapSp_{\Sp^{BS^1}}\left(M_{hC_p}, \fib(F_N)\right)
\]
which sends a map $g: M \to N$ in $\CycSp^\Fr_p$ to the factorization of the map
\[
g V_M - V_N g_{hC_p} : M_{hC_p} \to N
\]
through the map $\fib(F_N) \to N$ induced from the canonical nullhomotopy 
\[
F_N(g V_M - V_n g_{hC_p}) \simeq g F_M V_M - F_N V_N g_{hC_p} \simeq 0,
\]
arising from the natural transformation $(-)_{hC_p}\xrightarrow{\Nm_{C_p}}(-)^{hC_p}$.

\begin{proposition}\label{mappingspectrum}
    For every pair of topological Cartier modules $M$ and $N$ there is a fiber  sequence 
    \[
    \MapSp_{\TCart_p}(M, N) \to \MapSp_{\CycSp^\Fr_p}(M, N) \xto{\vartheta} \MapSp_{\Sp^{BS^1}}\left(M_{hC_p}, \fib(F_N)\right)
    \]
    of spectra.
\end{proposition}

\begin{proof}
From the pullback square of Lemma~\ref{pullbackTCart} we get an induced pullback square on mapping spectra of the form
\[
\xymatrix{
\MapSp_{\TCart_p}(M, N) \ar[rr]\ar[d] &&  \MapSp_{\Sp^{BS^1}}(M_{hC_p}, N_{hC_p}) \times
    \MapSp_{\Sp^{BS^1}} (M^{hC_p}/F_M, N^{hC_p}/F_N) \ar[d] \\
\MapSp_{\CycSp^\Fr_p}(M, N) \ar[rr] &&  \MapSp_{\left(\Sp^{BS^1}\right)^{\Delta^1}}(M_{hC_p}
        \to M^{hC_p}/F_M, N_{hC_p} \to N^{hC_p}/F_N).
}
\]
Thus, to establish the fiber sequence in question, we have to identify the
cofiber of the right hand map. The right hand map has an obvious section which
is induced from the functor $(\Sp^{BS^1})^{\Delta^1} \xto{(\ev_0, \ev_1)}
\Sp^{BS^1} \times \Sp^{BS^1}$. The fiber of this section is equivalent to
the mapping spectrum
\[
\MapSp_{\Sp^{BS^1}}\left(M_{hC_p}, \Omega (N^{hC_p}/F_N) \right).
\]
Since $ \Omega (N^{hC_p}/F_N) \simeq \fib(F_N)$ this shows that the cofiber of the map
$\MapSp_{\TCart_p}(M, N) \to \MapSp_{\CycSp^\Fr}(M, N)$ has the claimed homotopy type.
Tracing through the identification lets us identify the maps as stated. 
\end{proof}

\begin{remark}\label{remark_fiber}
For every $S^1$-spectrum $M$ we can consider $M_{hC_p}$ as a cyclotomic
spectrum with Frobenius lift, where we choose the Frobenius lift to be the zero
map $M_{hC_p} \to (M_{hC_p})^{hC_p}$. For every cyclotomic spectrum $N$
with Frobenius lift $F_N: N \to N^{hC_p}$ we get an equivalence
\[
\MapSp_{\Sp^{BS^1}}\left(M_{hC_p}, \fib(F_N)\right) \simeq 
\MapSp_{\CycSp^\Fr}(M_{hC_p}, N)
\]
as one can directly verify from the description of mapping spectra in
$\CycSp_p^\Fr$ using~\cite{nikolaus-scholze}*{Proposition~II.1.5}. Using this
equivalence we can rewrite the fiber sequence of Proposition
\ref{mappingspectrum} as
\[
\MapSp_{\TCart_p}(M, N) \to \MapSp_{\CycSp^\Fr_p}(M, N) \to \MapSp_{\CycSp^\Fr_p}(M_{hC_p}, N),
\]
in which the right hand map can be described as the map sending 
$g: M \to N$ to the map $g V_M - V_n f_{hC_p} : M_{hC_p} \to N$ in $\Sp^{BS^1}$ which
canonically refines to a map in $\CycSp^\Fr_p$. This fact will be useful in Section \ref{sub:adjunction}. 
\end{remark}
There is also a `dual' version of Proposition \ref{mappingspectrum} which we record to use later. We consider the $\infty$-category 
\[
\Alg_{(-)_{hC_p}}\left(\Sp^{BS^1}\right) := 
\LEq\left( \xymatrix{\Sp^{BS^1}\ar[rr]<2pt>^{(-)_{hC_p}}\ar[rr]<-2pt>_\id && \Sp^{BS^1}} \right)
\]
of $(-)_{hC_p}$-algebras. We will also abbreviate this $\infty$-category as $\Alg_{(-)_{hC_p}}$.
For every topological Cartier module $(M, F, V, \sigma)$ we get an object $( M, V) \in \Alg_{(-)_{hC_p}}$. 

\begin{proposition}\label{prop_algebra}
For every pair of topological Cartier modules $M, N$ there is a fiber sequence
\[
\MapSp_{\TCart_p}(M, N) \to \MapSp_{\Alg_{(-)_{hC_p}}}(M, N) \to \MapSp_{\Sp^{BS^1}}\left(\cofib{(V_M)}, N^{hC_p} \right) 
\]
where the right hand map admits a dual description to the of  Proposition \ref{mappingspectrum}.
\end{proposition}
\begin{proof}
The proof is entirely dual to the one of Proposition~\ref{mappingspectrum}:  we first get a pullback square of stable $\infty$-categories
\[
\xymatrix{
\TCart_p \ar[r] \ar[d] & \Sp^{BS^1} \times \Sp^{BS^1}\ar[d] \\
\Alg_{(-)_{hC_p}}\ar[r] & \left(\Sp^{BS^1} \right)^{\Delta^1}
}
\]
similar to the one of Lemma \ref{pullbackTCart}.
Here the upper horizontal map sends $(M, V, F, \sigma)$ to $(\fib(V), M^{hC_p})$,  the lower
horizontal map sends $(M,V)$ to the map $(\fib(V) \to M_{hC_p} \xto{\Nm} M^{hC_p})$  and the
right vertical map sends a pair $(X,Y)$ to the zero morphism from $X$ to $Y$. Then as in the
proof of Proposition~\ref{mappingspectrum} we compute the mapping space in this pullback to get the
result.
\end{proof}

\begin{proposition}\label{prop:tdforget}
    The $\infty$-category $\TCart_p$ is stable and presentable and the forgetful functor
    $\TCart_p\rightarrow\Sp^{BS^1}$ preserves limits and colimits.
\end{proposition}

\begin{proof} 
    We follow the proof of~\cite{nikolaus-scholze}*{II.1.5}. The
    $\infty$-category $\TCart_p$ is
    stable because it is the pullback of stable $\infty$-categories along exact functors
    (see~\cite{ha}*{1.1.4.2}).
    Since $\TCart_p$ is the pullback of accessible $\infty$-categories along accessible functors,
    it is accessible itself by~\cite{htt}*{5.4.6.6}. To see that $\TCart_p$ admits colimits and
    that $\TCart_p\rightarrow\Sp^{BS^1}$ preserves them, first suppose that
    $K\rightarrow\TCart_p$ is a
    diagram which admits an extension to $K^\triangleright\rightarrow\TCart_p$ such that the
    induced map $K^\triangleright\rightarrow\Sp^{BS^1}$ is a colimit diagram. Since
    $(-)_{hC_p}$ preserves colimits and since the forgetful functor
    $\CycSp_p^\Fr\rightarrow\Sp^{BS^1}$ preserves colimits
    (see~\cite{nikolaus-scholze}*{Proposition~II.1.5}), it follows from Proposition~\ref{mappingspectrum} 
    that $K^\triangleright\rightarrow\TCart_p$ is a
    colimit too. Now, let $K\rightarrow\TCart_p$ be an arbitrary diagram where $K$ is a small
    simplicial set. The composition $K\rightarrow\TCart_p\rightarrow\Sp^{BS^1}$ admits a colimit
    because $\Sp^{BS^1}$ is presentable. Let $$(\colim_{k\in
    K}M(k))_{hC_p}\we\colim_{k\in K}(M(k)_{hC_p})\xrightarrow{\colim
    V_{M(k)}}\colim_{k\in K}M(k)$$ be the Verschiebung. Similarly, let $$\colim_{k\in
    K}M(k)\xrightarrow{\colim_{k\in K}F_{M(k)}}\colim_{k\in
    K}M(k)^{hC_p}\rightarrow(\colim_{k\in K}M(k))^{hC_p}$$ define the Frobenius
    map. It is enough to show that the composition is equivalent to $\Nm_{C_p}$.
    But, the composition is equivalent to
    \begin{equation}\label{eq:nmcomposition}\colim_{k\in
    K}(M(k)_{hC_p})\xrightarrow{\colim_{k\in K}\Nm_{C_p}|_{M(k)}}\colim_{k\in
    K}(M(k)^{hC^p})\rightarrow(\colim_{k\in K}M(k))^{hC_p}.
    \end{equation}
    Since $\Nm_{\C_p}$ is a natural transformation, there is a canonical
    commutative diagram
    $$\xymatrix{
    \colim_{k\in K}M(k)_{hC_p}\ar[r]\ar[d]^{\colim{\Nm_{C_p}|_{M(k)}}}  &
    \left(\colim_{k\in K}M(k)\right)_{hC_p}\ar[d]^{\Nm_{C_p}|_{\colim M(k)}}\\
        \colim_{k\in K}M(k)^{hC_p}\ar[r]&\left(\colim_{k\in
    K}M(k)\right)^{hC_p}.
    }$$
    The top arrow is an equivalence, and we see that~\eqref{eq:nmcomposition}
    is the desired norm.
    The proof that $\TCart_p\rightarrow\Sp^{BS^1}$ preserves limits is the same,
    using Proposition~\ref{mappingspectrum}, the fact that the forgetful functor
    $\CycSp_p^\Fr\rightarrow\Sp^{BS^1}$ preserves limits
    by~\cite{nikolaus-scholze}*{Proposition~II.1.5}, and the fact
    that $(-)^{hC_p}$ preserves limits.
\end{proof}

\begin{remark}\label{rempullback}
    The $\infty$-category $\TCart_p$ is equivalent to the $\infty$-category of pullback squares in $\Sp^{BS^1}$ of the form
    \[
    \xymatrix{
    M \ar[r]\ar[d] & X\ar[d] \\
    M^{hC_p} \ar[r] & M^{tC_p}
    }
    \]
    (where $X = M/V$). This description is akin to Tate squares and shows that a topological
    Cartier module $M$ gives rise to a genuine $C_p$-spectrum\footnote{The
    $C_p$-spectrum $M$ admits an $S^1$-action which is compatible with the
    $C_p$-equivariant structure, so it is really an $S^1$-spectrum that is
    genuine for the family consisting only of $C_p$ and the trivial group.} whose
    categorical fixed points are equivalent to the underlying spectrum. We will make this analogy rigorous in Section \ref{genuineTCart}.
    This category of pullback squares has an obvious symmetric monoidal
    structure (algebras are such that all objects are algebras and all
    maps are algebra maps), which will be discussed in more detail in future
    work (see also Section~\ref{sec:monoidal}).
\end{remark}

\begin{remark}\label{rem:raynaudcartier}
    Proposition \ref{prop:tdforget} formally implies that $\TCart_p$ has a single
    compact generator $K \in \TCart_p$: we denote the left adjoint to the forgetful
    functor $\TCart_p \to \Sp$ by $L$ and set $K = L(\SS)$. Then $K$ is compact
    since if $M\we\colim_iM_i$ is a filtered colimit in $\TCart_p$, then
    \[
    \MapSp_{\TCart_p}(K, \colim_i M_i) \simeq \MapSp_\Sp(\SS, \colim_i M_i) \simeq
    \colim_i \MapSp_\Sp(\SS, M_i) \simeq \colim_i\MapSp_{\TCart_p}(K,M_i).
    \]
    It is a generator because $\MapSp_{\TCart_p}(K, M) \we 0$ implies by adjunction
    that $M \we 0$. It thus follows that $\TCart_p$ is equivalent to
    the $\infty$-category of module spectra over 
    \[
    \MapSp_{\TCart_p}(K, K) = \MapSp_\Sp(\SS, K) \simeq K.
    \]
    This endomorphism ring spectrum is a spectral version of the Raynaud--Cartier
    ring which controls the operations on the de Rham--Witt complex. We will analyze
    this ring spectrum in future work.
\end{remark}

\subsection{The $t$-structure on $\TCart_p$}\label{sub:ttcm}

We introduce a $t$-structure on $\TCart_p$ which will turn out to be compatible
via the constructions in Examples~\ref{TrDieu} and~\ref{ex:tcm2cyc} with the
$t$-structure introduced in Section~\ref{sec:cyct} on cyclotomic spectra.

\begin{definition}
    Let $(\TCart_p)_{\geq 0}$
    denote the full subcategory of topological Cartier modules whose underlying
    spectrum is connective and let $(\TCart_p)_{\leq 0}$ denote the full subcategory whose
    underlying spectrum is coconnective. We will call the objects of these
    $\infty$-categories the connective and coconnective $p$-typical topological Cartier
    modules, respectively.
\end{definition}

\begin{proposition}\label{prop:tdt}
    The pair $((\TCart_p)_{\geq 0},(\TCart_p)_{\leq 0})$ defines an accessible
    $t$-structure on $\TCart_p$
    with the following properties:
    \begin{enumerate}
        \item[{\rm (i)}] the forgetful functor $\TCart_p\rightarrow\Sp^{BS^1}$ is $t$-exact;
        \item[{\rm (ii)}] the functor $(-)/V\colon\TCart_p\rightarrow\CycSp_p$ constructed
            in Example~\ref{ex:tcm2cyc} is right $t$-exact;
        \item[{\rm (iii)}] the $t$-structure is compatible with filtered colimits;
        \item[{\rm (iv)}] the $t$-structure is left and right complete;
        \item[{\rm (v)}] the heart $\TCart^\heart_p$ is equivalent to the abelian
            category of abelian groups equipped with endomorphisms $V$ and $F$
            such that $FV=p$.
    \end{enumerate}
\end{proposition}

\begin{proof}
    Because $(\TCart_p)_{\geq 0}$ is presentable and closed under colimits and
    extensions in
    $\TCart_p$, by~\cite{ha}*{1.4.4.11},
    there is some $t$-structure $((\TCart_p)_{\geq 0},\Cscr)$ on $\TCart_p$ with
    connective part $(\TCart_p)_{\geq 0}$ and coconnective part some full
    subcategory $\Cscr\subseteq\TCart_p$. Note that by
    Proposition~\ref{mappingspectrum}, for every $M\in(\TCart_p)_{\geq 0}$ and
    every $N\in(\TCart_p)_{\leq
    -1}$, the mapping space $\Map_{\TCart_p}(M,N)$ is contractible. It follows that
    $(\TCart_p)_{\leq 0}\subseteq\Cscr$. To prove that the inclusion is an
    equivalence, we will prove that every topological Cartier module
    structure on a spectrum $M$ with $S^1$-action extends to a structure on
    $\tau_{\geq 0}M$ and $\tau_{\leq -1}M$. In fact, it will be enough to do
    this for $\tau_{\geq 0}M$ since it follows then for $\tau_{\leq -1}M$ by
    taking cofibers.

    Consider the composition $(\tau_{\geq 0}M)_{hC_p}\rightarrow
    M_{hC_p}\xrightarrow{V}M$. Since $(\tau_{\geq 0}M)_{hC_p}$ is connective,
    this factors through the connective cover $\tau_{\geq 0}M\rightarrow M$,
    giving a commutative diagram
    $$\xymatrix{
    (\tau_{\geq 0}M)_{hC_p}\ar[r]\ar[d]&\tau_{\geq 0}M\ar[d]\\
    M_{hC_p}\ar[r]_V&M.
    }$$ Similarly, the composition $\tau_{\geq 0}M\rightarrow
    M\xrightarrow{F}M^{hC_p}$ factors through $\tau_{\geq
    0}(M^{hC_p})\we\tau_{\geq 0}((\tau_{\geq 0}M)^{hC_p})$, where the
    equivalence follows because the right adjoint functor $\tau_{\geq 0}$ commutes with limits.
    Hence, we get a commutative diagram
    $$\xymatrix{
        \tau_{\geq 0}M\ar[r]\ar[d]&(\tau_{\geq 0}M)^{hC_p}\ar[d]\\
        M\ar[r]_F&M^{hC_p}.}$$
    Consider the prism $\Delta^2\times\Delta^1$. We would like to construct a
    new object of $(\Sp^{BS^1})^{\Delta^2}$ defining a topological Cartier
    module structure on $\tau_{\geq 0}M$ where the $F$ and $V$ maps are the
    top horizontal arrows in the two commutative diagrams above. Moreover, we
    need to have a map of topological Cartier modules $\tau_{\geq
    0}M\rightarrow M$. This will be provided by a map
    $\Delta^1\rightarrow(\Sp^{BS^1})^{\Delta^2}$ or, by adjunction, by a map
    $\Delta^2\times\Delta^1\rightarrow\Sp^{BS^1}$.

    The commutative diagrams constructed above as well as the natural
    transformation $\Nm_{C_p}$ from $(-)_{hC_p}$ to $(-)^{hC_p}$ provide a map
    from $(\partial\Delta^2)\times\Delta^1$ to $\Sp^{BS^1}$. Moreover, the
    `bottom' of this triangular cylinder can be filled in with the $2$-simplex
    $\sigma$ associated to the topological Cartier module structure on $M$.
    Decomposing $\Delta^2\times\Delta^1$ into three $3$-simplices (tetrahedra),
    we can use the fact that we can fill inner horns (since $\Sp^{BS^1}$ is an
    $\infty$-category), to inductively fill in $\Delta^2\times\Delta^1$, thus
    obtaining the desired map.

    Now, the cofiber of $\tau_{\geq 0}M\rightarrow M$ is a $\TCart_p$-structure on
    $\tau_{\leq -1}M$. Since $\tau_{\geq 0}M\in(\TCart_p)_{\geq 0}$ and $\tau_{\leq
    -1}M\in(\TCart_p)_{\leq -1}\subseteq\Cscr[-1]$, it follows that in fact this is
    the truncation sequence associated to $M$ in the $t$-structure $((\TCart_p)_{\geq
    0},\Cscr)$. In particular, we see that $(\TCart_p)_{\leq 0}\we\Cscr$, as desired.

    Accessibility of the $t$-structure follows because $(\TCart_p)_{\geq 0}$ is
    presentable. The $t$-exactness of $\TCart_p\rightarrow\Sp^{BS^1}$ follows by
    construction. The right $t$-exactness of $\TCart_p\rightarrow\CycSp_p$ follows
    because if $M\in(\TCart_p)_{\geq 0}$, then $M_{hC_p}$ is also connective and hence
    so is the cofiber $M/V$. Compatibility with filtered colimits follows
    because filtered colimits are preserved by the forgetful functor
    $\TCart_p\rightarrow\Sp^{BS^1}$ by Proposition~\ref{prop:tdforget} and
    because the $t$-structure on $\Sp^{BS^1}$ is
    compatible with filtered colimits by Proposition~\ref{prop:xsp}.
    Right completeness follows from the right complete version
    of~\cite{ha}*{Proposition~1.2.1.19} since $\TCart_p$ is right separated (by the
    conservativity of the forgetful functor $\TCart_p\rightarrow\Sp^{BS^1}$) and
    because $(\TCart_p)_{\leq 0}$ is closed under filtered
    colimits and hence in particular countable coproducts inside of $\TCart_p$.
    For left completeness, use that $\TCart_p$ is left separated and recall
    that $\TCart_p\rightarrow\Sp^{BS^1}$ preserves limits by
    Proposition~\ref{prop:tdforget}. It follows, as usual by
    conservativity, that $(\TCart_p)_{\geq 0}$ is closed under countable products. We
    conclude by~\cite{ha}*{Proposition~1.2.1.19}.
    Finally, we see that the heart consists precisely of topological
    Cartier modules whose underlying spectrum is discrete. Since the
    $S^1$-action is automatically trivial, we are reduced to the objects of
    Example~\ref{example_classical}.
\end{proof}

\begin{remark}
    Alternatively, to prove the existence, accessibility, and completeness of
    the $t$-structure on $\TCart_p$, one can show that the ring spectrum $K$ of
    Remark~\ref{rem:raynaudcartier} is connective and identify the
    $t$-structure above on $\TCart_p$ with the Postnikov $t$-structure on
    $K$-module spectra, which has the desired properties
    by~\cite{ha}*{7.1.1.13}.
\end{remark}

\subsection{Topological Cartier modules and cyclotomic spectra}\label{sub:adjunction}

The functor
\[
(-) /V: \TCart_p \to \CycSp_p
\]
of Example~\ref{ex:tcm2cyc}
preserves all colimits and thus is a left adjoint by the
adjoint functor theorem. The first result of this section identifies the right
adjoint.

\begin{proposition}\label{thm:tr}
    The functor $\TR\colon\CycSp_p\rightarrow\TCart_p$ is right adjoint to
    $(-)/V$.
\end{proposition}

\begin{construction}\label{const:tr}
Let us review our candidate functor $\TR$ first, which was sketched in
Example~\ref{TrDieu}. By construction, 
for every cyclotomic spectrum $X$ (no boundedness assumptions) the spectrum 
\[
\TR(X) = \left( \ldots \times_{(X^{tC_p})^{hC_{p^2}}}  X^{hC_{p^2}} \times_{(X^{tC_p})^{hC_p}} X^{hC_p} \times_{X^{tC_p}} X \right)
\] 
comes equipped with a canonical $S^1$-action and an $S^1$-equivariant map $\pi: \TR(X) \to X$. 
Moreover there is an evident $S^1$-equivariant equivalence
\begin{equation}\label{eq:trpullback}
    \Phi\colon \TR(X) \xto{\simeq} X \times_{X^{tC_p}}
    \TR(X)^{hC_p},
\end{equation}
where the pullback involves the  cyclotomic Frobenius $X\rightarrow X^{tC_p}$ and the map
$\TR(X)^{hC_p}\rightarrow X^{tC_p}$ is given by one of the two equivalent compositions in the square 
\[
\xymatrix{
\TR(X)^{hC_p}  \ar[d]^-{\can} \ar[r]^-{\pi^{hC_p} } &  X^{hC_p} \ar[d]^{\can} \\
\TR(X)^{tC_p} \ar[r]^-{\pi^{tC_p}} &  X^{tC_p}.
}
\]
Under $\Phi$, the map $\pi$ corresponds to projection onto the first factor. 
Now let us construct the structure of a topological Cartier module  $\TR(X)$
from Example \ref{TrDieu} a little bit more carefully. The map $F: \TR(X) \to
\TR(X)^{hC_p}$ is obtained as the composite
\[
F \colon \quad\TR(X) \xto{\Phi} X \times_{X^{tC_p}} \TR(X)^{hC_p} \xto{\mathrm{pr}_2} \TR(X)^{hC_p}
\]
and the map $V$ as the composite
\[
V\colon \TR(X)_{hC_p} \xto{(0,\mathrm{Nm})}  X \times_{X^{tC_p}} \TR(X)^{hC_p} \xto{\Phi^{-1}} \TR(X)
\]
where the first map is zero onto the first
factor and the norm into the second factor, using the canonical nullhomotopy
of $\can \circ \mathrm{Nm}$. Now the composition $F \circ V$ comes by
definition with a homotopy to the norm.

Moreover by construction of $V$ the composite $\TR(X)_{hC_p} \xto{V} \TR(X) \xto{\pi} X$ is
canonically nullhomotopic and thus gives rise to a map $\TR(X)/V \to X$. This map admits a refinement to a map of cyclotomic spectra induced by the compatible nullhomotopies of the horizontal maps in the commutative diagram
\[
\xymatrix{
\TR(X)_{hC_p} \ar[r]^-V \ar[d]& \TR(X) \ar[d]^F\ar[r]^-\pi & X \ar[d]\\
\TR(X)_{hC_p} \ar[r]^-{\mathrm{Nm}} & \TR(X)^{hC_p} \ar[r]  & X^{tC_p} \ .
}
\]
Unwinding the definitions, we see that the map $\TR(X)/V \to X$ is an equivalence precisely if the map
\begin{equation}\label{eq:equivalenttate}
    \pi^{tC_p}\colon \TR(X)^{tC_p} \to X^{tC_p}
\end{equation}
is an equivalence. Indeed, the fiber $F$ of $\TR(X)\rightarrow X$ is naturally
equivalent to the fiber of $\TR(X)^{hC_p}\rightarrow X^{tC_p}$ and there is a
natural map from $\TR(X)_{hC_p}$ to this fiber. We have $\TR(X)_{hC_p}\we F$
(equivalently, $\TR(X)/V\we X$) if and only if $X^{tC_p}$ is the cofiber of the norm
map $\TR(X)_{hC_p}\rightarrow\TR(X)^{hC_p}$.
\end{construction}

In general, the map $\TR(X)/V\rightarrow X$ induces a natural transformation of
functors $\TR\rightarrow R$, where $R$ is the right adjoint to $(-)/V$.

\begin{proof}[Proof of Proposition~\ref{thm:tr}]
  We have to verify that $\TR$ satisfies the universal property of the right
  adjoint, i.e., that for every pair of a topological Cartier module $M$ and a
  cyclotomic spectrum $X$ the induced map
  \[
  \MapSp_{\TCart_p}(M, \TR(X)) \to   \MapSp_{\CycSp_p}(M/V, X) 
  \]
is an equivalence of spectra. We use the fiber sequence for the mapping space in $\TCart_p$ in the form established in Remark \ref{remark_fiber}:
\[
\MapSp_{\TCart_p}(M, \TR(X)) \to \MapSp_{\CycSp^\Fr_p}(M, \TR(X)) \to \MapSp_{\CycSp^\Fr_p}(M_{hC_p}, \TR(X)) \ .
\]
Recall that the objects of $\CycSp^\Fr_p$ are spectra $M$ with $S^1$-action equipped
    with an $S^1$-equivariant map $\psi_p\colon M\rightarrow M^{hC_p}$.
    Composing with the canonical map $M^{hC_p}\rightarrow M^{tC_p}$ defines a
    $p$-typical cyclotomic spectrum, giving a colimit preserving functor
    $$\CycSp_p^\Fr\rightarrow\CycSp_p.$$ A result of
    A. Krause and the second author~\cite{krause-nikolaus}*{Proposition~10.3} says
    that the right adjoint is given by $\TR$. Using this we can retranslate the fiber sequence above into 
    \[
\MapSp_{\TCart_p}(M, \TR(X)) \to \MapSp_{\CycSp_p}(M, X) \to \MapSp_{\CycSp_p}(M_{hC_p}, X) \ .
 \]
 where the cyclotomic spectrum $M_{hC_p}$ has the trivial Frobenius $M_{hC_p} \xto{0}
 (M_{hC_p})^{tC_p}$. Using the description of the maps we see that the right hand map in
 this fiber sequence is the map induced from the map $V: M_{hC_p} \to M$ of cyclotomic
 spectra.\footnote{Note that this is not a map of cyclotomic spectra with
     Frobenius lifts,
even though source and target have Frobenius lifts.} Now the claim follows from the
fact that the cyclotomic spectrum $M/V$ is by construction the cofiber of the map
$M_{hC_p} \to M$.
\end{proof}

\begin{lemma}\label{lem:trbounded}
    The natural map $\TR(X)/V\rightarrow X$ is an equivalence if $X$ is bounded
    below.
\end{lemma}

\begin{proof}
    We first recall that $(-)_{hC_p}$ and $(-)^{tC_p}$ preserve sequential
    limits of uniformly bounded below spectra with $S^1$-action by
    Lemma~\ref{lem:littlelimits}.
    Now, as noted above, $\TR(X)/V\rightarrow X$ is an equivalence if and only if
    the arrow of~\eqref{eq:equivalenttate} is an equivalence. There are
    cofiber sequences $\TR^n(X)_{hC_p}\rightarrow\TR^{n+1}(X)\rightarrow X$ for
    each $n\geq 0$. Applying $(-)^{tC_p}$ and using the Tate orbit lemma, which
    holds as $X$ is bounded below, we
    see that $\TR^{n+1}(X)^{tC_p}\we X^{tC_p}$ for all $n$. Since
    $(-)^{tC_p}$ commutes with limits of towers of uniformly bounded below
    spectra by the discussion at the beginning of the proof, $\TR(X)^{tC_p}\we X^{tC_p}$.
\end{proof}

Let $\TCart_p^-$ and $\CycSp_p^-$ denote the full subcategories of $\TCart_p$
and $\CycSp_p$ on the bounded below objects.

\newcommand{\CartV}{\widehat{\Cart}_p}
\newcommand{\TCartV}{\widehat{\TCart}_p}
\newcommand{\TCartVm}{\widehat{\TCart}{}_p^-}

\begin{definition}\label{def_complete}
    Say that a bounded below $M\in\TCart_p^-$ is $V$-complete if 
    the limit of the tower $$\cdots\rightarrow
    M_{hC_{p^2}}\xrightarrow{V_{hC_p}}M_{hC_p}\xrightarrow{V}M$$ vanishes.
    Write $V^n$ for the composition
    $M_{hC_{p^n}}\xrightarrow{V_{hC_{p^{n-1}}}}M_{hC_{p^{n-1}}}\rightarrow\cdots\rightarrow
    M_{hC_p}\xrightarrow{V}M$. Let $\TCartVm\subseteq\TCart_p^-$ be the full subcategory of
    $V$-complete bounded below $p$-typical topological Cartier modules.
    We will say that a map $M\rightarrow \hat{M}$ in $\TCart_p^-$ is a
    $V$-completion if $\hat{M}$ is $V$-complete and if the natural forgetful map
    $\MapSp_{\TCart_p}(\hat{M},N)\rightarrow\MapSp_{\TCart_p}(M,N)$ is an
    equivalence for every $N$ in $\TCartVm$. 
\end{definition}

\begin{theorem}\label{thm:boundedbelow}
    The functor $\TR\colon\CycSp_p^-\rightarrow\TCart_p^-$ is fully faithful
    and $t$-exact
    with left adjoint $M\mapsto M/V$. The essential image is the full
    subcategory of $V$-complete bounded below $p$-typical topological Cartier
    modules.
\end{theorem}

\begin{proof}
    The first claim follows immediately from the fact that $M\mapsto M/V$ is right
    $t$-exact, that $\TR$ is $t$-exact, and that the counit map
    $\TR(X)/V\rightarrow X$ is an equivalence for $X\in\CycSp_p$ bounded below,
    which was proved in Lemma~\ref{lem:trbounded}. The second claim follows
    from the next proposition.
\end{proof}

\begin{proposition}
    If $M\in\TCart_p^-$ is bounded below, then $\lim_nM/V^n$ admits the structure
    of a $p$-typical Cartier module and there is a natural map
    $M\rightarrow\lim_nM/V^n$ exhibiting $\lim_nM/V^n$ as the $V$-completion of $M$.
\end{proposition}

\begin{proof}
    By the first part of Theorem~\ref{thm:boundedbelow}, there are two claims
    to check: first that $\TR(X)$ is $V$-complete for any bounded below
    $p$-typical cyclotomic spectrum $X$ (apply this to $X=M/V$); second that $\TR(M/V)\we\lim_n M/V^n$
    for a bounded below $p$-typical topological Cartier module. Indeed, the
    $p$-typical Cartier module structure on $\lim_n M/V^n$ comes from that on
    $\TR(M/V)$ via Construction~\ref{const:tr}. The natural map is then the unit map
    $M\rightarrow\TR(M/V)$ of the adjunction. We have that $\TR(N/V)$ is
    $V$-complete and hence $N\we\TR(N/V)$ for any $V$-complete $M$. In
    particular,
    $\MapSp_{\TCart_p}(\TR(M/V),N)\we\MapSp_{\TCart_p}(\TR(M/V),\TR(N/V))\we\MapSp_{\CycSp_p}(M/V,N/V)\we\MapSp_{\TCart_p}(M,N)$
    for $N\in\TCartVm$ by Proposition~\ref{thm:tr},
    which shows that $\TR(M/V)$ is a $V$-completion of $M$.

    For the first claim, note that for each $n$ the cofiber sequence
    $$\TR(X)_{hC_{p^{n}}}/\TR(X)_{hC_{p^{n+1}}}\rightarrow\TR(X)/\TR(X)_{hC_{p^{n+1}}}\rightarrow\TR(X)/\TR(X)_{hC_{p^{n}}}$$
    is equivalent to the cofiber sequence $$X_{hC_{p^{n}}}\rightarrow
    X^{C_{p^{n}}}\xrightarrow{R} X^{C_{p^{n-1}}},$$
    from which it follows immediately that
    $\TR(X)\we\lim_n\TR(X)/\TR(X)_{hC_{p^{n+1}}}$, or equivalently that
    $\lim_n\TR(X)_{hC_{p^{n+1}}}\we 0$.
    Similarly, the cofiber sequence
    $$(M/V)_{hC_{p^n}}\rightarrow(M/V)^{C_{p^n}}\xrightarrow{R}(M/V)^{C_{p^{n-1}}}$$
    is inductively equivalent to $$M_{hC_{p^n}}/M_{hC_{p^{n+1}}}\rightarrow
    M/M_{hC_{p^{n+1}}}\rightarrow M/M_{hC_{p^n}},$$
    which is what we wanted to show.
\end{proof}

\begin{remark}\label{modV}
We have defined what it means to be $V$-complete for a bounded below
topological Cartier module. In general, one can consider the Bousfield
localization of $\TCart_p$ at the mod $V$-equivalences, i.e. all the maps in
$\TCart_p$ that become equivalences after applying the functor $\TCart_p \to
\CycSp_p$. This defines a Bousfield
localization of the $\infty$-category $\TCart_p$ by the results of
\cite[Section 5.5.4]{htt}.  The local objects are those  topological Cartier
modules $N$ for which the induced map
\[
\MapSp_{\TCart_p}(M, N) = 0
\]
whenever $M/V = 0$. It follows from the adjunction of Proposition \ref{thm:tr}
that $\TR(X)$ is mod $V$-local for every cyclotomic spectrum $X$. In general,
the map $M \to \TR(M/V)$ is not a mod $V$-equivalence, but it is if $M$ is
bounded below. Therefore, a bounded below topological Cartier module is mod
$V$-local precisely if it is $V$-complete in the sense of Definition
\ref{def_complete}. For a bounded below topological Cartier module the map $M
\to \TR(M/V) \simeq \lim M/V^n$ is the mod $V$-localization. When $M$ is not
bounded below, this
localization is mysterious and we do not know how to describe it or how to
understand the mod $V$-local objects. See also
Proposition~\ref{prop:rightcompletion} below.
\end{remark}

\subsection{The heart}\label{sub:tcmheart}

From Theorem~\ref{thm:boundedbelow}, we see that $\CycSp_p^\heart$ is the full subcategory
of $\TCart_p^\heart$ consisting of $V$-complete objects.
This notion of $V$-complete is in the ambient stable $\infty$-category
$\TCart_p^-$. The goal of this section is to describe a more intrinsic notion, making
reference only to the abelian category of $p$-typical Cartier modules, that
agrees with $V$-completeness for objects of $\TCart_p^\heart$.

\begin{definition}\label{Cartier_module_complete}
    If $M$ is a $p$-typical Cartier module, we say that $M$ is {\bf derived
    $V$-complete} if the limit $\lim(\cdots\rightarrow
    M\xrightarrow{V}M\xrightarrow{V}M)$ vanishes (in the
    derived category $\Dscr(\ZZ)$), or 
    equivalently, if the canonical map $M \to \lim_n
    \cofib(M\xrightarrow{V^n}M)$ is an equivalence.
    We denote the category of $p$-typical Cartier modules by $\Cart_p$ and the full subcategory
    of derived $V$-complete Cartier modules by $\CartV$. 
\end{definition}

\begin{lemma}\label{lem:complete}
    Let $M$ be a $p$-typical topological Cartier module concentrated in a single degree as
    in Example~\ref{example_classical}. Then $M$ is $V$-complete as a
    $p$-typical topological Cartier module if and only if it is derived $V$-complete as a
    $p$-typical Cartier module.
\end{lemma}

\begin{proof}
    First, assume that $M$ is derived $V$-complete. We have to show that the limit of the diagram
    \begin{equation}\label{diagram_vanish}
    \cdots \xto{V_{hC_{p^2}}} M_{hC_{p^2}} \xto{V_{hC_p}} M_{hC_p} \xto{V} M
    \end{equation}
    of spectra
    vanishes as well. This diagram can be written as the diagonal of the diagram
    \[
    \xymatrix{
    & \vdots \ar[d]^{V \otimes BC_{p^2}} & \vdots \ar[d]^{V \otimes BC_p} & \vdots \ar[d]^{V} \\
    \cdots \ar[r]^-{\id \otimes t_2} & M \otimes BC_{p^2} \ar[r]^-{\id \otimes t_1}\ar[d]^{V \otimes BC_{p^2}} & M \otimes BC_p \ar[r]^-{\id \otimes t_0}\ar[d]^{V \otimes BC_p} & M \ar[d]^{V} \\ 
    \cdots \ar[r]^-{\id \otimes t_2} & M \otimes BC_{p^2} \ar[r]^-{\id \otimes t_1} & M \otimes BC_p \ar[r]^-{\id \otimes t_0} & M
    }
    \] 
    where $t_n\colon BC_{p^{n+1}} \to BC_{p^n}$ is the canonical projection. Then
    the diagonal limit can also be computed by first computing the vertical
    limits followed by the horizontal one. We claim that all the vertical limits are already trivial so that  also the limit of \eqref{diagram_vanish} vanishes. For the rightmost column this is by assumption true. For the other columns we invoke part (b) of Lemma \ref{lem:littlelimits} to see that they also vanish.

    The converse statement follows since for any diagram of spectra 
    $\cdots \to X_2 \to X_1 \to X_0$ with vanishing limit the induced diagram
    $\cdots\to
    \pi_0(X_2) \to \pi_0(X_1) \to \pi_0(X_0)$ as a diagram of abelian groups has
    vanishing derived limit. This can be seen by writing the  inverse limit as the
    fiber of a map
    \[
    \prod X_i \to \prod X_i
    \]
    which is then an equivalence since the limit vanishes. Thus it induces a
    bijection on $\pi_0$ which implies the claim since the homotopy groups of
    the derived limit of the $\pi_0$-diagram are the kernel
    ($\lim$) and cokernel ($\lim^1$) of this
    map.
\end{proof}

Theorem~\ref{thm:boundedbelow} and Lemma~\ref{lem:complete} imply the following
theorem, which completes the proof of Theorem~\ref{mt:t}.

\begin{theorem}\label{thm:heartid}
    The heart $\CycSp^\heart_p\subseteq\TCart_p^\heart$ is the full subcategory of
    derived $V$-complete $p$-typical Cartier modules.
\end{theorem}

\begin{proof}
    Since $\TR\colon\CycSp_p^-\rightarrow\TCart_p^-$ is fully faithful and
    $t$-exact, it embeds the heart $\CycSp_p^\heart$ fully faithfully into
    $\TCart_p^\heart$. An object $M$ of $\TCart_p^\heart$ is moreover in
    $\CycSp_p^\heart$ if and only if $M\rightarrow\TR(M/V)$ is an equivalence.
    Indeed, if $M\in\TCart_p^\heart$ is equivalent to $\TR(N)$ for some
    $N\in\CycSp_p^\heart$, then $M$ is $V$-complete by
    Theorem~\ref{thm:boundedbelow} and $M\we\TR(M/V)$. On the other hand, if
    $M\we\TR(M/V)$, then $\pi_i^\cyc M/V=0$ for
    $i>0$ since $\TR$ is fully faithful on bounded below objects. Since the cyclotomic $t$-structure is left separated, this implies
    $M/V\in\CycSp^\heart_p$. Now, $M\we\TR(M/V)$ if and only if $M$ is
    $V$-complete as a $p$-typical
    topological Cartier module by Theorem~\ref{thm:boundedbelow}. This happens if and only if $M$ is derived
    $V$-complete by Lemma~\ref{lem:complete}.
\end{proof}

Using the theorem, we can prove that the $t$-structure on $\CycSp_p$ is not
compatible with filtered colimits.

\begin{example}\label{ex:filteredcolimits}
    For the purposes of this example, we consider
    every abelian group $M$ as a $p$-typical Cartier module with $V = p$ and
    $F= \id$. Then, $M$
    is derived $V$-complete if and only if $M$ is derived $p$-complete. Now we have the
    filtered colimit $$\colim_n \ZZ/p^n = \QQ_p / \ZZ_p,$$ which we consider as a
    filtered colimit of $p$-typical Cartier modules. The $p$-typical Cartier modules $\ZZ/p^n$ are derived
    $V$-complete. The module $\QQ_p / \ZZ_p$ is not derived $V$-complete and the derived $V$-completion
    is given as the cofiber of the derived $V$-completions of $\ZZ_p$ and $\QQ_p$. The first
    is already derived $V$-complete and the derived $V$-completion of $\QQ_p$
    is trivial. Hence, the derived $V$-completion of $\QQ_p/\ZZ_p$ is
    $\ZZ_p[1]$. Thus the heart $\CycSp_p^\heart$ is not closed under filtered colimits in
    $(\CycSp_p)_{\geq 0}$.
\end{example}

Theorem~\ref{thm:heartid} generalizes by induction to give a description of all $V$-complete
bounded below $p$-typical topological Cartier modules. The reader should
compare this to the fact that a complex $M\in\Dscr(\ZZ)$ is $p$-complete if and
only each homology group $\H_n(M)$ is derived $p$-complete.

\begin{proposition}\label{prop:vcompletemodule}
    A bounded below $p$-typical topological Cartier module $M$ is $V$-complete
    precisely if all its homotopy groups are derived $V$-compete when
    considered as $p$-typical Cartier modules. 
\end{proposition}

\begin{proof}
    Assume that $M$ is connective and $V$-complete. Then, as in the second part
    of the proof of Lemma~\ref{lem:complete}, we see that $\pi_0M$ is derived $V$-complete as a
    $p$-typical Cartier module. We consider the fiber sequence $\tau_{\geq 1} M
    \to M \to \pi_0M$. By Lemma~\ref{lem:complete}, $\pi_0M$ is $V$-complete and by
    assumption $M$ is. It follows that $\tau_{\geq 1}M$ is as well. By
    induction we get that all homotopy groups are derived $V$-complete when considered
    as classical Cartier modules.

    For the converse assume that all homotopy groups are derived $V$-complete as
    $p$-typical Cartier modules. We first write $M$ as the limit of its Postnikov
    tower. Since completion commutes with this limit (the construction uses
    homotopy orbits, cofibers, and a sequential limit, all of which
    preserve limits of uniformly bounded below objects by
    Lemma~\ref{lem:littlelimits}), we can assume that $M$ is bounded above. Then, we can
    reduce to a single homotopy group by using iterated extensions as in the
    first part of the proof. Finally, we can apply Lemma~\ref{lem:complete}.
\end{proof}

We now  describe the algebraic `completion' functor for algebraic Cartier
modules to get a better understanding of the heart. This will be relevant in
the next section.

\begin{lemma}\label{alg_completion}
    The inclusion of derived $V$-complete $p$-typical Cartier modules into $p$-typical Cartier modules is reflective.
    The left adjoint $LM$ can be described as $\pi_0$ of the $V$-completion of $M$ as a
    topological Cartier module or as $\pi_0\TR(M/V)$. It is also equivalent to
    $\H_0$ of the derived inverse limit of the tower
    \[
    \cdots \to \cofib(M \xto{V^2} M) \to \cofib(M \xto{V} M)
    \]
    in the derived category $\Dscr(\ZZ)$.
\end{lemma}

\begin{proof}
By the previous results we have a commutative diagram
\[
\xymatrix{
\CartV  \ar[r]\ar[d] & \Cart_p\ar[d] \\
(\TCartV)_{\geq 0} \ar[r] & (\TCart_p)_{\geq 0},
}
\]
where $(\TCartV)_{\geq 0}=(\TCart_p)_{\geq 0}\cap\TCartVm$.
Passing to left adjoints and applying to discrete topological Cartier modules we get the first claim. Now observe that $\pi_0$ of the inverse limit
\[
M^\wedge_V = \lim( \cdots\to M/V^3 \to M/V^2 \to M/V)
\]
(recall that $M/V^n$ means the cofiber of the map $V^n: M_{hC_{p^n}} \to M$) is equivalent to $\pi_0$ of the inverse limit
\[
\lim( \cdots\to \tau_{\leq 1} (M/V^3) \to \tau_{\leq 1} (M/V^2) \to \tau_{\leq 1} (M/V))
\]
and that the canonical map $\tau_{\leq 1} (M/V^n)$ to $\cofib(M \xto{V^n} M)$ is an equivalence which follows immediately from the long exact sequences.
\end{proof}

We warn the reader that $L$ is a slightly non-standard operation and we will refer to it as the algebraic derived $V$-completion.
The crucial property is that $L: \Cart_p \to \CartV$ is a left adjoint (and hence
is right exact) and that for $V$-torsion free modules it is just given by the
usual inverse limit $\lim_nM/V^n$.

\begin{lemma}\label{torsionfree}
    If $M$ is $V$-torsion free, then so is $LM$.
\end{lemma}

\begin{proof}
    Since $M$ is $V$-torsion free,
    the kernel of $V\colon\lim_n M/V^n \to \lim_n M/V^n$ is given by the inverse
    limit of the kernels of the maps
    \[
        V\colon \coker(M\xrightarrow{V^n}M)\rightarrow\coker(M\xrightarrow{V^n}M) .
    \]
    But these kernels are isomorphic to $\coker(M\xrightarrow{V}M)$ through the
    map $\coker(M\xrightarrow{V}M) \to \ker(\coker(M\xrightarrow{V^n}) \to
        \coker(M\xrightarrow{V^n}M)$ which sends $m$ in $M$ to $V^{n-1}(m)$. Under this equivalence the diagram over which we have to take the limit is the diagram
    \[
        \cdots \xto{0} \coker(M\xrightarrow{V}M) \xto{0}
        \coker(M\xrightarrow{V}M) \xto{0} \coker(M\xrightarrow{V}M)
    \]
    whose limit is $0$. This shows that $LM$ is $V$-torsion free. 
\end{proof}

\begin{lemma}\label{lemmodVequivalence}
    For a map $M \to N$ between $V$-torsion free Cartier modules, the induced map
    $LM \to LN$ on algebraic $V$-completions is an isomorphism precisely if the
    induced map $\coker(M\xrightarrow{V}M)\rightarrow\coker(N\xrightarrow{V}N)$
    is an isomorphism.
\end{lemma}

\begin{proof}
By the derived $V$-completion we shall mean the object of the derived category given on underlying abelian groups by the (derived) inverse limit
\[
    \lim_n \cofib(M\xrightarrow{V^n}M).
\]
The $V$-operator on this limit can be described in the obvious way coming from
a map of diagrams. We claim that a map $M \to N$ between not necessarily
$V$-torsion free $p$-typical Cartier modules induces an equivalence on derived
$V$-completions precisely if it induces an equivalence
$\cofib(M\xrightarrow{V}M)\rightarrow\cofib(N\xrightarrow{V}N)$. 
To see this note that $\cofib(\lim_n
\cofib(M\xrightarrow{V^n}M)\xrightarrow{V}\lim_n\cofib(M\xrightarrow{V^n}M))
\simeq \cofib(M\xrightarrow{V}M)$ which implies that a map
which is an equivalence after derived $V$-completion is also an equivalence
after mod $V$ reduction. Conversely, if $\cofib(M\xrightarrow{V}M) \to
\cofib(N\xrightarrow{V}N)$ is an equivalence, then one shows inductively using the fiber sequence
\[
    \cofib(M\xrightarrow{V}M) \to \cofib(M\xrightarrow{V^n}M) \to
    \cofib(M\xrightarrow{V^{n-1}}M)
\]
and the analogous one for $N$ to show that
that $\cofib(M\xrightarrow{V^n}M) \to\cofib(N\xrightarrow{V^n}N)$ is an equivalence for every $n$ which implies that the map on the limit is an equivalence.
\end{proof}

To conclude, we note that the right $t$-completion of the $\infty$-category of cyclotomic spectra is a
localization of $\TCart_p$. This will not be needed in the rest of the paper.

\begin{proposition}\label{prop:rightcompletion}
    The functor $(-)/V\colon\TCart_p\rightarrow\CycSp_p$ factors through the
    right completion 
    \[
    \Sp((\CycSp_p)_{\geq 0})\rightarrow\CycSp_p
    \]
     of
    $p$-typical cyclotomic spectra with respect to the cyclotomic
    $t$-structure. The induced map $\TCart_p\rightarrow\Sp((\CycSp_p)_{\geq 0})$
    exhibits the right completion as a localization of $\TCart_p$. The local objects are those topological Cartier modules all of whose
homotopy groups are derived $V$-complete.\footnote{We would like to thank the anonymous referee for suggesting this explicit characterization of the local objects.}
\end{proposition}

\begin{proof}
    Because $\TR$ is $t$-exact, we have a commutative diagram
    $$\xymatrix{
    \cdots\ar[r]&(\CycSp_p)_{\geq n}\ar[r]^{\tau_{\geq
    n+1}}\ar[d]^{\TR}&(\CycSp_p)_{{\geq n+1}}\ar[r]\ar[d]^{\TR}&\cdots\\
    \cdots\ar[r]&(\TCart_p)_{\geq n}\ar[r]^{\tau_{\geq n+1}}&(\TCart_p)_{\geq n+1}\ar[r]&\cdots
    }$$ of $\infty$-categories, where the vertical maps are fully faithul with essential image consisting of $n$-connective topological Cartier modules that are $V$-complete. By Propositon \ref{prop:vcompletemodule} these are equivalently those $n$-connective topological Cartier modules whose homotopy groups are derived $V$-complete. 
    Taking the limit, as
    $n\rightarrow-\infty$, we obtain a fully faithful embedding
    $$\Sp((\CycSp_p)_{\geq 0})\rightarrow\TCart_p,$$
    where $\Sp((\CycSp_p)_{\geq 0})$ is the right completion of $\CycSp_p$ with
    respect to the cyclotomic $t$-structure. (Recall from
    Proposition~\ref{prop:tdt} that $\TCart_p$ is already right complete.)
    This inclusion functor preserves limits and is accessible, and thus admits
    a left adjoint. Moreover the essential image consists of those topological Cartier modules all of whose homotopy groups are derived $V$-complete.
\end{proof}

\subsection{$p$-typical Cartier complexes}\label{sub:cc}

Let $M$ be a $p$-typical topological Cartier module.
We have already seen in Section \ref{exhomotopygroup} that the homotopy groups
$\pi_*M$ admit the structure of $p$-typical Cartier modules, i.e. for every
$k \in \ZZ$ the induced maps 
\[
V: \pi_kM \to \pi_k M \qquad \text{and} \qquad F: \pi_k M \to \pi_k M
\]
satisfy $FV = p$. Moreover from the $S^1$-action on $M$ we get a `Connes operator'
\[
d: \pi_k M \to \pi_{k+1} M
\]
which satisfies $d^2= \eta d = d \eta$ where $\eta$ here denotes the map 
\[
\eta: \pi_kM \to \pi_{k+1}M
\]
given by acting with the Hopf element $\eta \in \pi_1(\SS)$. In particular the map $\eta$ is $2$-torsion and $\eta^4 = 0$.  
To see this we note that the map $d$ is defined by acting with an element
$d$ in $\pi_1$ of the spherical group ring $\SS[S^1]$ which is given by the
fundamental class in $\pi_1(S^1)$ shifted to the basepoint $0$ in $\SS[S^1]$.
Then the claim follows from the fact that
\[
\pi_*(\SS[S^1]) = (\pi_*\SS) [d] /(d^2 = \eta d)
\]
which is implied by the determination of stable homotopy class of the
multiplication map $S^1 \times S^1 \to S^1$. This map is after a single
suspension given by the map
\[
S^2 \vee S^2 \vee S^3 \to S^2
\]
which is the inclusions on the first two summands and the Hopf map on the last.
See also \cite{hesselholt-madsen-drw} for a discussion.

\begin{lemma}\label{lem:cartiercomplexes}
For a topological Cartier module $M$ we have on $\pi_*M$ the relations
\begin{align*}
 Vd = p dV, \qquad dF = p Fd, \qquad FdV = \begin{cases}
  d  & \text{for } p >2, \\
  d + \eta & \text{for } p =2
  \end{cases}
\end{align*}
and the maps $F,V$ commute with $\eta$. 
\end{lemma}
\begin{proof}
The $S^1$-equivariant map $V: M_{hC_p} \to M$ can equivalently be considered as
an $S^1$-equivariant map $M \to \mathrm{res}_p M$ where $\mathrm{res}_p M$ has
the $S^1$-action given by restricting the $S^1$-action on $M$ along the
$p$-fold cover map $S^1 \to S^1$ given by $z \mapsto z^p$. Thus we get that on
homotopy groups $V(d(x)) = d' V(x)$ where $d': \pi_k M \to \pi_{k+1} M$ is the
Connes operator associated with the $S^1$-action on $\mathrm{res}_p M$. But
$d' = pd$ since the $p$-fold cover map is of degree $p$. Hence, we see
that \[
Vd = d'V = p d V .
\]
The Frobenius can dually be considered as a map $\mathrm{res}_p M \to M$ and we therefore find
\[
d F  = F d'  =  F pd  = p F d .
\]
To identify  $FdV: \pi_*M \to \pi_{*+1}M$ we note that we can write this as a composite
\[
    \pi_*M\xto{\iota} \pi_*(M_{hC_p}) \xto{V_*} \pi_*M \xto{d} \pi_{*+1}M
    \xto{F_*} \pi_{*+1}(M^{hC_p}) \xto{\kappa} \pi_{*+1}M,
\]
where $\iota$ and $\kappa$ are the structure maps and the maps $V_*$ and $F_*$
are the maps induced by the spectral maps $V$ and $F$ on homotopy groups. The
operator called abusively $V$ above was defined as the composite $V_* \circ
\iota$ and similar for $F$. Since $V$ is $S^1$-equivariant we find that the
composite of the maps is equal to
\[
 \kappa \circ d' \circ F_* \circ V_* \circ \iota = \kappa \circ d' \circ \Nm_{C_p} \circ \iota \ ,
\]
where $d'$ is the Connes operator for the residual $S^1$-action on
$M_{hC_p}$.\footnote{This map is also equal to $\kappa \circ \Nm_{C_p} \circ d'
\circ \iota$ for $d'$ corresponding to the residual $S^1$-action on $M^{hC_p}$
but we shall not need this fact. } Now we claim that for every $S^1$-spectrum
$M$ (not necessarily a $p$-typical topological Cartier module) the composite
$\kappa \circ d' \circ \Nm_{C_p} \circ \iota $ is given by $d: \pi_*M \to
\pi_{*+1}M$ for $p$ odd and by $d + \eta: \pi_*M \to \pi_{*+1}M$ for $p =2$. It
is enough to show this claim for $M$ the free $S^1$-spectrum on a generator in
degree $0$, which is to say for $M = \Sigma^\infty_+ S^1$ and the class $x \in
\pi_0(M)$ induced by the basepoint $1 \in S^1$. We have identifications
\[
M \we M_{hC_p} \we M^{hC_p} \we \SS \oplus \Sigma \SS
\]
such that under these identifications the map $\iota: M \to M_{hC_p}$ is given
by the map $\id \oplus p : \SS \oplus \Sigma \SS \to \SS  \oplus \Sigma \SS$,
the norm is given by the identity and the map
$\kappa: M^{hC_p} \to M$ is given by the map $p \oplus \id: \SS \to \SS$ for $p$ odd and by the map 
\[
\begin{pmatrix}
p & \eta \\
0 & \id 
\end{pmatrix}  : \SS \oplus \Sigma \SS \to \SS \oplus \Sigma \SS
\]
for $p =2$ (see the proof of Lemma \cite[Lemma 1.5.1]{hesselholt-ptypical} for.an argument). Under these identifications the operator $d'$ takes $1 \in \SS
\oplus \Sigma \SS = M_{hC_p}$ to the unit element in $\pi_1$ of $\Sigma \SS$.
Then the claim follows by a straightforward computation.
The commutativity of the maps $F$ and $V$ with $\eta$ is clear since these are stable maps. 
\end{proof}

Now, we make this structure into a definition which is inspired by the
definition of a Dieudonn\'e complex in \cite{blm1}.

\begin{definition}\label{def_Cartier}
A $p$-typical Cartier complex is a $\ZZ$-graded abelian group $C^*$ together with operators
\[
    V, F: C^* \to C^* \qquad \text{and}\quad \eta, d: C^* \to C^{* +1} 
\] 
satisfying
\begin{align*}
&FV = p, \qquad d^2= \eta d = d \eta, \qquad 2 \eta = \eta^4 = 0,  \\
& Vd = p dV, \qquad dF = p Fd, \qquad  FdV = \begin{cases}
  d  & \text{for } p >2, \\
  d + \eta & \text{for } p =2.
  \end{cases}
\end{align*}
\end{definition}

\begin{remark}
    Hesselholt--Madsen~\cite{hesselholt-madsen-drw} and 
    Hesselholt~\cite{hesselholt-big}*{Definition~4.1} introduced the notion of a Witt
    complex. The universal example is the absolute de Rham--Witt complex. These
    Witt complexes give examples of $p$-typical Cartier complexes in the sense
    of Definition~\ref{def_Cartier} with $\eta$ given by multiplication with
    $d\log[-1] = [-1] \cdot d[-1]$. In fact Witt complexes should be considered
    as a multiplicative version of Cartier complexes (with some additional
    structure like a map from the Witt vectors). We will study the precise
    relation in future work.
\end{remark}

Lemma~\ref{lem:cartiercomplexes} shows that the homotopy groups $C^* := \pi_*M$ of a
$p$-typical topological Cartier module $M$ naturally form a $p$-typical Cartier
complex. The statement of Proposition \ref{prop:vcompletemodule} is that for
bounded below $M$ the $V$-completeness of $M$ is equivalent to the derived
$V$-completeness of the terms $\pi_*M$. In particular, the completeness of $M$ can be
entirely expressed in terms of the associated $p$-typical Cartier complex. More generally
one can ask that all the terms $C^i$ be
derived $V$-complete for a general Cartier complex $C^*$. It turns out in
practice that this condition is somewhat hard to verify since the $V$-adic
filtration does not take the differential $d$ into account. For
example, in the case of the de Rham--Witt complex which will be treated later, the
quotients by iterations of $V$ will not be the truncated de Rham--Witt complexes.
We will now give a completeness condition for a $p$-typical Cartier complex that is
equivalent to degreewise derived $V$-completeness but closer to the notions of
completeness that arise for the de Rham--Witt complex.

\begin{construction}\label{constr_completion}
Let $C^*$ be a $p$-typical Cartier complex. For every $i \in \ZZ$ and $r \geq 0$ we define
a derived quotient $C^{i} / (V^r + d V^r)$ in $\Dscr(\ZZ)$ as the total cofiber
of the commutative square
\[
\xymatrix{
C^{i-1} \ar[r]^{p^r} \ar[d]^d & C^{i-1} \ar[d]^{dV^r}\\
C^i \ar[r]^{V^r} & C^i.
}
\]
This total cofiber is by definition the cofiber of the map $V^r + dV^r :
C^{i}\oplus_{C^{i-1}} C^{i-1} \to C^i$ (where the
source is the derived pushout) and this justifies  the notation $C^{i} / (V^r
+ d V^r)$.

There is a natural diagram
\[
 \qquad \ldots \to C^{i} / (V^3 + d V^3) \to C^{i} / (V^2 + d V^2) \to C^{i} / (V + d V)
\]
where the map $C^{i} / (V^{r+1} + d V^{r+1}) \to C^{i} / (V^r + d V^r)$ is 
induced from the map of squares given by the identity  on the lower right term, by $pV$ on the upper left term  and by $V$ on the other two terms. 
We will denote the (derived) limit of this diagram as $(C^i)^{\wedge}_{V + dV}
\in \Dscr(\ZZ)$. There is a canonical map
\[
C^i \to (C^i)^{\wedge}_{V + dV}
\]
induced from the structure map out of lower right corner of the defining square
for $C^{i} / (V^r + d V^r)$.
\end{construction}

\begin{definition}
    A Cartier complex $C^*$ is called derived {\bf $(V+dV)$-complete} if for every $i$ the map
\[
C^i \to (C^i)^{\wedge}_{V + dV}
\]
from Construction \ref{constr_completion} is an equivalence in $\Dscr(\ZZ)$. 
\end{definition}

Now we show that for a bounded below Cartier complex this notion of
$(V+dV)$-completeness is equivalent to the naive notion of degreewise derived
$V$-completeness. The latter means that for every $i$ the Cartier module $C^i$ is derived complete in the sense of Definition \ref{Cartier_module_complete}.
Bounded below means that there exists $i_0$ such that $C^i = 0$ for $i  < i_0$ (we hope that the cohomological notation does not lead to confusion).

\begin{proposition}\label{prop:derivedcanonical}
    A bounded below Cartier module $C^*$ is derived $(V+dV)$-complete precisely
    if it is degreewise derived $V$-complete. 
\end{proposition}

\begin{proof}
We can assume by shifting that $C^i = 0$ for $i <0$. For $C^0$ we then obtain that
$C^0 / (V^r + dV^r) = C^0 / V^r$ and that $ (C^0)^{\wedge}_{V + dV} =
(C^0)^\wedge_V$ where the latter is the derived $V$-completion. It follows that
$C^0$ is derived $(V+dV)$-complete precisely if it is derived $V$-complete.

Now we proceed by induction over $i$. Assume that $C^{i-1}$ is derived $V$ and
$(V + dV)$-complete. We will show that $C^i$ is derived $(V+dV)$-complete
precisely if it is derived $V$-complete. By definition $(C^i)^\wedge_{V + dV}$
is the limit of total cofibers of squares
\[
\xymatrix{
C^{i-1} \ar[r]^{p^r} \ar[d]^d & C^{i-1} \ar[d]^{dV^r}\\
C^i \ar[r]^{V^r} & C^i.
}
\]
The limit is taken over maps of total cofibers obtained from maps of squares.
Thus we can equivalently described $(C^i)^\wedge_{V + dV}$  as the total
cofiber of the square obtained as the limit of these maps of squares. This
limit square is
$$\xymatrix{
\lim(\cdots \xto{pV} C^{i-1}  \xto{pV} C^{i-1} \xto{pV} C^{i-1})\ar[r]\ar[d]&
\lim(\cdots \xto{V} C^{i-1}  \xto{V} C^{i-1} \xto{V} C^{i-1})\ar[d]\\
\lim(\cdots \xto{V} C^{i}  \xto{V} C^{i} \xto{V} C^{i})\ar[r]&
C^i.}$$
The upper right term vanishes by the assumption that $C^{i-1}$ is derived $V$-complete;
the upper left term vanishes since the inverse limit can be written
equivalently as the limit of 
\[
    \ldots \xto{p} \lim_{i,V} C^{i-1} \xto{p} \lim_{i,V} C^{i-1} \xto{p}
    \lim_{i,V} C^{i-1}
\] 
where $\lim_{i,V} C^{i-1}$ is the term in the upper right corner, which we have
already noted is zero. In other words the limit square takes the form
\[
\xymatrix{
0 \ar[r]\ar[d] & 0 \ar[d] \\
\lim_{i,V} C^i \ar[r] & C^i.
}
\]
From this description it follows that the map from $C^i$ to the total cofiber
of this square is an equivalence precisely if the lower left corner vanishes,
i.e. if $C^i$ is derived $V$-complete. This finishes the proof.
\end{proof}

\section{The symmetric monoidal structure}\label{sec:monoidal}

According to Corollary \ref{cor_symmheart} we get an induced symmetric monoidal structure on
$\CycSp_p^\heart$ from the tensor product of cyclotomic spectra. In this section we shall
give an explicit formula for this symmetric monoidal structure and explore some
consequences. We will come back to this in later work.

\subsection{The tensor product of topological Cartier modules}

In this section we shall describe the symmetric monoidal structure on $\TCartVm$ induced
through the equivalence $\TCartVm\simeq \CycSp_p^-$ from the one on cyclotomic spectra. 
We denote the tensor product corresponding to the symmetric monoidal structure on $\TCartVm$ by
$\widehat \boxtimes$. More precisely we have that
\[
M \widehat{\boxtimes} N := \TR( M/V \otimes N/V),
\]
where $\otimes$ is the tensor product of cyclotomic spectra whose underlying spectrum is just
the tensor product of spectra. In this section we will follow the convention that $\otimes$
shall always refer to symmetric monoidal structures that are taken `underlying' in this
sense.
Our first task is to give a more explicit formula for $M \widehat \boxtimes N$ and then we
will identify the induced symmetric monoidal structure on the heart explicitly. In future
work we will show that this symmetric monoidal structure on $(\TCart^-_p)^\wedge_V$ is
induced by a natural  symmetric monoidal structure on $\TCart_p$ which will be denote by
$\boxtimes$ and use this to understand the relation to de Rham--Witt complexes that will be
explained in Section \ref{sec:schemes}  over general bases.

Recall the $\infty$-category $\CycSp_p^\Fr$ of cyclotomic spectra with Frobenius lift. There are natural forgetful functors
\[
\TCart_p \to \CycSp^\Fr_p \to \CycSp_p
\]
where the first functor just forgets the $V$ operator and the second sends $(X,\psi_p)$ to
the cyclotomic spectrum $X$ equipped with the Frobenius $X \to X^{hC_p} \xto{\can}
X^{tC_p}$.
As proven in~\cite{krause-nikolaus}*{Proposition~10.3}, the second functor
admits a right adjoint given by $\TR$, which, as we have proved above, factors canonically
through the first category. The first
functor $\TCart_p \to \CycSp^\Fr_p$ admits both adjoints by the adjoint functor theorem, but
we will only be concerned with the left adjoint here.

\begin{lemma}\label{lem_leftadj}
    The forgetful functor 
    \[
    \TCart_p \to \CycSp^\Fr_p
    \]
    is monadic. The underlying $S^1$-spectrum of the left adjoint applied to $M$ in
    $\CycSp^\Fr_p$ is given by 
    $M[V] := \bigoplus_{n \geq 0} M_{hC_{p^n}}$ with the structure of a topological Cartier module that will be described in the proof. 
\end{lemma}

\begin{proof}
    The forgetful functor reflects equivalences and preserves all colimits and limits. Thus, it
    follows from Lurie's version of the monadicity theorem that it is monadic
    (see~\cite{ha}*{4.7.0.3}).
    To understand the left adjoint functor we give a construction of an object $M[V] \in
    \TCart_p$. As a spectrum with $S^1$-action we set
    \[
    M[V] := \bigoplus_{n \geq 0} M_{hC_{p^n}}
    \]
    where $M_{hC_{p^n}}$ carries the residual $S^1/C_{p^n} \cong S^1$-action. We now want to
    equip $M[V]$ with the structure of a $p$-typical topological Cartier module. To this end, we define the $V$-operator as the inclusion
    \[
    M[V]_{hC_p} \we \bigoplus_{n \geq 1} M_{hC_{p^n}} \to \bigoplus_{n \geq 0} M_{hC_{p^n}}  = M[V] 
    \]
    and the $F$-operator as the composition
    \[
    \bigoplus_{n \geq 0} M_{hC_{p^n}}  \to  \bigoplus_{n \geq 0} \left(M_{hC_{p^n}}\right) ^{hC_p} \to  \left(\bigoplus_{n \geq 0} M_{hC_{p^n}}\right) ^{hC_p}  = M[V]^{hC_p}
    \]
    where the second map is the canonical interchange map and the first is given on
    the $(n=0)$-summand by $F$ and on the $n$th summand by the $C_p$-norm map
    \[
    M_{hC_{p^n}} \we (M_{hC_{p^{n-1}}})_{hC_p} \to (M_{hC_{p^{n-1}}})^{hC_p} 
    \]
    followed by the inclusion into the direct sum. There is a canonical equivalence
    between the composition $F \circ V$ and the norm map, as the latter also
    factors as the composition
    \[
    M[V]_{hC_p} \we  \bigoplus_{n \geq 1} M_{hC_{p^n}} \xto{\bigoplus \Nm_{C_p}} \bigoplus_{n \geq 1} (M_{hC_{p^{n-1}}})^{hC_p} \to \left(\bigoplus_{n \geq 1} M_{hC_{p^{n-1}}} \right)^{hC_p}.
    \]
    Now we compute the mapping spectrum $\MapSp_{\TCart_p}(M[V], N)$ and show
    that the canonical map to $\MapSp_{\CycSp^\Fr_p}(M, N)$, induced by the forgetful functor
    $\TCart_p\rightarrow\CycSp^\Fr_p$ and the map $M\rightarrow M[V]$ of
    cyclotomic spectra with Frobenius lifts, is an equivalence. Here we abusively denote the `underlying' cyclotomic spectrum with Frobenius lift of $N$ also by $N$. To this end we use the formula for the mapping space in $\TCart_p$ given in Proposition \ref{prop_algebra} and get a fibre sequence
    \begin{equation}\label{fibrecoalgebra}
    \MapSp_{\TCart_p}(M[V], N) \to \MapSp_{\Alg_{(-)_{hC_p}}}(M[V], N) \xto{\vartheta} \MapSp_{\Sp^{BS^1}}\left(\cofib{(V_M)}, N^{hC_p} \right) \ .
    \end{equation}
    Note that $\cofib{(V_M)} = M$. Moreover, it is straightforward to check
    that  $\MapSp_{\Alg_{(-)_{hC_p}}}(M[V], N) \simeq \MapSp_{\Sp^{BS^1}}(M, N)$. In fact $M[V]$ is by construction free as a $(-)_{hC_p}$-algebra if we neglect the $F$-operators. Under these identifications the fibre sequence \eqref{fibrecoalgebra} takes the form
    \[
    \MapSp_{\TCart_p}(M[V], N) \to \MapSp_{\Sp^{BS^1}}(M, N) \xto{\vartheta} \MapSp_{\Sp^{BS^1}}\left(M, N^{hC_p} \right) 
    \]
    with the right hand map $\vartheta$ is given by sending $g$ to $F_N \circ g - g^{hC_p} \circ F_M$. But this is also the mapping spectrum in $\CycSp^\Fr$.
\end{proof}

\begin{remark}
For any spectrum with $S^1$-action $M$ the spectrum $M[V] = \bigoplus_{n \geq
0} M_{hC_{p^n}}$ of the last lemma is free as a spectrum with $V$ operator.
The lemma shows that if $M$ admits an $F$-operator then $M[V]$ is a topological
Cartier module and is also free as a topological Cartier module. One can also
express this by saying that the commutative square of forgetful functors
\[
\xymatrix{
\TCart_p \ar[r] \ar[d] & \Alg_{(-)^{hC_p}} \ar[d] \\
\Alg_{(-)_{hC_p}}\ar[r] & \Sp^{BS^1}
}
\]
remains commutative after passing to left adjoints of the horizontal maps, i.e., the square is left adjointable.
\end{remark}

As usual we denote by $(\CycSp^\Fr_p)^- \subseteq \CycSp_p^\Fr $ the full subcategory of bounded below objects. 
The adjunction of Lemma \ref{lem_leftadj} induces an adjunction
\begin{equation}\label{adjucompl}
\xymatrix{
(\CycSp^\Fr_p)^-\ar[r]<2pt> &  \ar[l]<2pt> \TCartVm
}
\end{equation}
whose left adjoint sends $M$ to the $V$-completion of $\bigoplus_{n \geq 0} M_{hC_{p^n}}$
which is  $\prod_{n \geq 0} M_{hC_{p^n}}$ and will be denoted by $M[[V]]$. We note that this
is also equivalent to $\TR(M)$ where $M$ is considered as a cyclotomic spectrum with the
Frobenius $M \xto{F} M^{hC_p} \xto{\can} M^{tC_p}$ (since this is the mod $V$ reduction).

We now use that the $\infty$-category $\CycSp^\Fr_p$ has a symmetric monoidal structure
given by the `underlying' tensor product. Formally this symmetric monoidal structure is
constructed exactly as the one on $\CycSp_p$ in \cite[Construction IV.2.1]{nikolaus-scholze}
using that the functor $(-)^{hC_p}: \Sp^{BS^1} \to \Sp^{BS^1}$ admits a canonical lax
symmetric monoidal structure. The symmetric monoidal structure on $\CycSp^\Fr_p$ restricts
to one on the full subcategory $(\CycSp^\Fr_p)^- \subseteq \CycSp^\Fr_p$ since $X \otimes Y$
is bounded below for $X$ and $Y$ bounded below.

\begin{proposition}\label{symmon}
    The left adjoint $-[[V]]: (\CycSp^\Fr_p)^- \to \TCartVm$
    admits a canonical refinement to a symmetric monoidal functor.
\end{proposition}

\begin{proof}
    By definition, the functor $(-)/V:  \TCartVm \to \CycSp_p^-$ is a symmetric
    monoidal equivalence. Thus, it suffices to equip the composition
    \[
        (\CycSp^\Fr_p)^- \xto{-[[V]]} \TCartVm \xto{(-)/V} \CycSp_p^-
    \]
    with a symmetric monoidal structure. We claim that this composite is equivalent to the functor which takes 
    $X \in (\CycSp^\Fr_p)^-$ to the `underlying' cyclotomic spectrum of $X$, i.e. $X$ equipped with the 
    composition 
    $X \to X^{hC_p} \xto{\can} X^{tC_p}$ as Frobenius. To see this we have to compute the cofiber of the map
    \[
    V: \left(\prod_{n \geq 0} M_{hC_{p^n}}\right)_{hC_p} \to \prod_{n \geq 0} M_{hC_{p^n}} \ .
    \]
    Since the product is uniformly bounded below it commutes with the orbits and the claim
    follows from the description of the $V$ and $F$ operators given in the proof of Lemma
    \ref{lem_leftadj} above.
    Finally the functor $\CycSp^\Fr_p \to \CycSp_p$ admits by construction of the symmetric
    monoidal structures a symmetric  monoidal refinement since the transformation
    $(-)^{hC_p} \to (-)^{tC_p}$ is a symmetric monoidal transformation.
\end{proof}

We now observe that for every $p$-typical topological Cartier module there is a
natural cofiber sequence
\begin{equation}\label{coffree}
(M_{hC_p})[V] \to M[V] \to M
\end{equation}
of $p$-typical topological Cartier modules. Here the $S^1$-spectrum $M_{hC_p}$ is considered
as a cyclotomic spectrum with Frobenius lift, where the Frobenius lift is given by the zero
map. Then the map $(M_{hC_p})[V] \to M[V]$ is induced from the map $M_{hC_p} \to M[V]$ that
is given as the composition
\[
M_{hC_p} \xto{(\id, -V)} M_{hC_p} \oplus M \xto{i}  \oplus_{n \geq 0} M_{hC_{p^n}} = M[V]
\]
where $i$ is the summand inclusion. This is a map in $\CycSp^\Fr$ and therefore gives rise
to a map $(M_{hC_p})[V] \to M[V]$. Concretely this map is given by $i - V$ where $i$ is the
inclusion $(M_{hC_p})[V] \to M[V]$ and $V = \oplus_{n \geq 0} V_{hC_{p^n}}$ is just applied
levelwise.
It is easy to see that the composite to $M$ comes with a preferred nullhomotopy
(as we also only have to check that on $M_{hC_p}$) and that on underlying
spectra this is a cofiber sequence. This implies that it is a cofiber sequence
of $p$-typical topological Cartier modules.\footnote{Note that this gives rise to our
    standard cofiber sequence $M_{hC_p} \to M \to M/V$ of cyclotomic spectra upon taking the
        mod $V$ reductions (again $M_{hC_p}$ has the trivial Frobenius). But while the first
        map is a map of cyclotomic spectra between cyclotomic spectra that admit Frobenius
        lifts, it is not a map of cyclotomic spectra with Frobenius lifts. That is important
        to keep in mind in identifying some of the maps later.}

If $M$ is bounded below and $V$-complete, we also get a cofiber sequence
\begin{equation}\label{coff2}
(M_{hC_p})[[V]] \to M[[V]] \to M
\end{equation}
in $\TCartVm$ by completion of~\eqref{coffree}.

\begin{corollary}\label{corollary_tensor}
    For every pair of $V$-complete bounded below $p$-typical topological
    Cartier modules $M$ and $N$  the tensor product $M \widehat\boxtimes N \in
    \TCartVm$ is equivalent to the total cofiber of a square
    \[
    \xymatrix{
    (M_{hC_p} \otimes N_{hC_p}) [[V]] \ar[d] \ar[r] & (M \otimes N_{hC_p}) [[V]] \ar[d]  \\
    (M_{hC_p} \otimes N) [[V]] \ar[r] & (M \otimes N) [[V]].
    }
    \]
\end{corollary}

\begin{proof}
We use that the symmetric monoidal structure $\widehat{\boxtimes}$ on
$\TCartVm$ commutes with colimits and the cofiber sequence~\eqref{coff2} for
$M$ and $N$ to deduce that the tensor product is the total cofiber of a square
\[
\xymatrix{
M_{hC_p}[[V]] \widehat\boxtimes N_{hC_p} [[V]] \ar[d] \ar[r] & M[[V]] \widehat\boxtimes N_{hC_p} [[V]] \ar[d]  \\
M_{hC_p}[[V]] \widehat\boxtimes N [[V]] \ar[r] & M[[V]] \widehat\boxtimes N [[V]].
}
\]
Now we use that $-[[V]]$ is symmetric monoidal as shown in Proposition \ref{symmon} to get the result.
(Note that we do not need to complete the total cofiber as it is already
$V$-complete since this is a finite colimit.)
\end{proof}

\begin{remark}\label{rem_maps}
One can also work out the maps in the diagram of Corollary
\eqref{corollary_tensor} by using the fact that this is a diagram in $\TR$ for the corresponding diagram 
\[
\xymatrix{
(M_{hC_p} \otimes N_{hC_p}) \ar[d]^{\id \otimes V} \ar[r]^{V \otimes \id} & (M \otimes N_{hC_p})\ar[d]^{\id \otimes V}  \\
(M_{hC_p} \otimes N) \ar[r]^{V \otimes \id} & (M \otimes N)
}
\]
of cyclotomic spectra. We again issue the warning that this is not a diagram
with Frobenius lifts. Therefore one has to use the identification of $\TR$ with
$-[[V]]$ in each term but the maps are not compatible with that identification and we get
additional terms coming from the Frobenius operators.
 \end{remark}
Now assume that $M$ and $N$ are connective. We want to determine $\pi_0( M \widehat \boxtimes N)$ in terms of $\pi_0(M)$ and $\pi_0(N)$. 
For $M$ connective we have that $\pi_0(M[V]) = \oplus_{n \geq 0} \pi_0(M)$.
Thus, on $\pi_0$, the effect of $-[V]$ is given by adjoining $V$ freely as a module.
We will write this formula as $\pi_0(M[V]) = (\pi_0M)[V]$, thus use the
notation $A[V]$ for an abelian group with a map $F: A \to A$ for the
$p$-typical Cartier module $\oplus_{n \geq 0} A$ formed analogously to the
topological case. We shall write a typical element in $A[V]$ as a polynomial in
$V$, i.e. as $\sum_{i=1}^n a_i V^i$. Similarly, we have that $\pi_0(M[[V]]) =
\pi_0(M)[[V]]$.

Now, observe that in the formula of Corollary \ref{corollary_tensor} all terms  are
connective if $M$ and $N$ are.  As a result we find that $\pi_0( M \widehat \boxtimes N) $
is the total cokernel of a square
\begin{equation}\label{coequalizer}
\xymatrix{
(\pi_0M \otimes \pi_0N)  [[V]] \ar[d] \ar[r] & (\pi_0M \otimes \pi_0N) [[V]] \ar[d]  \\
(\pi_0M \otimes \pi_0N ) [[V]] \ar[r] & (\pi_0M \otimes \pi_0N) [[V]].
}
\end{equation}
Of course the total cokernel does not depend on the upper left term (this is
why this is not a concept one ever hears of), so it is just the quotient of $(\pi_0M \otimes \pi_0N) [[V]]$ by the image of the two maps into it.

\begin{corollary}\label{tensor_heart}
For any pair of $V$-complete connective $p$-typical topological Cartier modules $M$ and $N$ we can describe $\pi_0(M \widehat\boxtimes N)$ as $\pi_0$ of the 
algebraic derived $V$-completion (see Lemma \ref{alg_completion} and the
following discussion for this notion) of a $p$-typical Cartier module 
\[
(\pi_0 M \otimes \pi_0 N)[V] / \sim
\]
where the equivalence relation is generated additively by
\[
(m \otimes Vn) V^k \sim (Fm \otimes n) V^{k+1} 
\qquad (Vm \otimes n) V^k \sim (m \otimes Fn) V^{k+1}
\]
for all $m \in \pi_0M$, $n \in \pi_0N$ and $k \in \mathbb{N}$. 
\end{corollary}

\begin{proof}
It is clear that the square  \eqref{coequalizer} is the $V$-completion of a square with
terms $(\pi_0M\otimes \pi_0N)[V]$ instead of  $(\pi_0M\otimes \pi_0N)[V]$. The rest follows
from the description of the maps in the square \eqref{coequalizer} which can be done as
explained in Corollary \ref{rem_maps} or just using that they have to be compatible with $F$
(where $F$ is zero on the terms with orbits). Concretely one gets the following description
of the maps:
\begin{enumerate}
\item
 the upper horizontal map sends $(m \otimes n) V^k$ to $(Vm \otimes n)V^k$;
 \item
 the left vertical map sends $(m \otimes n) V^k$ to $(m \otimes Vn)V^k$;
\item
the lower horizontal map sends $(m \otimes n) V^k$ to $(Vm \otimes n) V^k - (m \otimes Fn) V^{k+1}$;
\item
the right vertical map sends $(m \otimes n) V^k$ to $(m \otimes Vn) V^k - (Fm \otimes n) V^{k+1}$.
\end{enumerate}
This implies the claim.
\end{proof}

Corollary~\ref{tensor_heart} identifies the tensor product on the category of derived
$V$-complete $p$-typical Cartier modules induced from the symmetric monoidal
structure on $\CycSp_p$. We
shall describe this tensor product in the next section purely algebraically.

\subsection{The tensor product of classical Cartier modules}\label{sec:symmonDieu}

\newcommand{\DFp}{D_{\FF_p}}

Inspired by the results of the last section, we will construct
in this section a symmetric monoidal structure 
$\boxtimes$ on the abelian category $\Cart_p$ of $p$-typical Cartier modules. Our guiding
principle is that if $R$ is a commutative ring, then the ring of $p$-typical Witt vectors $W(R)$, with its usual
Frobenius and Verschiebung operations $F$ and $V$ defines  a commutative algebra object with respect to $\boxtimes$.
In particular we will get induced symmetric monoidal structures on modules over $W(R)$ (internal to the category $\Cart_p$). We will explicitly identify this category and in the case $R=\FF_p$ we will see that the category is equivalent to classical Dieudonn\'e modules. 

The induced symmetric monoidal structure $\boxtimes_{W(\FF_p)}$ on Dieudonn\'e modules was first defined by
Goerss in~\cite{goerss-hopf} and later studied by~\cite{buchstaber-lazarev}. Under the relation to genuine spectra and Mackey functors that will be explained in Section \ref{sec:genuine} our tensor product can be understood using the tensor product of Mackey functors. 

We thank Achim Krause for helpful discussion around this section and for the crucial idea in the proof of Proposition \ref{complete}.

\begin{definition}\label{def}
    Let $M,N,Q$ be $p$-typical Cartier modules. 
    A bilinear map of abelian groups $(-,-):M\times N\rightarrow Q$ will be
    called $(V,F)$-bilinear if it satisfies the relations
    \begin{gather*}
        F(x,y)=(F(x),F(y))\\
        V(x,F(y))=(V(x),y)\\
        V(F(x),y)=(x,V(y))
    \end{gather*}
    for any $x\in M$ and $y\in N$.
    We let $\Hom_{(V,F)}(M\times N,Q)$ denote the group of
    $(V,F)$-bilinear maps $M\times N\rightarrow Q$.
\end{definition}

\begin{example}
For a commutative ring $R$ the multiplication $W(R) \times W(R) \to W(R)$ is $(V,F)$-bilinear. 
\end{example}

For an abelian group  $M$ with a map $F: M \to M$ we will denote by $M[V]$ the
$p$-typical Cartier module $\oplus_{n \in \mathbb{N}} M$ with the $V$ and $F$
operators as in the last section. We will again adopt the notation that we
write an element as a polynomial in $V$, i.e. in the form $\sum a_i V^i$.

\begin{lemma}\label{lemexp}
    Given $p$-typical Cartier modules $M$ and $N$, the functor
    $\Hom_{(V,F)}(M\times N,-)$ from $p$-typical Cartier modules to abelian
    groups is corepresentable by a $p$-typical Cartier module $M\boxtimes
    N$. Explicitly we have that 
    \[
M \boxtimes N = (M \otimes N)[V] / \sim
\]
where the equivalence relation is generated additively by
\[
(m \otimes Vn) V^k \sim (Fm \otimes n) V^{k+1} 
\qquad (Vm \otimes n) V^k \sim (m \otimes Fn) V^{k+1}
\]
for all $m \in M$, $n \in N$ and $k \in \mathbb{N}$ with the $V$ and $F$ operators induced from the ones on $(M \otimes N)[V]$.
\end{lemma}

\begin{proof}
We want to define $M \boxtimes N $ by the above formula. First we have to show
that this is even a $p$-typical Cartier module, i.e. that the $V$ and $F$ operators descend to the quotient. For $V$ this is obvious, since 
\[
    V((m \otimes Vn) V^k ) = (m \otimes Vn) V^{k+1}.
\]
The operator $F$ is given on additive generators by
\[
F((m \otimes n) V^k ) = \begin{cases}
(Fm \otimes Fn) V^0 & \text{for } k = 0 \\
p (m \otimes n) V^{k-1} & \text{for } k > 0 
\end{cases}
\]
Thus it follows that 
\[
F((m \otimes Vn) V^k ) \sim F( (Fm \otimes n) V^{k+1})
\]
and similarly for the other relation.  Finally we have to understand maps $M
\boxtimes N \to Q$ in the category of $p$-typical Cartier modules. By
construction these are just maps $\beta: (M \otimes N) [V] \to Q$ of
$p$-typical Cartier modules which satisfy the relations
\[
\beta((m \otimes Vn) V^k) = \beta((Fm \otimes n) V^{k+1} ) \qquad \beta((Vm \otimes n) V^k ) = \beta((m \otimes Fn) V^{k+1}) \ .
\]
Being $V$-linear, every such map $\beta$ is determined by its restriction $M \otimes N \to M\otimes N[V] \to  Q$. This restriction precisely satisfies the relation of Definition \ref{def}. Vice verse every map satisfying the relations of Definition \ref{def} can be extended to such a map $\beta$ in a unique way. 
\end{proof}
\begin{remark}\label{rem_cokernel}
We can also write this tensor product also the cokernel  of the map of
$p$-typical Cartier modules
\[
(M \otimes N)[V] \oplus (M \otimes N) [V] \to (M\otimes N)[V]
\]
given by  assembling together the maps $(M \otimes N)[V] \to (M \otimes N)[V]$ given by
\begin{align*}
(m \otimes n) V^k  &\mapsto (Vm \otimes n) V^k - (m \otimes Fn) V^{k+1},\\
(m \otimes n) V^k &\mapsto (m \otimes Vn) V^k - (Fm \otimes n) V^{k+1}.
\end{align*}
Here, the summands $(M \otimes N)[V]$ in the source of the map are $p$-typical Cartier modules with Frobenius that is zero in $V$-degree $0$. 
Note that this cokernel description is basically reversing the line of thought in the proof of Corollary \ref{tensor_heart}.
\end{remark}
\newcommand{\rat}{\mathrm{rat}}

\begin{proposition}
    The assignment $(M,N)\mapsto M\boxtimes N$ defines a symmetric monoidal
    structure on $p$-typical Cartier modules which is compatible with small colimits in
    each variable and with unit the $p$-typical Cartier module $\ZZ[V]
    \subseteq W(\ZZ)$ which is given by
    \[
    \ZZ[V] = \bigoplus_{n \geq 0} \ZZ \cdot V^n(1) \subseteq \prod_{n \geq 0}
    \ZZ \cdot V^n(1) = \ZZ[[V]] \iso W(\ZZ),
    \]
    a subring of $W(\ZZ)$.\footnote{There is another similar subring of
    $W(\ZZ)$, namely the rational Witt vectors $W_\rat(\ZZ)$. There are
    canonical inclusions $\ZZ[V] \subseteq W_\rat(\ZZ) \subseteq W(\ZZ)$ both of which are proper. }
\end{proposition}

\begin{proof}
    The proof is the same as the proof that the tensor product on abelian
    groups defines a symmetric monoidal structure. Since the bilinear
    conditions are symmetric, the result is symmetric. The claim about colimits
    is clear because $\Hom_{(V,F)}(M\times N,-)$ takes colimits in $N$ to
    limits in the functor category, and hence colimits of corepresentables.
    Now, we claim that $\ZZ[V]\boxtimes M\iso M\iso M\boxtimes \ZZ[V]$.
    For this, by symmetry,
    it is enough to construct a natural isomorphism $\Hom_{(V,F)}(M\times
    \ZZ[V],-)\iso\Hom_{\Cart_p}(M,-)$. Given a $(V,F)$-bilinear pairing $(-,-):M\times
    \ZZ[V]\rightarrow Q$, we let $f:M\rightarrow N$ be given by
    $f(x)=(x,1)$. Then, $$f(F(x))=(F(x),1)=(F(x),F_W(1))=F(x,1)=F(f(x))$$ and
    $$f(V(x))=(V(x),1)=V(x,F_W(1))=V(x,1)=V(f(x)),$$ so that $f$ is a map of
    $p$-typical Cartier modules. Given a map $f:M\rightarrow N$ of $p$-typical Cartier modules we let
    $(x,y)_f:M\times \ZZ[V] \rightarrow N$ be defined by $$(x,y)_f=yf(x).$$
    This is clearly bilinear. We must check that it is $(V,F)$-bilinear. To
    that end, we have
    \begin{gather*}
        F(x,y)_f=F(y(f(x)))=F_W(y)F(f(x))=F_W(y)f(F(x))=(F(x),F_W(y))_f,\\
        V(x,F_W(y))_f=V(F_W(y)f(x))=yV(f(x))=yf(V(x))=(V(x),y)_f,\\
        V(F(x),y)_f=V(yf(F(x)))=V(yF(f(x)))=V_W(y)f(x)=(x,V_W(y))_f,
    \end{gather*}
    which is what we needed to show.
    These operations are mutually inverse. Associativity isomorphisms are
    constructed by observing that $(M\boxtimes N)\boxtimes P$ corepresents
    $(V,F)$-bilinear morphisms $(M\boxtimes N)\times P\rightarrow Q$ or
    $(V,F)$-multilinear morphisms $(M\times N)\times P\rightarrow Q$. These are
    multilinear maps $(-,-,-)$ satisfying
    \begin{gather*}
        F(x,y,z)=(F(x),F(y),F(z)),\\
        V(x,F(y),F(z))=(V(x),y,z),\\
        V(F(x),y,F(z))=(x,V(y),z),\\
        V(F(x),F(y),z)=(x,y,V(z)).
    \end{gather*}
    But, it is easy to check that these are precisely the same relations
    satisfied by the bilinear maps classified by maps $M\boxtimes
    (N\boxtimes P)\rightarrow Q$. The unit, pentagon, and
    hexagon axioms will be left to the reader.
\end{proof}

\begin{lemma}\label{lemmodVsymm}
    The functor $\coker(V): \Cart_p \to \Ab$ admits a symmetric monoidal structure.
\end{lemma}

\begin{proof}
    The abelian group $\coker(M \boxtimes N\xrightarrow{V}M\boxtimes N)$ is obtained from $M \boxtimes N$ by an equivalence relation. But 
    according to Lemma \ref{lemexp} the group $M \boxtimes N$ is itself
    obtained by quotienting $(M \otimes N)[V]$ by an equivalence relation. Thus
    we obtain that  $\coker(M \boxtimes N\xrightarrow{V}M\boxtimes N)$ is the quotient of $(M \otimes N)[V]$
    obtained by the combined equivalence relation which is generated additively
    by
    \[
    (m \otimes Vn) V^k \sim (Fm \otimes n) V^{k+1} 
    \qquad (Vm \otimes n) V^k \sim (m \otimes Fn) V^{k+1}
    \qquad (m \otimes n) V^n \sim 0  \text{ for } n \geq 1
    \]
    In view of the third relation the first two relations are equivalent to 
    \[
    m \otimes Vn \simeq 0, \qquad Vm \otimes n \simeq 0.
    \]
    Therefore we just obtain the quotient of $M \otimes N$ by these relations.
    But this is $\coker(M\xrightarrow{V}M) \otimes \coker(N\xrightarrow{V}N)$.
    Together with the isomorphism
    \[
        \coker(\ZZ[V]\xrightarrow{V}\ZZ[V]) \iso \ZZ,
    \]
    this gives the functor $\coker(V)\colon\Cart_p \to \Ab$ the desired symmetric monoidal structure.
\end{proof}

We shall also need a derived version of Lemma \ref{lemmodVsymm}. To this end we
introduce a `derived' version of $\boxtimes$ which we denote $\boxtimes^\L$. This is just defined as the total cofiber of the diagram
\begin{align}\label{square_derived}
\xymatrix{
(M \otimes^\L N)[V] \ar[r]\ar[d] & (M \otimes^\L N)[V] \ar[d] \\
(M \otimes^\L N)[V] \ar[r] & (M \otimes^\L N)[V]
}
\end{align} 
where $(-)[V]$ is as before just the infinite product of $M \otimes^\L N$ (not
homotopy orbits) and the maps are as in the proof of Corollary
\ref{tensor_heart}. This is an object in the derived category. The ordinary
tensor product $\boxtimes$ is $\H_0$ of $\boxtimes^\L$. \footnote{We note that the notation $\boxtimes^\L$ is meant to indicate that $\boxtimes^\L$ is really a tensor product on the derived category. We do not claim that it is actually the derived functor of $\boxtimes$ in any sense and hope that this notation does not lead to confusion.}

\begin{lemma}\label{reduction}
For every pair of $p$-typical Cartier modules $M$ and $N$ we have an equivalence
\[
    \cofib(M \boxtimes^\L N\xrightarrow{V}M\boxtimes^\L N) \simeq
    \cofib(M\xrightarrow{V}M) \otimes^\L \cofib(N\xrightarrow{V}N)
\]
in $\Dscr(\ZZ)$.
\end{lemma}

\begin{proof}
We can interchange the mod $V$ reduction with taking the total cofiber in the square \eqref{square_derived}. Then the square is just of the form
\[
\xymatrix{
(M \otimes^\L N) \ar[r]^{V \otimes \id}\ar[d]^{\id \otimes V} & (M \otimes^\L N) \ar[d]^{\id \otimes V} \\
(M \otimes^\L N) \ar[r]^{V \otimes \id} & (M \otimes^\L N)
}
\]
whose total cofiber evidently has the claimed form.
\end{proof}

We now show that the tensor product of $p$-typical Cartier modules also induces
a tensor product of derived $V$-complete $p$-typical Cartier modules. We remind
the reader of the algebraic derived $V$-completion  $L: \Cart_p \to
\CartV$ discussed in Lemma~\ref{alg_completion} and afterwards.

\begin{proposition}\label{complete}
    The localization $L: \Cart_p \to  \CartV$  is compatible with the
    symmetric monoidal structure, that is we have an induced symmetric monoidal
    structure on $\CartV$, which we denote by $\widehat \boxtimes$ and
    which is given by
    \[
    M \widehat\boxtimes N := L (M \boxtimes N)
    \]
    and has unit $W(\ZZ)$. We can also describe $M \widehat\boxtimes N$ as the cokernel of the map
    \[
    (M \otimes N)[[V]] \oplus (M \otimes N) [[V]] \to (M\otimes N)[[V]],
    \]
    the completion of the map in Remark \ref{rem_cokernel}.
\end{proposition}

\begin{proof}
First it is clear that the cokernel of the map $(M \otimes N)[[V]] \oplus (M \otimes N) [[V]] \to (M\otimes N)[[V]]$ agrees with the completion of the tensor product by the fact that the corresponding statement is clear before completion (Remark \ref{rem_cokernel}) and the completion of a cokernel is the cokernel of the completions. 

By the usual criterion for symmetric monoidal localizations we have to verify
that for any pair of $p$-typical Cartier modules $M$ and $N$ the map
\[
L(M \boxtimes N) \to L(LM \boxtimes N) 
\]
is an isomorphism. 
We have a resolution $\ker(q) \to M[V] \xto{q} M$ of $M$ by $V$-torsion free $p$-typical Cartier modules. 
Since $\boxtimes$ and $L$ both commute with colimits (i.e. are right exact) we
can thereby reduce the question to the $V$-torsion free case. Thus we assume
that $M$ is $V$-torsion free. Then also $LM$ is $V$-torsion free by Lemma
\ref{torsionfree} and the map $M \to LM$ is a mod $V$ equivalence by Lemma
\ref{lemmodVequivalence}. Since the mod $V$ reduction is also the derived mod
$V$ reduction we get using Lemma~\ref{reduction} that
\begin{align*}
    \cofib(M \boxtimes^\L N\xrightarrow{V}M\boxtimes^\L N)
    &\simeq\cofib(M\xrightarrow{V}M) \otimes^\L \cofib(N\xrightarrow{V}N)\\
    &\simeq\cofib(LM\xrightarrow{V}LM)\otimes^\L\cofib(N\xrightarrow{V}N)\\
    &\simeq\cofib(LM\boxtimes N\xrightarrow{V}LM\boxtimes N).
\end{align*}
Thus the induced map is a derived mod $V$-equivalence, thus an equivalence
after derived mod $V$ reduction (see the proof of Lemma
\ref{lemmodVequivalence}). But then $L$ is $\H_0$ of the derived completion
(Lemma \ref{alg_completion}), thus it is also an equivalence on $L$.
\end{proof}

We recall from Theorem \ref{thm:cyclotomict} that the symmetric monoidal structure on $\CycSp_p$ is compatible with the $t$-structure. In particular it induces a symmetric monoidal structure on the heart as shown in Corollary \ref{cor_symmheart} (also see Appendix \ref{sub:compatible}). 

\begin{theorem}\label{mainsymm}
    The symmetric monoidal structure on the heart $\CycSp_p^\heart \simeq \CartV$
    induced from the one of cyclotomic spectra is equivalent to the symmetric
    monoidal structure  $\widehat \boxtimes$ constructed above.
\end{theorem}

\begin{proof}
    From Corollary \ref{tensor_heart} and the description above we get that the
    two tensor products are naturally isomorphic. By construction it is clear
    that these isomorphisms are also compatible with the associators and the
    symmetry (since the Bar construction can be iterated as well as the
    construction of the algebraic tensor product). The fact that we get a
    compatibility of units follows since $\pi_0^\cyc(\SS) \iso \TR_0(\SS) \simeq
    W(\ZZ)$ which is a result of Hesselholt and Madsen that holds for any
    commutative ring; see the proof of Theorem~\ref{thm:pi0} for precise references. But in this case it is
    also easy to see directly since the unit is also given by $\SS[[V]]$ as a
    consequence of Proposition \ref{symmon}.
\end{proof}

\subsection{Tensor product of Witt vectors and modules}\label{sub:wittmodules}

We note that an equivalent statement to Theorem \ref{mainsymm} is to say that 
\[
\pi_0^\cyc = \TR_0: (\CycSp_p)_{\geq 0} \to \CartV
\] 
is a symmetric monoidal functor, where $\CartV$ is equipped with the
$\widehat\boxtimes$ symmetric monoidal functor. This functor also preserves colimits and therefore also preserves relative tensor products over an algebra object in connective $p$-cyclotomic spectra. 
For rings $R$ and $S$ we have that 
$
\THH(R \otimes_\ZZ S) = \THH(R) \otimes_{\THH(\ZZ)} \THH(S).
$
It follows from the fact that $\pi_0^\cyc(\THH(R))=W(R)$ (see Theorem \ref{thm:pi0}) that we have 
\begin{align*}
W(R \otimes S) &= \pi^\cyc_0(R\otimes_\ZZ S) \\
&= \pi_0^\cyc\THH(R) \widehat\boxtimes_{\pi_0^\cyc\THH(\ZZ)} \pi_0^\cyc\THH(S) \\
& = W(R) \widehat\boxtimes W(S),
\end{align*}
where the last step uses that $W(\ZZ) = \pi_0^\cyc\THH(\ZZ)$ is the unit for
$\widehat\boxtimes$. We have not found a statement of this sort in the
literature and we think that this is an important property of Witt vectors, as
it for example explains how to understand Witt vectors of polynomial rings in
several variables. Therefore we want to give a purely algebraic proof of this
fact.

\begin{theorem}\label{thm:witttensor}
For every pair of rings $R$ and $S$ there is a natural isomorphism
\[
W(R) \widehat\boxtimes W(S) \cong W(R \otimes S) \ .
\]
\end{theorem}

Note that we do not require that $R$ and $S$ are commutative here and use non-commutative Witt vectors, but of course the statement is of most interest in the commutative case. 

\begin{proof}

Our first task is to produce a natural map
\[
W(R) \widehat\boxtimes W(S) \to W(R \otimes S) \ .
\]
In the commutative case this is easy: $W(R)$ and $W(S)$ are commutative algebra
objectives in $\CartV$ and therefore the tensor product is the coproduct. As a result to
construct such a map we can just combine the maps $W(R) \to W(R \otimes S)$ and $W(S) \to
W(R \otimes S)$ that come from the maps $R \to R\otimes S$ and $S \to R \otimes S$.

In the associative case such map is by construction of $\widehat \boxtimes$ the same as a
bilinear map $W(R) \times W(S) \to W(R \otimes S)$ of abelian groups with a certain
compatibility with Verschiebung and Frobenius. But such a map has exactly been constructed in
\cite{krause-nikolaus}.
By construction of the ring structure on $W(R)$ and $W(S)$ in the commutative case through
this lax symmetric monoidal map the two maps agree in the commutative case.

Now that we have the map we are left to show that the is an isomorphism. We can
resolve all our rings by rings of the form $\ZZ[M]$ for a monoid $M$ and since
$W(-)$ commutes with split coequalizers reduce to this case. But for $R =
\ZZ[M]$ we have a natural isomorphism $W(R) \iso R[[V]]$ with the terminology from Section
\ref{sec:symmonDieu}. We will verify this at the end of the proof, but let us assume it for
the moment.
Thus the statement reduces to show that the canonical map
\[
R[[V]] \widehat \boxtimes S[[V]] \to (R\otimes S)[[V]]
\]
is an isomorphism. This in turn follows from the statement that the canonical map 
\[
R[V] \otimes S[V] \to (R \otimes S)[V]
\]
is an isomorphism in $\Cart_p$. The latter is easy to verify using the explicit
formula for the tensor product in $\Cart_p$ given in Lemma \ref{lemexp}.

Now we want to verify that for a ring of the form $R= \ZZ[M]$ with a monoid $M$ we have a
canonical isomorphism of Cartier modules $R[[V]] \iso W(R)$.\footnote{This easily follows
    using Theorem \ref{thm:pi0} from the topological statement that $\TR(\SS[M]) \simeq
        \THH(\SS[M])[[V]]$ which is a formal consequence of the fact that $\THH(\SS[M])$ has a
        Frobenius lift. But of course the purpose of this proof is to avoid using
        topological arguments.}
For simplicity we assume that $M$ is a commutative monoid. The statement in fact holds with
the same proof also of non-commutative monoids but we do not want to get involved into
non-commutative Witt vectors here (see \cite{krause-nikolaus} for a discussion which shows
        that the arguments given here carry over directly except one has to verify that $V$
        is injective). First, we know that $W(R)$ is derived $V$-complete. Therefore  to
construct a map of Cartier modules $R[[V]] \to W(R)$ we have to construct a map $R \to W(R)$
compatible with $F$-operators where the $F$-operator on $R= \ZZ[M]$ is given by the $p$-th
power on $M$. Such a map in turn is the same as a map of multiplicative monoids $M \to W(R)$
and is canonically provided by the Teichm\"uller. Now both sides are $V$-torsion free, since
the Witt vectors of every commutative ring are $V$-torsion free and $R[[V]]$ evidently is.
Thus by Lemma \ref{lemmodVequivalence} is suffices to show that the map  $R[[V]] \to W(R)$
is an equivalence on the cokernel of $V$ which is obvious.
\end{proof}

Finally we want to study categories of modules over ring objects $A$ in
$\Cart_p$. The main example we have in mind is $W(R)$ for a commutative ring
$R$.  Thus let $A$ be a ring object in $\Cart_p$, that is $A$ is a ring, equipped with $V$
and $F$ operators such that
\[
FV = p, \qquad F(ab) = F(a)F(b), \qquad V(a F(b)) = V(a) b, \qquad V(F(a) b) =
a V(b).
\]
By an $A$-module, we mean an $A$-module object in $\Cart_p$ with respect to the symmetric monoidal structure $\boxtimes$. 

\begin{lemma}\label{lemma_dieudonne}
A left $A$-module structure on an object $M \in \Cart_p$ is the same as 
a left module structure for the underlying ring $A$ on the underlying abelian
group $M$ such that the following conditions hold for any $x\in A$ and $y\in
M$:
    \begin{enumerate}
        \item[(i)] $V(F_A(x)y)=xV(y)$;
        \item[(ii)] $V(xF(y))=V_A(x)y$;
        \item[(iii)] $F(xy)=F_A(x)F(y)$.
    \end{enumerate}\end{lemma}
\begin{proof}
This immediately follows from the definition of bilinear maps as in Definition~\ref{def}.
\end{proof}

\begin{example}\label{dieudonne}
    For $A = W(k)$ where $k$ is a perfect ring of characteristic $p$ we get
    that the category  $\Mod_{W(k)}(\Cart_p)$ is the category of
    $W(k)$-modules $M$ (in the classical sense of modules) together with a
    Frobenius semilinear map $F: M \to M$ and a map $V: M \to M$ that is
    semilinear for the inverse of Frobenius on $W(k)$, such that $VF = FV = p$. 
    This is
    the category of Dieudonn\'e modules.     
\end{example}

\section{The relationship to genuine cyclotomic spectra}\label{sec:genuine}

In this section we shall compare the $\infty$-category of $p$-typical topological Cartier
modules with the $\infty$-category of genuine $p$-typical cyclotomic spectra.
This generality will be
relevant in order to understand some cases of $V$-completion. It will also
provide a more conceptual proof of the main result of the Section~\ref{sec:tcm} on the
adjointness of $(-)/V$ and $\TR$. Moreover, although we do not explore it here,
we find an equivalent formulation of the theory of $p$-typical topological
Cartier modules in Proposition~\ref{prop_forget}, which immediately generalizes to the
integral (or big) situation. The
main results can be found in Section~\ref{consequences}. 


\subsection{Genuine $p$-typical topological Cartier modules}\label{genuineTCart}

The next definition parallels that of~\ref{def:gencyc}.

\begin{definition}
    The $\infty$-category $\TCart^\gen_p$ of {\bf genuine
    $p$-typical topological Cartier modules} is defined to be fixed points for the
    endofunctor $(-)^{C_p}$:
    \[
        \TCart_p^\gen := \mathrm{Fix}_{(-)^{C_p}}(\bT\Sp^\gen_p).
    \]
    In other words, a genuine $p$-typical topological Cartier module is a genuine
    $S^1$-spectrum $M$ equipped with an equivalence $M^{C_p}\we X$ of genuine
    $S^1$-spectra.
\end{definition}

Our primary aim in this section is to prove that $\TCart_p^\gen\we\TCart_p$, up
to a theorem about genuine $S^1$-spectra which we will prove in the next
section. In the next proposition, we will use the notion of an $F$-algebra for an
endofunctor $F\colon\Cscr\rightarrow\Cscr$. By definition, these are organized
into an $\infty$-category
$\Alg_F(\Cscr)\we\LEq(F,\id\colon\Cscr\rightarrow\Cscr)$. In other words, an
$F$-algebra is an object $X$ of $\Cscr$ equipped with a map $FX\rightarrow X$.
Similarly, if $G$ is an endofunctor of $\Dscr$, we have the opposite notion of
a $G$-coalgebra, namely an object $Y\in\Dscr$ equipped with a map $Y\rightarrow
GY$. These are objects of the $\infty$-category $\CoAlg_G(\Dscr)$. For details, see~\cite{nikolaus-scholze}*{Section~II.5}.

\begin{proposition}\label{gen_adjunction}
    There is an adjunction
    \[
    L: 
    \xymatrix{
    \TCart^\gen_p \ar[r]<2pt> & \CycSp^\gen_p \ar[l]<2pt> 
    }: \overline{\TR}
    \]
    with the functors
    \[
    L(M) = \colim(M \to M^{\Phi C_p} \to M^{\Phi C_{p^2}} \to \cdots )
    \]
    where the map $M \to M^{\Phi C_p}$ is given by the composite $M \simeq M^{C_p} \to M^{\Phi C_p}$. The right adjoint is given by
    \[
    \overline{\TR}(X) = \lim\left( \cdots \to X^{C_{p^2}} \to X^{C_p} \to X \right)
    \]
    where the maps are induced by $X^{C_p} \to X^{\Phi C_p} \simeq X$.
\end{proposition}

\begin{proof}
This is more generally true: if we have an $\infty$-category $\Cscr$ and a
natural transformation $F \to G$ of functors $F, G:  \Cscr \to \Cscr$ such that
$F$ preserves sequential limits (and these exist in $\Cscr$) and $G$ preserves
sequential colimits (also these are assumed the exist in $\Cscr$), then we get
an adjunction
\[
L: \xymatrix{
\mathrm{Fix}_F(\Cscr)  \ar[r]<2pt> & \mathrm{Fix}_G(\Cscr) \ar[l]<2pt> 
}: R
\]
where the underlying object of $L(X)$ is given by the colimit
\[
\colim\left(X \to GX \to G^2X \to \cdots\right)
\]
and the fixed point structure for $G$ comes from commuting this colimit with $G$. 
The underlying object of $R(Y)$ is given by 
\[
\lim\left(\cdots \to F^2Y \to FY \to Y\right)
\] 
and the fixed point structure comes from commuting the limit with $F$. 
To see that these functors are adjoint to one another, we observe that there is natural transformation
\[
\id \to RL
\]
given on underlying objects by the natural map
\[
X = \lim_k F^k X \to \lim_k F^k (\colim_n ( G^n X)).
\]
It now suffices to verify the mapping space property. To this end, we claim that for every $X \simeq FX \in \mathrm{Fix}_F(\Cscr) $ and every 
$Y \simeq GY \in \mathrm{Fix}_G(\Cscr)$ one has
\[
\Map_{\mathrm{Fix}_F(\Cscr) }(X, RY) \simeq \Map_{\mathrm{Alg}_F(\Cscr)}(X, Y),
\]
where $Y$ is considered as an $F$-algebra through the map $FY \to GY
\xto{\simeq} Y$. To see this equivalence we have used that taking the limit over iterated applications of $F$ is the right adjoint to the inclusion $\mathrm{Fix}_F(\Cscr) \subseteq \mathrm{Alg}_F(\Cscr)$ see 
\cite{nikolaus-scholze}*{Section~II.5}. The space $\Map_{\mathrm{Alg}_F(\Cscr)}(X, Y)$ is equivalent to the equalizer
\[
\xymatrix{
\Map_{\Cscr}(X,Y) \ar[r]<2pt> \ar[r]<-2pt> & \Map_{\Cscr}(FX, Y)
}
\]
where the first map is given by precomposition with $FX \to X$ and the second by applying $F$ and then postcomposition with $FY \to  Y$. 
We can similarly describe the mapping space $\Map_{\mathrm{Fix}_G(\Cscr) }(LX,
Y)$ via $\CoAlg_G(\Cscr)$ as the equalizer 
\[
\xymatrix{
\Map_{\Cscr}(X,Y) \ar[r]<2pt> \ar[r]<-2pt> & \Map_{\Cscr}(X, GY)
}
\]
along the analogous maps. Since $FX \simeq X$ and $GY \simeq Y$ is clear that
$\Map_{\Cscr}(FX, Y) \simeq \Map_{\Cscr}(X, Y) \simeq \Map_{\Cscr}(X,GY)$.
Unfolding the definitions we see that under this equivalence the maps in the
equalizers correspond to each other and such that they are equivalent. Moreover
the equivalence we get this way agrees (at least up to homotopy) with the map
induced from the transformation $\id \to RL$.
\end{proof}

In order to understand the functors $L$ and $\overline \TR$ we need the following result, which is in some sense a genuine version of the Tate orbit Lemma, but much easier.

\begin{lemma}\label{gen_tateorbit}
For every object  $M \in \bT\Sp^\gen_p$, the canonical map
\[
\left(M^{C_{p}} \right)^{\Phi C_p} \to M^{\Phi C_{p^{2}}}
\]
is an equivalence in $\bT\Sp^\gen$.
\end{lemma}

\begin{proof}
For this proof we will write the left hand side as $\left(M^{C_{p}} \right)^{\Phi C_{p^2}/C_p}$ and note that is carries a residual $S^1/ C_{p^2}$-action. 
By definition (see for example~\cite{nikolaus-scholze}*{Definition~II.2.3}),
the $C_{p^2}/C_p$-geometric fixed points are given by
\[
\left(M^{C_{p}} \right)^{\Phi C_{p^2}/C_p} \simeq \left(\left(M^{C_{p}} \right) \otimes \widetilde{E(S^1/C_p)}\right)^{C_p^2/C_p},
\]
where the pointed genuine $S^1/C_p$-space $\widetilde{E(S^1/C_p)}$ is characterized by having fixed points
\[
\widetilde{E(S^1/C_p)}^{C_{p^k}/C_p} = \begin{cases}
\ast & k = 1, \\
S^0 & k > 1 
\end{cases}
\]
and the obvious maps (so that it receives a map from the space $S^0$). 
By the projection formula we get an equivalence  of $S^1/C_p$-spectra
\[
\left(M^{C_{p}} \right) \otimes \widetilde{E(S^1/C_p)} \simeq  M \otimes f^*\left(\widetilde{E(S^1/C_p)} \right)
\]
where $f: S^1 \to S^1/C_p$ is the canonical projection and $f^*$ is inflation
along this projection (i.e. restriction of the action which is left adjoint to
taking fixed points). The pullback $f^*\left(\widetilde{E(S^1/C_p)} \right)$ is
equivalent to the  genuine $S^1$-space $\widetilde{E_{p^2}S^1}$ which has fixed
points
\[
\widetilde{E_{p^2}S^1}^{C_{p^k}} = \begin{cases}
\ast & k = 0,1, \\
S^0 & k > 1.
\end{cases}
\]
Putting everything together we conclude that 
\[
\left(M^{C_{p}} \right)^{\Phi C_{p^2}/C_p} \simeq  \left(M \otimes \widetilde{E_{p^2}S^1}\right)^{C_{p^2}} 
\]
which is by definition equivalent to $M^{\Phi C_{p^{2}}}$.
\end{proof}

Let $M$ be a genuine $p$-typical topological Cartier module. In the next
corollary, we identify the $C_{p^n}$ fixed points of the genuine $p$-typical
cyclotomic spectrum $LM$ from Proposition \ref{gen_adjunction}.

\begin{corollary}\label{corfixed}
    For every genuine $p$-typical topological Cartier module $M$ we have a
    natural equivalence
    \[
    (LM)^{C_{p^n}} \simeq M/ V^{n+1},
    \]
    where $V^{n+1}$ denotes the map $M_{hC_{p^{n+1}}} \to M^{C_{p^{n+1}}} \xto{\simeq} M$.
\end{corollary}

\begin{proof}
    By definition we have that $L(M) = \colim(M \to M^{\Phi C_p} \to M^{\Phi C_{p^2}} \to \cdots )$ and since (categorical) fixed points commute with colimits we have that 
    \[
    (LM)^{C_{p^n}} = \colim(M^{C_{p^n}} \to (M^{\Phi C_p})^{C_{p^n}} \to (M^{\Phi C_{p^2}})^{C_{p^n}} \to \cdots).
    \]
    The transition maps in this colimit are given by the maps
    \[
    (M^{\Phi C_{p^k}})^{C_{p^n}} \xto{\simeq} ((M^{C_p})^{\Phi C_{p^k}})^{C_{p^n}} \to (M^{\Phi C_{p^{k+1}}})^{C_{p^n}} 
    \]
    which are equivalences by Lemma \ref{gen_tateorbit} for $k \geq 1$. As a result, the colimit is equivalent to the second term 
    $(M^{\Phi C_p})^{C_{p^n}}$. This sits in an isotropy separation cofiber sequence
    \[
    M_{hC_{p^{n+1}}} \to M^{C_{p^{n+1}}} \to (M^{\Phi C_p})^{C_{p^n}}
    \]
    where the middle term is identified with $M$ through the general
    $p$-typical Cartier module structure and the left map with $V^{n+1}$. This shows the claim. 
\end{proof}

Now we come to the key assertion of the section.

\begin{proposition}\label{prop_forget}
    The canonical forgetful functor $\TCart^\gen_p \to \TCart_p$ (which will be made explicit in the proof) is an equivalence of $\infty$-categories. 
\end{proposition}

\begin{proof}
    Fix $n\geq 0$.
We  consider the $\infty$-category
\[
\bT\Sp_{C_{p^n}}^\gen
\]
of genuine $S^1$-spectra for the family of subgroups of $C_{p^n}\subseteq S^1$. In particular $\bT\Sp_{C_{p^0}}^\gen = \Sp^{BS^1}$. 
Taking $C_p$-fixed points gives a functor 
\[
(-)^{C_p} \colon \quad \bT\Sp_{C_{p^{n+1}}}^\gen \to \bT\Sp_{C_{p^n}}^\gen
\]
and there is another functor $U_{n+1}\colon \bT\Sp_{C_{p^{n+1}}}^\gen \to
\bT\Sp_{C_{p^n}}^\gen$ which forgets the $C_{p^{n+1}}$-fixed points. We define a stable $\infty$-category
$(\TCart^\gen_p)_n$ as the equalizer
\[
    (\TCart^\gen_p)_n := \Eq\left((-)^{C_p},U_{n+1}\colon  \bT\Sp_{C_{p^{n+1}}}^\gen \xymatrix{\ar[r]<2pt>\ar[r]<-2pt> &} \bT\Sp_{C_{p^{n}}}^\gen \right)
\]
of $(-)^{C_p}$ and $U_{n+1}$. There is a canonical map
$(\TCart^\gen_p)_n \to (\TCart^\gen_p)_{n-1}$ induced from the pair of
commutative diagrams\footnote{Here, we mean that the diagram with the
$(-)^{C_p}$ horizontal maps is commutative as is the diagram with the $U_{n+1}$
and $U_n$ horizontal maps.}
\[
\xymatrix{
\bT\Sp_{C_{p^{n+1}}}^\gen
\ar[d]_{U_{n+1}}\ar[r]<2pt>^{(-)^{C_p}}\ar[r]<-2pt>_{U_{n+1}} &
\bT\Sp_{C_{p^{n}}}^\gen \ar[d]^{U_{n}}  \\
\bT\Sp_{C_{p^{n}}}^\gen \ar[r]<2pt>^{(-)^{C_p}}\ar[r]<-2pt>_{U_n} &
\bT\Sp_{C_{p^{n-1}}}^\gen.
}
\]
The main point of the proof is to show that  $(\TCart_p^\gen)_n \to (\TCart_p^\gen)_{n-1}$ is an equivalence for every $n \geq 1$. Let us first assume that and finish the proof.
Since $\bT\Sp_p^\gen \simeq \underleftarrow{\lim} \bT\Sp_{C_{p^{n}}}^\gen$ it
follows that $\TCart^\gen_p \simeq  \lim_n
(\TCart_p^\gen)_{n}$, and therefore, under our assumption,
\[
\TCart^\gen_p \simeq (\TCart^\gen_p)_0 \ .
\]
To identify the  $n = 0$ case, we claim that there is a pullback square
 \begin{equation}\label{pulli}
 \xymatrix{
 \bT\Sp_{C_{p}}^\gen\ar[rr]\ar[d]_{U_1}  &&
    \left(\Sp^{BS^1}\right)^{\Delta^2}\ar[d]^{\partial^1}\\
    \Sp^{B\S^1}\ar[rr]^-{\Nm_{C_p}}&&\left(\Sp^{B\S^1}\right)^{\Delta^1}
    }
 \end{equation}
 of $\infty$-categories. This fact is equivalent to the fact that the `Tate square'
 determines $ \bT\Sp_{C_{p}}^\gen$ which is well-known, see e.g. \cite{mnn-descent}*{Theorem~6.24}. The translation between these two facts is given by the observation that a pullback square 
 \[
 \xymatrix{
 A \ar[r]\ar[d] & B \ar[d] \\
 C \ar[r] & D
 }
 \]
 in any stable $\infty$-category (here $\Sp^{BS^1}$) is equivalently determined by the triangle 
 \[
  \xymatrix{
 F \ar[r]\ar[rd] & A \ar[d]  \\
  & D
 }
 \]
 where $F$ is the horizontal fiber. 
 
 The two maps $\bT\Sp_{C_{p}}^\gen \to\bT\Sp_{C_{p^{0}}}^\gen = \Sp^{BS^1}$ that feature in the definition of $(\TCart^\gen_p)_0$ are in the  pullback description \eqref{pulli} given by the two maps 
that are given by the composites
\[
    (-)^{C_p}\colon\bT\Sp_{C_{p}}^\gen\to \left(\Sp^{BS^1}\right)^{\Delta^2}
\xto{\mathrm{ev_1}=\partial^0\partial^2} \Sp^{BS^1}
\]
and
\[
U_1\colon \bT\Sp_{C_{p}}^\gen \to  \Sp^{B\S^1} \ .
\]
If thus follows that the equalizer defining $(\TCart^\gen_p)_0$ is equivalent to the pullback defining $\TCart_p$. 

Now in order to show that the functor $(\TCart_p^\gen)_n \to (\TCart_p^\gen)_{n-1}$ is an equivalence we will use that for $n \geq 1$ the square
\begin{align}\label{square}
\xymatrix{
\bT\Sp_{C_{p^{n+1}}}^\gen\ar[r]^{(-)^{C_p}}\ar[d]_{U_{n+1}} &
\bT\Sp_{C_{p^{n}}}^\gen \ar[d]^{U_n}\\
\bT\Sp_{C_{p^{n}}}^\gen \ar[r]^{(-)^{C_p}}  & \bT\Sp_{C_{p^{n-1}}}^\gen
 }
\end{align}
is a pullback square of stable $\infty$-categories which we will prove in
Theorem \ref{pullback} below. Given this, the claim follows from the fact that
for a general pair of maps $f,g: A \to B$ in an $\infty$-category with finite
limits (here the $\infty$-category of stable $\infty$-categories), we have a
canonical equivalence
\[
    \Eq\left(A~_f\!\times_g A
    \xymatrix{\ar[r]<2pt>^{\mathrm{pr}_1}\ar[r]<-2pt>_{\mathrm{pr}_2} &} A \right) \simeq \Eq\left(A  \xymatrix{\ar[r]<2pt>^f\ar[r]<-2pt>_g &} B\right).
\]
This follows by taking the `total equalizer' in the commutative diagram
\[
\xymatrix{
A\times A \ar[r]<2pt>^-{\mathrm{pr}_1}\ar[r]<-2pt>_-{\mathrm{pr}_2}\ar[d]<2pt>^{g \mathrm{pr}_2}\ar[d]<-2pt> _{f \mathrm{pr}_1}& A\ar[d]<2pt>\ar[d]<-2pt> \\
B  \ar[r]<2pt>\ar[r]<-2pt> & \ast
}
\]
in two different ways (first horizontally and then vertically or vice versa). Then we get the diagram
\[
\xymatrix{
E \ar[d]\ar[r] & A \times_B A \ar[r]<2pt>\ar[r]<-2pt> \ar[d]& A \ar[d] \\
A \ar[r] \ar[d]<2pt>\ar[d]<-2pt>& A\times A \ar[r]<2pt>\ar[r]<-2pt>\ar[d]<2pt>\ar[d]<-2pt> & A\ar[d]<2pt>\ar[d]<-2pt> \\
B \ar[r] & B  \ar[r]<2pt>\ar[r]<-2pt> & \ast
}
\]
in which the two `outer forks' are equalizers.  
\end{proof}

\subsection{A pullback square}

Now, to complete the proof of Proposition~\ref{prop_forget}, we have to establish that for $n\geq 1$, the square~\eqref{square} is a
pullback. To this end we shall introduce some terminology and establish an abstract criterion for pullbacks of stable $\infty$-categories.

\begin{definition}
Let $U: \mathcal{C} \to \mathcal{D}$ be a functor of stable $\infty$-categories. We say that
$U$ exhibits a recollement if it admits a fully faithful left adjoint and a fully faithful
right adjoint, i.e. $U$ is a localization and a colocalization.
\end{definition}

For the next example we follow the notation of Proposition~\ref{prop_forget}.

\begin{example}\label{ex_tate}
    The functor $U_{n+1}\colon \bT\Sp_{C_{p^{n+1}}}^\gen \to \bT\Sp_{C_{p^{n}}}^\gen $
    exhibits a recollement. The left adjoint is given by forming the `free'
    $C_{p^{n+1}}$-spectrum $\mathrm{Fr}(X)$ on an  $C_{p^{n}}$-spectrum $X$. For the
    spectrum $\mathrm{Fr}(X)$ we have that 
\[
\mathrm{Fr}(X)^{C_{p^k}} = \begin{cases}
 X^{C_{p^k}} & \text{ for } k=0,..., n, \\
 \left(X^{C_{p^n}}\right)_{hC_p} & \text{ for } k= n+1.
\end{cases}
\]
The right adjoint is given by the `Borel complete' spectrum $BX$ with 
\[
BX^{C_{p^k}} = \begin{cases}  
X^{C_{p^k}} & \text{ for } k=0,..., n,\\
\left(X^{C_{p^n}}\right)^{hC_p} & \text{ for } k= n+1.
\end{cases}
\]
The canonical maps $U B \to \id \to U \mathrm{Fr}$ are equivalences. This shows that $B$ and $\mathrm{Fr}$ are fully faithful. 
\end{example}

\begin{lemma}\label{cartcoCart}
    If $U: \mathcal{C} \to \mathcal{D}$ exhibits a recollement, then $\ker(U)
    \to \mathcal{C} \to \mathcal{D}$ is a Verdier sequence. Moreover the
    functor $U$ is a Cartesian and coCartesian fibration.\footnote{Here we mean the
    invariant concept, i.e. that every replacement by a categorical fibration is a Cartesian
    and coCartesian fibration.}
\end{lemma}

\begin{proof}
The first statement is clear since $U$ is a localization (or colocalization)
and thus $\mathcal{D}$ is the Verdier quotient by the acyclics. Now we want to
show that $U$ is a Cartesian fibration in the sense
of~\cite{htt}*{Section~2.4}; the coCartesian case follows by passing
to opposite categories. We claim that a morphism $f: c \to c'$ in $\mathcal{C}$
is $U$-Cartesian precisely if the square
\begin{equation}\label{square1}
\xymatrix{
c \ar[r]^f\ar[d] & c' \ar[d] \\
RUc \ar[r]^{RUf} & RUc'
}
\end{equation}
is a pullback in $\Cscr$ where $R: \Dscr \to \Cscr$ is the right adjoint to $U$ and the vertical maps in the diagram are the unit maps. 
By definition, the morphism $f: c \to c'$ is Cartesian precisely if the square
\[
\xymatrix{
\mathcal{C}_{/c} \ar[r]^{f_*}\ar[d]^U & \mathcal{C}_{/c'} \ar[d]^U\\
\mathcal{D}_{/Uc} \ar[r]^{Uf_*} & \mathcal{D}_{/Uc'}
} 
\]
of $\infty$-categories is a pullback. We consider the larger commutative diagram
\begin{equation}\label{bigsquare}
\xymatrix{
\mathcal{C}_{/c} \ar[r]^{f_*}\ar[d]^U & \mathcal{C}_{/c'} \ar[d]^U\\
\mathcal{D}_{/Uc} \ar[r]^{Uf_*} \ar[d]^R& \mathcal{D}_{/Uc'}\ar[d]^R \\
\mathcal{C}_{/RUc} \ar[r]^{RUf_*}& \mathcal{C}_{/RUc'}
}  \ .
\end{equation}
We will show that the lower square is a pullback. Then the upper square is a
pullback if the outer square is a pullback which is evidently the case if and only
if the
square \eqref{square1} is a pullback. To see that the lower square  in \eqref{bigsquare} is
a pullback we first note that the vertical functors $R$ are fully faithful since the functor
$R: \Dscr \to \Cscr$ is. The left hand vertical functor identities $\mathcal{D}_{/Uc}$ with
the full subcategory of $\mathcal{C}_{/RUc}$ consisting of morphisms $c \to  RUc$ such that
$c$ is in the essential image of $R$. A similar description holds for the essential image of
$\mathcal{D}_{/Uc'}$ in $\mathcal{C}_{/RUc'}$. From this description it is obvious that the
square is a pullback if one uses that pullbacks of fully faithful subcategory inclusions are
fully faithful and given by the obvious preimage.

We now want to argue that $U\colon \Cscr \to \Dscr$ is Cartesian. Thus assume that we have
$c' \in \Cscr$ and a morphism $f': d \to Uc'$ in $\Dscr$. We need to find a Cartesian
morphism $f: c \to c'$ in $\Cscr$ with $Uc \simeq d$ and $Uf \simeq f'$.   We extend the
diagram
\[
\xymatrix{
& c' \ar[d] \\
Rd\ar[r]^{Rf'} & RUc' 
}
\]
to a pullback
\[
\xymatrix{
c\ar[r]^f\ar[d] & c' \ar[d] \\
Rd\ar[r]^{Rf'} & RUc' 
}
\]
and let $f$ be the upper horizontal morphism. First we claim that the morphism $RUc \to d$
induced from the left vertical map is an equivalence.  This follows by applying the exact
functor $RU$ to the whole diagram and noting that $RUR = R$. This shows that the whole
diagram is equivalent to a diagram
\[
\xymatrix{
c\ar[r]^f\ar[d] & c' \ar[d] \\
RUc\ar[r]^{RUf} & RUc' 
}
\]
From this we conclude that $f$ is Cartesian and that $Uf = f'$ which finishes
the proof.
\end{proof}

\begin{definition}
    Assume that $U: \Cscr \to\Dscr$ exhibits a recollement and $d \in\Dscr$. We define the Tate
    object $T(d) \in \ker(U)$ to be the cofiber of the canonical map $L(d) \to R(d)$ where $L$
    is the left adjoint of $U$ and $R$ is the right adjoint of $R$.
\end{definition}

We now give our criterion for a pullback. This is similar to the main result of \cite{BarSaul}. 

\begin{proposition}\label{criterion}
Assume that we have a commutative square 
\[
\xymatrix{
\mathcal{C} \ar[r]^F \ar[d]_{U} & \mathcal{C}' \ar[d]^{U'} \\
\mathcal{D} \ar[r]^G & \mathcal{D}'
}
\]
of stable $\infty$-categories and exact functors. If $U$ and $U'$ exhibit recollements then
the square is a pullback if and only of the following two conditions are satisfied:
\begin{enumerate}
    \item[{\rm (1)}] the vertical kernels agree, i.e. the canonical map $\ker(U) \to \ker(U')$ is an equivalence;
    \item[{\rm (2)}] under this equivalence the Tate object $T(d)$ is taken to the Tate
    object $T'(Gd)$, more precisely for every $d \in \Dscr$ the canonical map
$FT(d) \to T'(Gd)$ is an equivalence in $\ker(U')$. 
\end{enumerate}
\end{proposition}

\begin{proof}
    The functor $U: \mathcal{C} \to \mathcal{D}$ is a (co)Cartesian fibration.
    The functor   $U': \mathcal{C}' \to \mathcal{D}'$ also is a coCartesian
    fibration, thus so is the pullback $\mathcal{C}' \times_{\mathcal{D}'}
    \mathcal{D}\rightarrow\Dscr$. Therefore, in order to show that  the functor
    $\mathcal{C} \to  \mathcal{C}' \times_{\mathcal{D}'} \mathcal{D}$ is an
    equivalence, it suffices to show that it is a fiberwise equivalence over $\mathcal{D}$, or equivalently that
    the functor $\mathcal{C} \to \mathcal{C}'$ induces an equivalence $\mathcal{C}_d \to \mathcal{C}'_{Gd}$, and that the functor preserves coCartesian lifts.
    An easy adjunction argument shows that the fiber over an object $d \in \mathcal{D}$ is equivalent to $\ker(U)_{/T(d)}$. Similarly for the fibration
  the fiber over $Gd$ is given by 
    $\ker(U')_{/T'(Gd)}$. Under these identifications the map on fibers is given by the evident map
    \[
    \ker(U)_{/T(d)} \to \ker(U')_{/T'(Gd)}
    \]
    coming from the functor $F: \ker(U) \to \ker(U')$ and the map of Tate objects $FT(d) \to T'(Gd)$. 
    This immediately implies that the map is a fiberwise equivalence under the assumptions of the Proposition.

    To see that the functor $\Cscr \to \Cscr'$ preserves coCartesian lifts we note that for a map $d \to d'$ in $\mathcal{D}$ the induced map on fibers of 
    $\mathcal{C} \to \mathcal{D}$ is given by the functor
    \[
    \ker(U)_{/T(d)}  \to \ker(U)_{/T(d')} 
    \]
    induced from the map $Td \to Td'$. This follows immediately from the  way the fibers are
    identified and from the description of coCartesian lifts in the proof of Lemma
    \ref{cartcoCart}. Similar the map on fibers $\mathcal{C}' \to \mathcal{D}'$ induced from
    the map $Gd \to Gd'$ is described by
     \[
    \ker(U')_{/T(Gd)}  \to \ker(U')_{/T(Gd')} 
    \]
    induced from the map $T(Gd) \to T(G'd'')$. Then the claim follows from the fact that in the commutative square
    \[
    \xymatrix{
    FT(d) \ar[r]\ar[d] & FT(d') \ar[d] \\
        T(Gd) \ar[r] & T(Gd') 
    }
    \]
    the vertical maps are equivalences by assumption.
    \end{proof}
    
Below, $\bT\Sp_{C_{p^{n}}}^\gen$ is the $\infty$-category of genuine $S^1$-spectra with respect to
the family of subgroup of $C_{p^n}$.

\begin{theorem}\label{pullback}
The square 
\begin{align}
\xymatrix{
\bT\Sp_{C_{p^{n+1}}}^\gen \ar[r]^{(-)^{C_p}}\ar[d]_{U_{n+1}} &
\bT\Sp_{C_{p^{n}}}^\gen \ar[d]^{U_n}\\
 \bT\Sp_{C_{p^{n}}}^\gen \ar[r]^{(-)^{C_p}}  & \bT\Sp_{C_{p^{n-1}}}^\gen
 }
\end{align}
is a pullback of stable $\infty$-categories.
\end{theorem}

\begin{proof}
    By Example~\ref{ex_tate}, the vertical maps $U_{n+1}$ and $U_n$ exhibit
    recollements. We want to apply the criterion given in Proposition
    \ref{criterion}. The vertical fibers are both equivalent to $\Sp^{BS^1}$, where
    the equivalence $\ker(U_{n+1}) \to \Sp^{BS^1}$ is induced by taking
    ${C_{p^{n+1}}}$-fixed points for the left vertical map and similarly by taking  ${C_{p^{n}}}$-fixed points  for
    the right vertical map. Using the
    description of the adjoint given in Example \ref{ex_tate}, it follows
    immediately that the Tate objects are $T(X) \we (X^{C_{p^n}})^{tC_p} \in  \Sp^{BS^1}$ for $X$ in
    $\bT\Sp_{C_{p^{n}}}^\gen$ and  $T(X) \we (X^{C_{p^{n-1}}})^{tC_p} \in \Sp^{BS^1}$ for $X$ in
    $\bT\Sp_{C_{p^{n-1}}}^\gen$. This together with the description of the
    functors implies that criterion (2) of Proposition~\ref{criterion} is
    satisfied.
\end{proof}

\subsection{Consequences}\label{consequences}

Now we can deduce the following statement from the results of Section \ref{genuineTCart}.

\begin{theorem}\label{equivalencequestion}
We have an adjunction
\[
L: 
\xymatrix{
\TCart_p \ar[r]<2pt> & \CycSp^\gen_p \ar[l]<2pt> 
}:  \TR
\]
with $(LM)^{C_{p^n}} = M/V^{n+1}$ and $\TR(X) = \lim_{n,R}
X^{C_{p^n}}$. The unit of the adjunction is given by the `naive' $V$-completion
$M \to \lim  M/V^{n+1}$ and the counit is given on underlying spectra by the map $\TR(X)/V \to X$. 
\end{theorem}
\begin{proof}
By Proposition \ref{gen_adjunction} we have an adjunction
    \[
    L: 
    \xymatrix{
    \TCart^\gen_p \ar[r]<2pt> & \CycSp^\gen_p \ar[l]<2pt> 
    }: \overline{\TR}
    \]
which we compose with the equivalence $U: \TCart_p^\gen \xto{\simeq} \TCart_p$
of Proposition \ref{prop_forget} which extracts the `underlying' $p$-typical topological
Cartier module. The value of the composite $U \circ \overline{\TR}$ at $X$ has by definition of
$\overline{\TR}$ the  underlying spectrum $\TR(X) \simeq \lim_{n,R}
X^{C_{p^n}}$ as classically defined (we have used that taking the underlying spectrum commutes with taking limits of genuine spectra). 

Now we need to determine the fixed points of the spectrum associated with a
$p$-typical topological Cartier module $M \in \TCart_p$. First, by the way the equivalence
$U: \TCart_p^\gen \xto{\simeq} \TCart_p$ we find that the map $M_{hC_p^{n+1}}
\to M^{C_{p^{n+1}}}$ (where we have abusively identified $U^{-1}M$ with $M$) is
given by the composite
\[
V^{n+1}: M_{hC_{p^{n+1}}} \to \cdots \to M_{hC_p} \to M.
\]
Thus the claim follows from Corollary \ref{corfixed}.
\end{proof} 

So far, we have only defined what it means to be $V$-complete for bounded below
$p$-typical topological Cartier modules. In general we shall consider the Bousfield
localization of $\TCart_p$ at the mod $V$-equivalences. For this definition it
does not matter if one considers $(-)/V$ as taking values in $p$-typical cyclotomic spectra, genuine
$p$-typical cyclotomic spectra, or spectra. The local objects for this
Bousfield localization are in the bounded below case precisely the $V$-complete
$p$-typical topological Cartier modules which follows from Proposition~\ref{thm:tr} and
Theorem \ref{thm:boundedbelow}.

\begin{corollary}\label{cor_equiv}
For every genuine $p$-typical cyclotomic spectrum $X$ the induced $p$-typical topological Cartier module
$\TR(X)$ is complete with respect to the Bousfield localization at the mod $V$
equivalences (see Remark \ref{modV}). Moreover, there is an induced adjunction
\[
L: 
\xymatrix{
\TCartV \ar[r]<2pt> & \CycSp^\gen_p \ar[l]<2pt> 
}:  \TR,
\]
which is an equivalence on bounded below objects, where $\TCartV$ denotes the Bousfield localization at the mod $V$-equivalences.
\end{corollary}

\begin{proof}
We have to verify that $\TR(X)/V \to X$ is an equivalence for $X$ bounded
below. But this follows from the cofiber sequences $\TR^n(X)_{hC_p} \to \TR^{n+1}(X) \to X$
by passing to the limit and noting that homotopy orbits
commutes with this limit since it is uniformly bounded below.

Moreover from the description in Theorem \ref{equivalencequestion} it follows
that both functors restrict to an adjunction between bounded below objects.
Now,
if $M$ is derived $V$-complete, then by definition the unit of the adjunction is
an equivalence.
\end{proof}

\begin{remark}
    In Proposition~\ref{thm:tr} we have proven that there is an adjunction
    \[
        (-)/V\colon 
    \xymatrix{
    \TCart_p \ar[r]<2pt> & \CycSp_p \ar[l]<2pt> 
    }:  \TR
    \]
    for the `simplistic' version of $p$-typical cyclotomic spectra. Clearly the functor
    $(-)/V$ factors over the forgetful functor $\CycSp^\gen_p \to \CycSp_p$ which preserves
    colimits and is therefore left adjoint as well. Thus we get the adjunction of
    Proposition~\ref{thm:tr} as the composite of the two adjunctions
    \[
    \xymatrix{
    \TCart_p \ar[r]<2pt> & \CycSp_p^\gen \ar[l]<2pt> \ar[r]<2pt> &  \CycSp_p\ar[l]<2pt>.
    }
    \]
    Both of these adjunctions induce equivalences on subcategories of bounded
    below objects, where on the left we look at the full subcategory of
    $V$-complete bounded below objects. The left hand one as just shown in Corollary~\ref{cor_equiv} and the right hand one as shown in \cite{nikolaus-scholze}. The hardest part of the latter is in fact understanding the right adjoint to the forgetful functor $ \CycSp_p^\gen  \to \CycSp_p$. Thus Corollary \ref{cor_equiv} also gives an independent proof of Theorem \ref{thm:boundedbelow}.
    Alternatively, Corollary~\ref{cor_equiv} and Theorem~\ref{thm:boundedbelow}
    can be used to reprove the $p$-typical part of~\cite{nikolaus-scholze}*{Theorem~II.3.8}.
\end{remark}

It is now natural to ask if the adjunction of Corollary~\ref{cor_equiv} is an
equivalence. For the simplistic adjunction 
$
(-)/V: 
\xymatrix{
\TCart_p \ar[r]<2pt> & \CycSp_p \ar[l]<2pt> 
}:  \TR
$
this is obviously not the case since we have used the Tate orbit lemma in a
crucial way. The genuine
version has a much better chance of inducing an equivalence. The following
result shows that the genuine adjunction is not an equivalence.

\begin{proposition}\label{prop_not}
    The adjunction of Corollary~\ref{cor_equiv} does not form an equivalence of stable $\infty$-categories. 
\end{proposition}

\begin{proof}
If the adjunction of Corollary~\ref{cor_equiv} were an equivalence, then the counit of the adjunction
$\TR(X)/V \to X$ would be an equivalence for every genuine $p$-typical cyclotomic spectrum $X$. This is equivalent to the assertion that the inverse limit defining $\TR(X)$ commutes with taking homotopy orbits (since it commutes with taking cofibers and $X$ is the cofiber of the `commuted' map). 
We will show that this is not the case for $X = \KU^\triv$ where $\KU^\triv$
refers to the genuine $S^1$-spectrum given by the inflation of $\KU$ along the projection $\bT \to e$.
This carries a canonical $p$-typical cyclotomic structure, since the geometric fixed points are given by $\KU$ itself. More precisely there is a functor
\[
(-)^\triv: \Sp \to \CycSp^\gen_p
\]  
which is uniquely determined by requiring that it is left adjoint (i.e.
preserves all colimits) and sends the sphere to the genuine $p$-typical cyclotomic sphere. In other words we have that $\KU^\triv \simeq \SS^\triv \otimes \KU$. The composite 
\[
\Sp \to \CycSp^\gen_p \to \bT\Sp_p^\gen \xto{(-)^{C_{p^n}}} \Sp 
\] then also preserves all colimits for each $n$. Thus, we find using the tom Dieck splitting that 
\begin{align*}
\TR^{n+1}(\KU^\triv) &= (\KU^\triv)^{C_{p^n}} 
\\&\simeq \KU \otimes \SS^{C_{p^n}} \\
 &\simeq \KU \otimes (\bigoplus_{k = 0}^n \SS \otimes {BC_{p^k}}) 
  \simeq \bigoplus_{k = 0}^n \KU \otimes BC_{p^k}.
\end{align*}
The transition maps $R: \TR^{n+1}(\KU^\triv) \to \TR^{n}(\KU^\triv)$ are under
these identifications given by the projections away from the last
factor. As a result we find that
\[
\TR(\KU^\triv) \simeq \prod_{k= 0}^\infty (\KU \otimes BC_{p^k}) \ .
\]
Now we will prove that the canonical map
\begin{equation}\label{mappp}
\big(\prod_{k= 0}^\infty \KU \otimes BC_{p^k} \big)_{hC_p} \to \prod_{k= 0}^\infty \left(\KU \otimes BC_{p^k}\right)_{hC_p} 
\end{equation}
is not surjective on $\pi_1$. We write $\KU \otimes BC_{p^k}$ as the colimit of its negative
Postnikov tower $ \colim \tau_{\geq n} (\KU \otimes BC_{p^k})$ for $n \to -\infty$.
Then
\[
\big(\prod_{k= 0}^\infty \KU \otimes BC_{p^k} \big)_{hC_p} \simeq \colim_n \prod_{k= 0}^\infty  \tau_{\geq n}\left(\KU \otimes BC_{p^k} \right)_{hC_p} 
\]
since for uniformly bounded below spectra we can commute the homotopy orbits with the
infinite product. Similarly we get that the right hand side of \eqref{mappp} is given as
\[
 \prod_{k= 0}^\infty \left(\KU \otimes BC_{p^k}\right)_{hC_p}  \simeq  \prod_{k= 0}^\infty \colim_n \tau_{\geq n}\left(\KU \otimes BC_{p^k}\right)_{hC_p} 
\]
so that the map \eqref{mappp} is the canonical map that commutes the filtered colimit with the infinite product. On $\pi_1$ this map induces the map
\[
 \colim_n\prod_{k= 0}^\infty M_{k,n}  \to  \prod_{k= 0}^\infty \colim_n M_{k,n}
\]
with $M_{k,n} := \pi_1\left( \tau_{\geq n}\left(\KU \otimes BC_{p^k}\right)_{hC_p} \right)$.
Algebraically such a map is a surjective precisely if there exists an $n_0$ such that for
almost all $k$ the map
\[
 M_{k,n_0} \to \colim_n M_{k,n}
\]
is surjective. We will show that in our case none of the map $ M_{k,n_0} \to \colim_n  M_{k,n}$ is surjective. 

We begin by determining the homotopy groups of $\KU_{hC_{p^k}}$. Greenlees \cite{MR1217070} has
shown that they are isomorphic to the local cohomology groups of the
representation ring (and explicitly the homotopy groups had  been computed by
Wilson \cite{MR0328963} and Knapp \cite{MR0470961} before). We get that
\[
\pi_*\left(\KU_{hC_{p^k}}\right) = \begin{cases}
\ZZ & *  \text{ even}\\
I(C_{p^k}) \otimes \QQ_p/\ZZ_p & * \text{ odd}
\end{cases}
\]
where $I(C_{p^k}) \subseteq R(C_{p^k})$ is the augmentation
ideal\footnote{Really we should have the dual $I(C_{p^k})$ of the augmentation
ideal in odd degrees but we will only need that $I(C_{p^k}) = \ZZ^{\oplus(p^k-1)}$ as
abelian groups.} in the representation ring of $C_{p^k}$ which is isomorphic to
the group ring $\Z[C_{p^k}^\vee]$. This can also be seen by a direct
computation using the cofiber sequence
\[\KU_{hC_{p^k}} \to \KU^{hC_{p^k}} \to \KU^{tC_{p^k}}
\]
(the second and third term are even and the last term is rational) or by using
the universal coefficients theorem for $\KU$ and the Atiyah--Segal completion theorem. Now the filtration in question comes from the horizontal lines in the 
homotopy orbits spectral sequence
\begin{equation}\label{Tate_SS}
\E^2_{i,j} =   \H_{i}\left(BC_{p}, \pi_j(\KU_{hC_{p^k}})\right) \Rightarrow
\pi_{i + j} \left( \left(\KU_{hC_{p^k}}\right)_{hC_p} \right) 
\end{equation}
which is conditionally convergent. Recall that
\[
\H_i(BC_{p},\ZZ ) = \begin{cases}
\ZZ & * = 0, \\
\FF_p & *>0 \text{ odd}, \\
0 & \text{ else}
\end{cases}
\qquad  \text{and} \qquad
\H_i(BC_{p},\QQ_p/\ZZ_p ) = \begin{cases}
\QQ_p/\ZZ_p & * = 0, \\
\FF_p & *>0 \text{ even}, \\
0 & * \text{ odd}.
\end{cases}
\]
Thus
the homotopy orbits spectral sequence is in homological Serre grading a right
half plan spectral sequence  which is for $i> 0$ (i.e. right of the axis)
concentrated in odd total degree where each entry is a sum of copies of
$\FF_p$'s. There cannot be any differentials in this spectral sequence since
the only differentials that are possible for degree reasons would have to map a
$p$-torsion group to $\ZZ$.  Thus the whole spectral sequence collapses at
$E^2$ and is strongly convergent.\footnote{Note that we know the homotopy
groups of $\Sigma \left(\KU_{hC_{p^k}}\right)_{hC_p} = \Sigma
\left(\KU_{hC_{p^{k+1}}}\right)$ by the above computation. It is interesting how
the extensions  work out to lead to this result.} For every value of $j$ there
are elements in total degree 1 that are detected at the
horizontal line through
$-j$. Since this line corresponds to elements in the image of $M_{k,j} \to
\colim_n M_{k,n}$, this finishes the proof. 
\end{proof}

\section{$\THH$ of schemes and the cyclotomic $t$-structure}\label{sec:schemes}

In this section, we give a sampling of applications of the cyclotomic $t$-structure to the study of $\THH$
of commutative rings and to schemes. We will continue an abuse of notation by where we view
$\pi_n^\cyc \THH(X)$ simultaneously as an object of the abelian category $\CartV$ and as a
cyclotomic spectrum (with homotopy groups given by Figure~\ref{fig:homotopygroups} on
page~\pageref{fig:homotopygroups}).

\subsection{$\THH$ of ring spectra and schemes}

Our first result follows from a computations of Hesselholt--Madsen and
Hesselholt.

\begin{theorem}\label{thm:pi0}
    If $R$ is a connective $\EE_1$-ring spectrum, then
    $\pi_0^\cyc\THH(R)\iso\pi_0^\cyc\THH(\pi_0R)\iso W(\pi_0R)$ as a $p$-typical
    Cartier module, where $W(R)$ is equipped with the Witt vector Frobenius and
    Verschiebung maps.
\end{theorem}

\begin{proof}
    Since $\pi_0\THH(R)\iso\pi_0\THH(\pi_0R)\iso\pi_0R$, we see from the long
    exact sequence in cyclotomic homotopy groups that
    $\pi_0^\cyc\THH(R)\iso\pi_0^\cyc\THH(\pi_0R)$. Now, for any associative
    ring $A$, $\pi_0\TR(A)\iso W(A)$. In the commutative case, this is the
    content of~\cite{hesselholt-madsen-1}*{Theorem~F}. In the noncommutative
    case, see~\cite{hesselholt-noncommutative}*{Theorem~2.2.9}
    and~\cite{hesselholt-noncommutative-correction}. A recent exposition is
    given in~\cite{krause-nikolaus}*{Corollary~10.2}. The theorem now follows from Theorem~\ref{mt:td}.
\end{proof}

\begin{remark}
    \begin{enumerate}
        \item[(a)] As explained in Section~\ref{sub:wittmodules}, Theorem~\ref{thm:pi0}
            gives a topological proof of Theorem~\ref{thm:witttensor}, i.e., that $W(R)\widehat\boxtimes
            W(S)\cong W(R\otimes S)$ for rings $R$ and $S$.
        \item[(b)] It follows from Theorem~\ref{thm:pi0} that for any cyclotomic spectrum $X$,
            the cyclotomic homotopy groups $\pi_i^\cyc X$ are modules over $W(\ZZ)$ and
            hence over the cyclotomic spectra $\ZZ^\triv$ and $\THH(\ZZ)$. Indeed, both
            maps $\ZZ^\triv\leftarrow\SS^\triv\rightarrow\THH(\ZZ)$ induce equivalences on
            $\pi_0^\cyc$.
    \end{enumerate}
\end{remark}

Recall from Variant~\ref{var:t} that there is a $t$-structure
on $\CycSp_{\THH(R)}=\Mod_{\THH(R)}(\CycSp_p)$, the $\infty$-category of
$\THH(R)$-modules in cyclotomic spectra. We identify the heart in the next
corollary.

\begin{corollary}
    If $R$ is a connective $\EE_\infty$-ring spectrum, then the heart 
    $\CycSp_{\THH(R)}^\heart$ is equivalent to
    $\Mod_{W(\pi_0R)}(\CycSp_p^{\heart})\we\Mod_{W(\pi_0R)}(\CartV)$, the abelian
    category of $W(\pi_0R)$-modules in derived $V$-complete $p$-typical
    Cartier modules with respect to the $\widehat{\boxtimes}$-symmetric
    monoidal structure.
\end{corollary}

\begin{proof}
    This follows from Theorem~\ref{mainsymm}, Theorem~\ref{thm:pi0}, and Proposition~\ref{prop:permanence}.
\end{proof}

\begin{remark}
    Lemma~\ref{lemma_dieudonne} gives a concrete description of the objects of this abelian category.
\end{remark}

Now, we examine the cyclotomic homotopy groups of schemes in general. In
Section~\ref{sub:charp}, we will give more precise results in the case of
regular $\FF_p$-schemes.

It has been shown in~\cite{bms2}*{Corollary~3.3} that
$\THH(-)$ is a $\Sp$-valued fpqc sheaf on $\CAlg$, the category of commutative rings. They also note
in~\cite{bms2}*{Remark~3.4} that the proof extends to show that
$\THH(-)^{tC_p}$ is a $\Sp$-valued fpqc sheaf on $\CAlg$. This is enough to
prove that in fact $\THH(-)$ is a $\CycSp_p$-valued fpqc sheaf on $\CAlg$. We
expand on this and establish pro-\'etale hyperdescent for $\THH$ in $\CycSp_p$. We claim no
originality in our proof, which closely follows the argument
of~\cite{bms2}.

For details on the pro-\'etale topology, see~\cite{bhatt-scholze-proetale}.
We give a brief summary here. A map $R\rightarrow S$ is {\bf weakly \'etale} if
$R\rightarrow S$ and $S\otimes_RS\rightarrow S$ are both flat. We let
$\CAlg_R^{\proet}$ be the full subcategory of $\CAlg_R$ on the weakly \'etale
$R$-algebras. Note that any map $S\rightarrow T$ of weakly \'etale $R$-algebras
is itself weakly \'etale. The faithful weakly \'etale maps make
$\CAlg_R^{\proet}$ into a site, the pro-\'etale site of $R$. For the next
proof, we will need only to know that for every weakly \'etale map
$R\rightarrow S$ one has $\L_{S/R}\we 0$
(see~\cite{bhatt-scholze-proetale}*{Proposition~2.3.3(2)}, which follows
from~\cite{gabber-ramero}*{Theorem~2.5.36}).

\begin{proposition}\label{prop:hypercomplete}
    Write $\THH(\Oscr)$ for the presheaf which sends a commutative ring $R$
    to $\THH(R)$.
    \begin{enumerate}
        \item[{\rm (a)}] The presheaf $\THH(\Oscr)$ is a hypercomplete
            $\Sp$-valued pro-\'etale sheaf.
        \item[{\rm (b)}] The presheaf $\THH(\Oscr)$ is a hypercomplete
            $\CycSp_p$-valued pro-\'etale sheaf.
    \end{enumerate}
\end{proposition}

\begin{proof}
    One can use the main result of~\cite{mathew-thh}, which says that
    $S\otimes_R\THH(R)\we\THH(S)$ for an \'etale map $R\rightarrow S$, to give
    a proof of part (a). We will give a different proof, which will
    serve to motivate our proof of (b). Let $\CAlg_R^\proet$ be the category of
    weakly \'etale
    $R$-algebras.
    As mentioned above, $\THH(\Oscr)$ is a $\Sp$-valued pro-\'etale sheaf
    by~\cite{bms2}*{Corollary~3.3}. To prove (a), consider the
    presheaf
    $\THH(\Oscr)\otimes_{\THH(\ZZ)}\tau_{\leq n}\THH(\ZZ)$ (where $\tau_{\leq n}\THH(\ZZ)$
    is the ordinary (not cyclotomic) truncation), which sends a
    $S\in\CAlg_R^\proet$ to $\THH(S)\otimes_{\THH(\ZZ)}\tau_{\leq n}\THH(\ZZ)$. 
    For each $S\in\CAlg_R^\proet$, we have that
    $\THH(S)\we\lim_n\THH(S)\otimes_{\THH(\ZZ)}\tau_{\leq n}\THH(\ZZ)$.
    In particular,
    $\THH(\Oscr)\we\lim_n\THH(\Oscr_X)\otimes_{\THH(\ZZ)}\tau_{\leq
    n}\THH(\ZZ)$ as $\Sp$-valued presheaves. Since hypercomplete sheaves are
    closed under limits, it is enough to see that each
    $\THH(\Oscr)\otimes_{\THH(\ZZ)}\tau_{\leq n}\THH(\ZZ)$ is hypercomplete.
    As in the proof of~\cite{bms2}*{Corollary~3.3}, this reduces to showing
    that each $\THH(\Oscr)\otimes_{\THH(\ZZ)}\pi_n\THH(\ZZ)$ is hypercomplete
    and then to showing that $\THH(\Oscr)\otimes_{\THH(\ZZ)}\ZZ$ is
    hypercomplete, since
    $$\THH(\Oscr)\otimes_{\THH(\ZZ)}\pi_n\THH(\ZZ)\we\THH(\Oscr)\otimes_{\THH(\ZZ)}\ZZ\otimes_\ZZ\pi_n\THH(\ZZ).$$
    However, $\THH(\Oscr)\otimes_{\THH(\ZZ)}\ZZ\we\HH(\Oscr/\ZZ)$, the
    presheaf which sends $S\in\CAlg_R^\proet$ to $\HH(S/\ZZ)$, the Hochschild
    homology of $S$. Thus, we
    are reduced to proving that Hochschild homology is hypercomplete as a
    $\Sp$-valued pro-\'etale presheaf. Now, $\HH(\Oscr/\ZZ)$ admits a complete decreasing
    $\NN$-indexed filtration $\F^\star_{\HKR}\HH(\Oscr/\ZZ)$ with graded
    pieces given by
    $$\gr^n_{\HKR}\HH(\Oscr/\ZZ)\we\Lambda^n\L_{\Oscr/\ZZ}[n].$$
    In particular, $\gr^0_{\HKR}\HH(\Oscr/\ZZ)\we\Oscr$, which is
    hypercomplete. Since each $\Lambda^n\L_{\Oscr/\ZZ}$ is quasi-coherent (we
    use here that $\L_{S/R}\we 0$ for $R\rightarrow S$ pro-\'etale) and
    since quasi-coherent sheaves are hypercomplete (see for
    example~\cite{sag}*{Proposition~2.2.6.1}), it follows by induction that each
    $\F^n_{\HKR}\HH(\Oscr/\ZZ)$ is a hypercomplete pro-\'etale sheaf of spectra.
    Part (a) follows since the HKR filtration is complete.

    Now, we would like to prove that $\THH(\Oscr)$ is a hypercomplete $\CycSp_p$-valued
    pro-\'etale sheaf. Since the forgetful functor $\CycSp_p\rightarrow\Sp$
    does not commute with limits in general, this is not an immediate
    consequence of the previous paragraph.
    We
    invoke~\cite{nikolaus-scholze}*{Proposition~II.1.5}, which implies that it
    is enough for
    $\THH(R)^{tC_p}\rightarrow\lim_{\Delta}\THH(S^\bullet)^{tC_p}$ to be an
    equivalence for every pro-\'etale hypercover $R\rightarrow S^\bullet$.
    Since $(-)^{tC_p}$ commutes with sequential limits of increasingly
    connected maps (by Lemma~\ref{lem:littlelimits}(a)), we can reduce as in the
    previous paragraph to checking that
    $\HH(R)^{tC_p}\rightarrow\lim_{\Delta}\HH(S^\bullet)^{tC_p}$ is an
    equivalence. And, again, we can use the HKR filtration to reduce to checking that
    $(\Lambda^n\L_{R/\ZZ})^{tC_p}\we\lim_\Delta(\Lambda^n\L_{S^\bullet/\ZZ})^{tC_p}$
    is an equivalence for each $n\in\NN$. Since $(-)^{hC_p}$ commutes with limits, it is enough to
    prove that
    $(\Lambda^n\L_{R/\ZZ})_{hC_p}\rightarrow\lim_\Delta(\Lambda^n\L_{S^\bullet/\ZZ})_{hC_p}$
    is an equivalence. This follows from faithfully flat descent since $(\L_{R/\ZZ})_{hC_p}$ is
    quasicoherent.
\end{proof}

\begin{corollary}
    The presheaf $\TR(\Oscr)$ which sends a commutative ring $R$ to $\TR(R)$ is a
    hypercomplete $\TCart_p$-valued pro-\'etale sheaf.
\end{corollary}

\begin{proof}
    This is immediate as $\TR\colon\CycSp_p\rightarrow\TCart_p$ preserves limits
    by Proposition~\ref{thm:tr}.
\end{proof}

It follows that $\THH(\Oscr)$ and $\TR(\Oscr)$ are hypercomplete for any topology coarser
than the pro-\'etale topology as well. We will be mainly interested in the Zariski
topology below.

\begin{definition}
    Let $X$ be a quasi-compact and quasi-separated scheme.
    Let $\tau_{\leq n}\TR(\Oscr_X)$ denote the Zariski sheafification
    of $U\mapsto\tau_{\leq n}\TR(U)$ in the $\infty$-category
    $\TCart_p$. Similarly, let $\tau_{\geq n}\TR(\Oscr_X)$ denote the
    Zariski sheafification of $U\mapsto\tau_{\geq n}\TR(U)$.
\end{definition}

\begin{remark}
    \begin{itemize}
        \item[(a)] The forgetful functor $\TCart_p\rightarrow\Sp$ preserves limits and
            colimits, so in what follows we could as well work with presheaves of spectra.
        \item[(b)] In particular, $\tau_{\leq n}\TR(\Oscr_X)$ and $\tau_{\geq
            n}\TR(\Oscr_X)$ are both hypercomplete since the inclusion of hypercomplete
            sheaves of spectra into all sheaves of spectra is $t$-exact.\footnote{Let $\Xscr$ be an
            $\infty$-topos, $\Xscr^{\mathrm{hyp}}$ its hypercompletion, and consider the
            adjunction
            $f^*\colon\Shv_{\Sp}(\Xscr)\rightleftarrows\Shv_{\Sp}(\Xscr^{\mathrm{hyp}})\colon
            f_*$. As for general geometric morphisms of $\infty$-topoi, $f^*$ is
            $t$-exact and $f_*$ is left $t$-exact (see~\cite{sag}*{1.3.2.8}). We want
            to show that $f_*$ is right $t$-exact as well in this case. Thus, let
            $\Fscr\in\Shv_{\Sp}(\Xscr^{\mathrm{hyp}})_{\geq 0}$ and consider the cofiber
            sequence $$\tau_{\geq 0}f_*\Fscr\rightarrow f_*\Fscr\rightarrow\tau_{\leq
            -1}f_*\Fscr.$$ Applying $f^*$ again, we
            see that $$f^*\tau_{\geq 0}f_*\Fscr\we\tau_{\geq
            0}f^*f_*\Fscr\we\tau_{\geq 0}\Fscr\we\Fscr$$ using the $t$-exactness of
            $f^*$ and the fully faithfulness
            of $f_*$ (see~\cite{sag}*{1.3.3.2}). Thus, $f^*\tau_{\leq
            -1}f_*\Fscr\we 0$. In other words, $\tau_{\leq -1}f_*\Fscr$ is both
            $\infty$-connective and bounded above. Hence, since the $t$-structure on
            $\Shv_{\Sp}(\Xscr)$ is right complete by~\cite{sag}*{1.3.2.7}, $\tau_{\leq -1}f_*\Fscr\we 0$
            so $f_*\Fscr$ is connective.}
    \end{itemize}
\end{remark}

\begin{lemma}\label{lem:hypercomplete}
    If $X$ is a quasi-compact scheme of finite Krull dimension, then
    the natural map $\TR(\Oscr_X)\rightarrow\lim_n\tau_{\leq n}\TR(\Oscr_X)$ is an
    equivalence.
\end{lemma}

\begin{proof}
    In this case, the $\infty$-topos of sheaves of spaces on the Zariski site
    of $X$ has finite homotopy dimension by~\cite[Theorem~3.17]{clausen-mathew}.
    This implies that every sheaf of spaces on the Zariski site is Postnikov
    complete by~\cite[Proposition~7.2.1.10]{htt}. By taking infinite loop
    spaces, this implies that every connective sheaf of spectra on the Zariski
    site of $X$ is Postnikov complete.\footnote{Note that the $\infty$-category of
    connective sheaves of spectra is equivalent to the $\infty$-category of
    sheaves of connective spectra.}
\end{proof}

\begin{definition}
    Write $\pi_n^\cyc\THH(\Oscr_X)$ for the Zariski sheafification of
    $U\mapsto\pi_n^\cyc\THH(U)\iso\pi_n\TR(U)$, which is naturally a sheaf in the abelian
    category $\Cart_p$.
\end{definition}

\begin{example}
    By Theorem~\ref{thm:pi0}, we see that $\pi_0^\cyc\THH(\Oscr_X)\we W(\Oscr_X)$,
    Serre's Witt vector sheaf.
\end{example}

\begin{proposition}\label{prop:descentss}
    Let $X$ be a quasi-separated scheme of finite Krull dimension.
    There is a conditionally convergent Zariski descent spectral sequence
    $$\E_2^{s,t}\iso\H^{-s}(X,\pi_t^\cyc\THH(\Oscr_X))\Rightarrow\pi_{s+t}^\cyc\THH(X)$$
    in $\Cart_p$.
\end{proposition}

\begin{proof}
    This is the spectral sequence associated to the complete decreasing
    $\NN$-indexed Whitehead tower $\tau_{\geq t}\TR(\Oscr_X)$ for $\TR(\Oscr_X)$.
    We use that $\pi_{t-s}^\cyc\THH(X)\iso\pi_{t-s}\TR(X)$.
\end{proof}

\begin{example}
    Suppose that $X$ has Krull dimension $d$.
    The differentials $d_r$ in the descent spectral sequence have bidegree $(r,r-1)$, from which we see that the
    bottom cyclotomic homotopy group of $\THH(X)$ is given by
    $$\pi_{-d}^\cyc\THH(X)\iso\H^d(X,\pi_0^\cyc\THH(\Oscr_X))\iso\H^d(X,W(\Oscr_X)),$$
    where $W(\Oscr_X)$ is Serre's sheaf of Witt vectors.
\end{example}

\subsection{$\THH$ of smooth schemes over perfect rings in characteristic $p$}\label{sub:charp}

The following theorem follows easily from the description of the $t$-structure
in terms of $\TR$ and the work of Hesselholt and Madsen.

\begin{theorem}\label{thm:discrete}
    Suppose that $k$ is a perfect ring of characteristic $p$. Then,
    $\THH(k)\in\CycSp_p^\heart$. The associated Cartier module
    $\pi_0^\cyc\THH(k)$ is isomorphic to $W(k)$ with the Witt vector Frobenius
    $F$ and Verschiebung $V$.
\end{theorem}

\begin{proof}
    By Theorem~\ref{thm:boundedbelow}, it suffices to show that $\TR(R)$ is concentrated in degree zero and that
    $\TR_0(k)\iso W(k)$. When $k$ is a perfect field of characteristic $p$, this is the content
    of~\cite{hesselholt-madsen-1}*{Theorem~5.5}. When $k$ is an arbitrary
    perfect ring of characteristic $p$, there is a map $\THH(k)\rightarrow\pi_0^\cyc\THH(k)\we
    W(k)/V$ of commutative algebra objects in $\CycSp_p$ by Theorem~\ref{thm:pi0}. We already know that $\pi_*\THH(k)\iso
    k[b]$ where $b$ has degree $2$ (for instance by~\cite{bms2}*{Theorem~6.1}). Since $k$ is perfect, $\pi_* W(k)/V\iso
    k[b]$ as well. Thus, it is enough to see that $b$ maps to $b$ up to a unit.
    But, this follows by the commutative diagram
    $$\xymatrix{\THH(\FF_p)\ar[r]\ar[d]&\THH(k)\ar[d]\\
    W(\FF_p)/V\ar[r]&W(k)/V}$$ and the fact that we know the result when
    $k=\FF_p$.
\end{proof}

\begin{remark}
    \begin{enumerate}
        \item[(i)] The fact that $\THH(\FF_p)\in\CycSp_p^\heart$ is equivalent
            to B\"okstedt's original calculation via Theorem~\ref{thm:pi0}.
            We have already seen one direction as B\"okstedt's calculation is
            used to prove~\cite{hesselholt-madsen-1}*{Theorem~5.5}. So, assume
            that $\THH(\FF_p)\in\CycSp_p^\heart$. By Theorem~\ref{thm:pi0}, we
            find that in fact $\THH(\FF_p)\we W(\FF_p)/V$. But, additively,
            $\pi_*W(\FF_p)/V\iso\FF_p[b]$ where $|b|=2$ using
            Figure~\ref{fig:homotopygroups} on
            page~\pageref{fig:homotopygroups}. Thus, it is enough to determine
            the multiplicative structure. For this, we use that $\THH(\FF_p)$
            is a $\ZZ_p^\triv$-module in cyclotomic spectra as the trace map
            $\K(\FF_p)\rightarrow\THH(\FF_p)$ factors through $\tau_{\geq
            0}\TC(\FF_p)\we\ZZ_p$. Thus, we obtain a commutative diagram
            $$\xymatrix{
            \ZZ_p^\triv\ar[r]\ar[d] &   \THH(\FF_p)\ar[r]\ar[d] & W(\FF_p)/V\ar[d]\\
            \ZZ_p^{tC_p}\ar[r]&\THH(\FF_p)^{tC_p}\ar[r]&(W(\FF_p)/V)^{tC_p}}$$
            of commutative algebra objects in $\Sp^{BS^1}$. We know that
            $\ZZ_p^{tC_p}\we W(\FF_p)^{tC_p}\rightarrow (W(\FF_p)/V)^{tC_p}$ is an equivalence
            by the Tate orbit lemma. Thus, the composition of the bottom arrows
            is an equivalence. Since we are assuming that $\THH(\FF_p)\we
            W(\FF_p)/V$, this means that both bottom arrows are equivalences.
            On the other hand, this means that the middle and right vertical
            arrows are both $0$-truncated by Proposition~\ref{prop:segalcondition}.
            Since $\pi_1\ZZ_p^{tC_p}=0$ and because the map
            $\THH(\FF_p)\rightarrow\THH(\FF_p)^{tC_p}$ is a ring map, we see
            that in fact $\THH(\FF_p)\we\tau_{\geq
            0}\THH(\FF_p)^{tC_p}\we\tau_{\geq 0}\ZZ_p^{tC_p}$. This shows that
            $\pi_*\THH(\FF_p)\we\FF_p[b]$ multiplicatively as
            well.
        \item[(ii)]
            Our original proof, prior to the discovery of $p$-typical topological
            Cartier modules, used only the fact that $\pi_*\THH(k)\iso
            k[b]$, where $b$ has degree $2$ when $k$ is a perfect ring of
            characteristic $p$. Indeed, one can construct a tower of
            spectra with $S^1$-action whose limit is $\tau_{\geq 1}^\cyc X$ for any
            $p$-typical cyclotomic spectrum $X$. Studying this tower for $X\we\THH(k)$
            where $k$ is a perfect ring of characteristic $p$ gives an alternative proof of
            Theorem~\ref{thm:discrete}.
    \end{enumerate}
\end{remark}

\begin{remark}\label{rem:dieudonne}
    \begin{enumerate}
        \item[(a)]
            Recall from Example~\ref{dieudonne} that if $R=\FF_p$, then the
            objects of $\Mod_{W(\FF_p)}(\TCart_p)$ are called Dieudonn\'e
            modules. Thus, $\CycSp_{\THH(\FF_p)}^\heart$ is the abelian category of derived
            $V$-complete Dieudonn\'e modules.
        \item[(b)]
            Since $\THH(\FF_p)$ is a $\ZZ^\triv$-module in $\CycSp_p$, it follows that
            $\eta=0$ on $\pi_n^\cyc X$ for any $\THH(\FF_p)$-module $X$ in $p$-typical
            cyclotomic spectra.
    \end{enumerate}
\end{remark}

\begin{definition}\label{def:dieudonne}
    A {\bf Dieudonn\'e complex} is a $p$-typical Cartier complex (see
    Definition~\ref{def_Cartier}) with $VF=p$ and $\eta=0$. Saturated Dieudonn\'e complexes in the sense
    of~\cite{blm1}*{Definition~2.2.1} naturally admit the structure of
    Dieudonn\'e complexes in our sense by~\cite{blm1}*{Proposition~2.2.4}. 
\end{definition}

Combining the previous remarks, Theorem~\ref{thm:discrete}, and
Lemma~\ref{lem:cartiercomplexes}, we see that for
any $X\in\CycSp_{\THH(\FF_p)}$, the graded abelian group $\pi_*^\cyc X$
naturally admits the structure of a Dieudonn\'e complex.

\begin{theorem}\label{thm:derhamwitt}
    Let $k$ be a perfect field of characteristic $p$. If $R$ is
    a smooth $k$-algebra, then $\pi_\ast^\cyc\THH(R)\iso\W\Omega^\ast_R$, as Dieudonn\'e
    complexes, where $\W\Omega^\ast_R$ is the de Rham--Witt complex.
    In particular, if $R$ has relative dimension at most $d$ over $k$, then
    $\THH(R)\in(\CycSp_p)_{[0,d]}$.
\end{theorem}

\begin{proof}
    At the level of $\TR(R)$, this is the content
    of~\cite{hesselholt-ptypical}*{Theorem~C}, which says that, with the $F$,
    $V$, and $d$ operations, $\TR_*(R)\iso\W\Omega^*_R$.
\end{proof}

\begin{remark}
    In particular, we find that $\W\Omega^n_R$ is derived $V$-complete for
    all $n$. Since this is an unusual filtration to consider, we note that this
    is easy to deduce classically. Let $p^\star\W\Omega^n_R$ be the $p$-adic
    filtration on $\W\Omega^n_R$ and let
    $V^\star\W\Omega^n_R+dV^\star\W\Omega^n_R\subseteq\W\Omega^n_R$ be the
    submodule generated by the images of $V^\star$ and $dV^\star$. Then,
    $p^\star\W\Omega^n_R\subseteq
    V^\star\W\Omega^n_R\subseteq
    V^\star\W\Omega^n_R+dV^\star\W\Omega^{n-1}_R$. Since $\W\Omega^n_R$ is complete with
    respect to $p^\star\W\Omega^n_R$ and to
    $V^\star\W\Omega^n_R+dV^\star\W\Omega^{n-1}_R$, it is complete with respect to
    $V^\star\W\Omega^n_R$. Since $V$ is injective, this is the same as derived
    $V$-completeness. By Proposition~\ref{prop:derivedcanonical}, $\TR_*(R)$ is 
    also complete with respect to the derived version of the
    $(V+dV)$-filtration. In fact, in this case, the derived and
    non-derived filtrations are pro-equivalent.
\end{remark}

\begin{corollary}\label{cor:derhamwitt2}
    Let $k$ be a perfect field of characteristic $p$ and let $S$ be an
    ind-smooth $k$-algebra. Then, $\TR_*(S)\iso\W\Omega^*_S$.
\end{corollary}

\begin{proof}
    Because $\TR$ does not commute with filtered colimits as a functor to
    $\TCart_p$, this theorem is not a
    formal consequence of Theorem~\ref{thm:derhamwitt}. To correct this,
    we need to dig into the proof of~\cite{hesselholt-ptypical}*{Theorem~C}.
    Hesselholt's result follows from a finer
    result~\cite{hesselholt-ptypical}*{Theorem~B}, which says that if $R$ is
    a smooth $k$-algebra, then
    $\pi_*\THH(R)^{C_{p^{n-1}}}\iso\W_n\Omega^*_R[b_n]$ and that moreover the map
    $R\colon\THH(R)^{C_{p^{n-1}}}\rightarrow\THH(R)^{C_{p^{n-2}}}$ sends $b_n$ to
    $pb_{n-1}$ (up to a unit). Passing up the tower, one obtains the
    computation of $\TR$.

    Now, each $\THH(-)^{C_{p^{n-1}}}$ commutes with sifted colimits of
    commutative rings and in
    particular filtered colimits. Suppose that $S\we\colim_{i\in I}S_i$ where
    $I$ is a filtered category and each $S_i$ is a smooth $k$-algebra.
    Then, we have
    $$\pi_*\THH(S)^{C_{p^{n-1}}}\iso\colim_i\pi_*\THH(S_i)^{C_{p^{n-1}}}\iso\colim_i\W_n\Omega_{S_i}^*[b_n]\iso\W_n\Omega_S^*[b_n]$$
    for each $n\geq 1$, as $\W_n\Omega_-^*$ commutes with filtered colimits
    (see~\cite{illusie-derham-witt}*{I.1.10}).
    We still have that
    $R\colon\THH(S)^{C_{p^{n-1}}}\rightarrow\THH(S)^{C_{p^{n-2}}}$ sends $b_n$ to $pb_{n-1}$ up to a unit, so taking the limit gives the result.
\end{proof}

\begin{example}
    \begin{enumerate}
        \item[(1)] By Popescu's theorem, every regular noetherian $\FF_p$-algebra is a
            filtered colimit of smooth $\FF_p$-algebras. So, the theorem applies in
            particular to all regular noetherian $\FF_p$-algebras.
        \item[(2)] We can say somewhat more in the special case when $R$ is a
            filtered colimit of smooth $k$-algebras of a uniformly bounded
            dimension. For example, suppose that $K=k(x_1,\ldots,x_d)$. Then,
            $\THH(K)\in(\CycSp_p)_{[0,d]}$. Indeed, in that case, each
            $\W_n\Omega^\ast_K$ is concentrated in degrees $0$ to $k$ and hence so is the limit.
        \item[(3)] If $R$ is a smooth $\FF_p$-algebra, then the theorem applies
            pro-\'etale locally on the pro-\'etale site of $R$. Indeed, for
            every weakly \'etale map $R\rightarrow S$ there is a faithfully
            flat ind-\'etale
            map $S\rightarrow T$ such that $R\rightarrow T$ is ind-\'etale. In
            particular, $k\rightarrow T$ is ind-smooth.
            See~\cite{bhatt-scholze-proetale}*{Theorem~2.3.4}. We will use this
            below to compare the BMS filtration on $\TC$ with the filtration
            coming from the cyclotomic $t$-structure.
    \end{enumerate}
\end{example}

\begin{corollary}
    Let $k$ be a perfect field of characteristic $p$ and let $X$ be a smooth
    quasi-compact $k$-scheme. There is a convergent spectral sequence
    $$\E_2^{s,t}=\H^s(X,\W\Omega^t_X)\Rightarrow\pi_{t-s}^\cyc\THH(X)$$
    in the abelian category of derived $V$-complete $p$-typical Cartier
    modules. If $X$ has Krull dimension at most $d$,
    then $\E_2^{s,t}=0$ for $t>d$ and $s>d$.
\end{corollary}

\begin{proof}
    Combine Proposition~\ref{prop:descentss} and Theorem~\ref{thm:derhamwitt}. In this
    case, the spectral sequence convergences because it collapses at some
    finite stage for each connected component of $X$ (above the
    dimension of the connected components).
\end{proof}

Now, we can compare the BMS filtration and the cyclotomic $t$-structure filtrations on
$\TC(R)$ when $X=\Spec R$ is a smooth affine scheme over a perfect field $k$.
The BMS filtration is a decreasing multiplicative $\NN$-indexed filtration
$\F^\star_{\BMS}\TC(R)$ on $\TC(R)$ with graded pieces given by
$\gr^n_\BMS\TC(R)\we\ZZ_p(n)(R)[2n]\we\W\Omega^n_{X,\mathrm{log}}(R)[n]$
by~\cite{bms2}*{Theorems~1.12 and~1.15}. It is defined in the following way. If
$S$ is a quasiregular semiperfect $\FF_p$-algebra, then $\TC^-(S)$ and $\TP(S)$
are concentrated in even degrees. One sets $\F^\star_{\BMS}\TC^-(S)=\tau_{\geq
2\star}\TC^-(S)$ and $\F^\star_{\BMS}\TP(S)=\tau_{\geq 2\star}\TP(S)$. The canonical
and Frobenius maps $\TC^-(S)\rightarrow\TP(S)$ preserve the BMS filtration
on each side and hence define a filtration $\F^\star_{\BMS}\TC(S)$. One now
constructs the BMS filtration on a general quasisyntomic $\FF_p$-algebra by
quasisyntomic descent.

Note that the filtrations $\F^\star_{\BMS}\TC^-(R)$ and
$\F^\star_{\BMS}\TP(R)$ do not come from a cyclotomic filtration on $\THH(R)$.
Indeed, if $R=\FF_p$ then $\gr^n_{\BMS}\TP(\FF_p)\we\ZZ_p[2n]$ whereas $\TP$ of
any cyclotomic $\THH(\FF_p)$-module will be $2$-periodic. However, we will see
that if $R$ is smooth, then the filtration on $\TC(R)$ does come from a
cyclotomic filtration on $\THH(R)$.

Consider the decreasing multiplicative $\NN$-indexed
filtration $\TC(\tau_{\geq\star}^\cyc\THH(R))$ on $\TC(R)$. Since
$$\TC(\tau_{\geq\star}^\cyc\THH(R))\we\fib\left(\tau_{\geq\star}\TR(R)\xrightarrow{1-F}\tau_{\geq\star}\TR(R)\right),$$
we see from Theorem~\ref{thm:derhamwitt} that the graded pieces are given by
$$\fib\left(\W\Omega^n_R[n]\xrightarrow{1-F}\W\Omega_R^n[n]\right)\we\fib\left(\W\Omega^n_R\xrightarrow{1-F}\W\Omega^n_R\right)[n].$$

\begin{theorem}\label{thm:syntomic}
    Let $X=\Spec R$ be a smooth affine scheme over a perfect field $k$ of
    characteristic $p$. There is a natural equivalence
    $\TC(\tau_{\geq\star}^\cyc\THH(R))\rightarrow\F^\star_{\BMS}\TC(R)$
    of filtered spectra.
\end{theorem}

\begin{proof}
    The BMS filtration satisfies pro-\'etale descent, by definition, since it
    satisfies quasisyntomic descent.
    In other words, the assignment which sends an pro-\'etale $R$-algebra $S$ to
    $\F^n_{\BMS}\TC(S)$ is a pro-\'etale sheaf. In fact, it is hypercomplete as it
    is bounded above on $\CAlg_R^{\proet}$. Indeed, if $R$ has Krull dimension
    $d$, then for any pro-\'etale $R$-algebra $S$, $\F^n_{\BMS}\TC(S)$ is
    $d$-truncated. Moreover, pro-\'etale-locally, the BMS
    filtration reduces to the Whitehead tower on $\TC$
    by~\cite{bms2}*{Theorem~1.15}. Thus, to map in to $\F^\star_\BMS\TC(R)$ it
    is enough to prove that pro-\'etale locally $\TC(\tau_{\geq n}^\cyc\THH(R))$ is
    $n$-connective. For this, we need to know that
    $\W\Omega^n_X\xrightarrow{1-F}\W\Omega^n_X$ is pro-\'etale-locally
    surjective. As shown in the proof of~\cite{bms2}*{Proposition~8.4}, this follows from a result of
    Illusie~\cite{illusie-derham-witt}*{Th\'eor\`em~I.5.7.2}. Moreover, they
    show that the kernel (in pro-\'etale sheaves) is $\W\Omega^n_{X,\log}$.
    This proves both that there is a map of filtered objects
    $\TC(\tau_{\geq\star}^\cyc\THH(R))\rightarrow\F^\star_{\BMS}\TC(R)$ and
    that the map on graded pieces is the natural map
    $$\fib\left(\W\Omega^n_R\xrightarrow{1-F}\W\Omega^n_R\right)\rightarrow\W\Omega^n_{X,\mathrm{log}}(R),$$
    which is an equivalence (see for
    example~\cite{geisser-hesselholt-1}*{Lemma~4.1.3} and the following discussion).
\end{proof}

\subsection{Crystalline cohomology and $\TP$}

\newcommand{\DF}{\mathrm{DF}}
\newcommand{\DBF}{\mathrm{DBF}}

Let $k$ be a perfect ring of characteristic $p$.
Let $R$ be a smooth commutative $k$-algebra. We show in this section how to use
Theorem~\ref{thm:derhamwitt} to extract the de Rham--Witt complex from $\TP(R)$.

Recall that when $R$ is smooth over a perfect field $k$, then the crystalline cohomology
$\R\Gamma_{\crys}(R/W(k))$ has a canonical cochain complex model given by the de
Rham--Witt complex
$$0\rightarrow\W\Omega^0_R\rightarrow\W\Omega^1_R\rightarrow\cdots,$$
which is a Dieudonn\'e complex in the sense of Definition~\ref{def:dieudonne}.

Our next theorem complements a result of Bhatt, Morrow, and Scholze. They prove
in~\cite{bms2}*{Theorem~1.10} that if $R$ is a smooth $k$-algebra for a perfect
field $k$, then there is a filtration $\F^\star_{\BMS}\TP(R)$ with graded
pieces $$\gr^n_\BMS\TP(R)\we\R\Gamma_{\crys}(R/W(k))[2n].$$ Moreover, each
$\F^n_{\BMS}\TP(R)$ is itself equipped with a filtration and the maps
$\F^{n+1}_\BMS\TP(R)\rightarrow\F^n_\BMS\TP(R)$ are compatible with this
secondary filtration. On graded pieces, one obtains the Nygaard filtration on
$\R\Gamma_\crys(R/W(k))[2n]$.

We recover the BMS filtration on $\TP(R)$ in the next theorem by using the
cyclotomic $t$-structure. Our filtration also comes equipped with a secondary
filtration as well, but this time the induced filtration on the graded pieces is the
Hodge, or Hodge--Witt, filtration, which lets us say that the graded pieces are given by
shifts of the de Rham--Witt complex, as opposed to an object of the derived
$\infty$-category equivalent to the de Rham--Witt complex.

\begin{theorem}\label{thm:drw}
    Let $k$ be a perfect field of characteristic $p$ and let $R$ be an ind-smooth $k$-algebra and let
    $\TP(\tau_{\geq\star}^\cyc\THH(R))$ be the filtration on $\TP(R)$ induced by the
    cyclotomic Whitehead tower. The induced Whitehead tower with respect to the
    Beilinson $t$-structure on filtered spectra defines
    a natural complete exhaustive multiplicative decreasing
    $\ZZ$-indexed filtration
    $\F^\star_\B\TP(R)$ in filtered spectra with graded pieces given by
    $$\gr^i_\B\TP(R)\we\W\Omega_R^\bullet[2i]$$ for all $i$. Moreover, this
    filtration agrees with the BMS filtration on $\TP(R)$ after forgetting the
    secondary filtrations.
\end{theorem}

\begin{proof}
    Applying the $S^1$-Tate construction to the cyclotomic Whitehead tower of
    $\THH(R)$, we obtain a complete decreasing $\NN$-indexed filtration
    $(\tau_{\geq\star}^\cyc\THH(R))^{tS^1}$ on $\TP(R)$.
    The graded pieces are
    $$(\pi_\star^\cyc\THH(R)[\star])^{tS^1}\we(\W\Omega^\star_R[\star])^{tS^1}$$
    by Corollaries~\ref{mt:tate} and~\ref{cor:derhamwitt2}.
    We let $\F^\star_\B\TP(R)$ denote the double-speed Whitehead tower with respect
    to the Beilinson $t$-structure (see~\cite{bms2}*{Section~5}) on filtered spectra
    associated to the filtered spectrum $(\tau_{\geq\star}^\cyc\THH(R))^{tS^1}$.
    This is a complete exhaustive $\ZZ$-indexed filtration on $\TP(R)$
    (see for example~\cite{antieau-derham}*{Lemma~3.2}). Moreover,
    $\gr^n_\B\TP(R)[-2n]$ is an object of the heart of
    $\Fun(\ZZ^\op,\Dscr(W(\FF_p)))$, the
    $\infty$-category of filtered objects of the derived $\infty$-category of
    $W(\FF_p)$. Thus, $\gr^n_\B\TP(R)[-2n]$ is canonically a cochain complex; it
    is not hard to see that it is of the form
    $$0\rightarrow\W\Omega_R^0\rightarrow\W\Omega_R^1\rightarrow\cdots,$$
    where the differential is Connes' $B$-operator. By Theorem~\ref{thm:derhamwitt},
    the $B$-operator is given by the differential in the de Rham--Witt complex.
    This completes the proof of the first part of the theorem. For more details, see~\cite{antieau-derham}*{Example~2.4}.

    To continue, we must produce a map
    $\F^\star_\B\TP(R)\rightarrow\F^\star_{\mathrm{BMS}}\TP(R)$ of filtered
    spectra. To do so, we first Kan extend so that we can compare in the
    quasiregular semiperfect case, and then we descend back to the smooth case.

    As in the first paragraph, 
    if $R$ is smooth over $\FF_p$, we can produce a complete filtration
    $\F^\star_\B\TR(R)_{hS^1}$ on $\TR(R)_{hS^1}$ with graded pieces
    $\gr^n_\B\TR(R)_{hS^1}\we\W\Omega_R^{\bullet\leq n}[2n]$. The natural map
    $\TP(R)\rightarrow\TR(R)_{hS^1}[2]$ induces a filtered map
    $\F^\star_\B\TP(R)\rightarrow\F^\star_\B\TR(R)_{hS^1}[2]$ which on graded pieces
    is the natural quotient
    $\W\Omega_R^{\bullet}[2n]\rightarrow\W\Omega^{\bullet\leq n-1}_R[2n]$ of chain complexes.

    Fix a generator $v\in\pi_{-2}\ZZ^{hS^1}\iso\ZZ$. Recall that if
    $X\in\Fun(BS^1,\D(\ZZ))$ is a chain complex with $S^1$-action, then
    $X_{hS^1}$ is a module over $\ZZ^{hS^1}$ and the natural map
    $X^{tS^1}\rightarrow X_{hS^1}[2]$ induces an equivalence
    $$X^{tS^1}\we\lim\left(\cdots\rightarrow
    X_{hS^1}[2n-2]\xrightarrow{v}X_{hS^1}[2n]\xrightarrow{v}X_{hS^1}[2n+2]\rightarrow\cdots\right).$$
    In the case of $\TR(R)_{hS^1}$, $v$ induces a filtered map
    $\F^\star_\B\TR(R)_{hS^1}\rightarrow\F^{\star-1}_\B\TR(R)_{hS^1}[2]$
    which on graded pieces identifies with the quotient map
    $\W\Omega^{\bullet\leq n}[2n]\rightarrow\W\Omega^{\bullet\leq n-1}[2n]$.

    For a general object $X\in\Mod_{\ZZ^{hS^1}}$, we let
    $\T_vX=\lim\left(\cdots\rightarrow
    X[2n-2]\xrightarrow{v}X[2n]\xrightarrow{v}X[2n+2]\rightarrow\cdots\right)$.
    This construction defines a functor
    $\Mod_{\ZZ^{hS^1}}\rightarrow\Mod_{\ZZ^{tS^1}}$.

    Let $\L\TP(-)$ denote the left Kan
    extension of $\TP(-)$ from polynomial $\FF_p$-algebras to all simplicial
    commutative $\FF_p$-algebras as a functor with values in
    $\Mod_{\ZZ^{tS^1}}$. We let $\L\TR(-)$ denote the left Kan extension of
    $\TR$ as a functor with values in $\Fun(BS^1,\D(\ZZ))$. We let
    $\L(\TR(-)_{hS^1})$ be the left Kan extension of $\TR(-)_{hS^1}$ as a functor
    with values in $\D(\ZZ^{hS^1})$. Since taking homotopy orbits commutes with
    colimits, there is an equivalence of functors
    $(\L\TR(-))_{hS^1}\we\L(\TR(-)_{hS^1})$. We write $\L\TR(-)_{hS^1}$ for this common
    functor.

    Consider the commutative diagram
    $$\xymatrix{
        \L\TP(S)\ar[r]\ar[d]&\T_v\L\TR(S)_{hS^1}\ar[d]\\
        \TP(S)\ar@{=}[r]&\T_v\TR(S)_{hS^1}
    }$$
    in $\Mod_{\ZZ^{tS^1}}$. The spectrum $\L\TP(S)$ carries the Kan extended
    Beilinson filtration $\F^\star_\B\L\TP(S)$ with graded pieces
    $\L\W\Omega_S$. The spectrum $\T_v\TR(S)_{hS^1}$ carries the inverse limit
    of the Kan extended Beilinson filtration on $\TR(S)_{hS^1}$. This is a
    filtration $\F^\star_\B\T_v\L\TR(S)_{hS^1}$ with graded pieces given by
    $\widehat{\L\W\Omega}_S$, the completion of $\L\W\Omega_S$ with respect to
    the Hodge filtration $\L\W\Omega^{\geq\star}_S$.

    The Kan extended filtration $\F^\star\TR(S)_{hS^1}$ is complete. Indeed, if
    $S$ is smooth, $\gr^n_\B\TR(S)_{hS^1}\we\W\Omega^{\leq n}[2n]$ is
    in $\D(\ZZ)_{[n,2n]}$. Thus, since the filtration is complete in this case,
    $\F^n_\B\TR(S)_{hS^1}$ is in $\D(\ZZ)_{\geq n}$. Hence, the Kan extension
    $\F^n\L\TR(S)_{hS^1}$ is in $\D(\ZZ)_{\geq n}$. In particular, the Kan
    extended Beilinson filtration is complete on $\TR(S)_{hS^1}$ for any $S$.
    Since $\F^\star_\B\T_v\TR(S)_{hS^1}$ is an inverse limit of complete
    filtrations, it is complete.

    It follows that the map
    $\F^\star_\B\L\TP(S)\rightarrow\F^\star_\B\T_v\L\TR(S)_{hS^1}$ factors
    through the completion $\widehat{\L\TP}(S)$ of $\L\TP(S)$ with respect to
    the Beilinson filtration. On $\widehat{\L\TP}(S)$ we have the completed
    Beilinson filtration $\F^\star_\B\widehat{\L\TP}(S)$ with graded pieces
    $\gr^n_\B\widehat{\L\TP}(S)\we\L\W\Omega_S[2n]$.

    We thus have obtained a map $\widehat{\L\TP}(S)\rightarrow\TP(S)$ for any
    commutative ring $S$. If $S$ is quasiregular semiperfect, then
    $\L\W\Omega_S[2n]$ is $p$-adically concentrated in degree $2n$
    (see~\cite[Theorem~8.14]{bms2}). The completeness of the
    filtration implies that $\F^n_\B\widehat{\L\TP}(S)$ is $p$-adically
    $2n$-connective.
    Thus, since $\TP(S)$ is $p$-complete, the map $\widehat{\L\TP}(S)\rightarrow\TP(S)$ automatically upgrades
    to a filtered map $\F^\star_\B\widehat{\L\TP}(S)\rightarrow\tau_{\geq
    2\star}\TP(S)\we\F^\star_{\mathrm{BMS}}\TP(S)$.

    Now, by
    quasisyntomic descent, it follows that there exists a natural map
    $\F^\star_\B\TP(R)\rightarrow\F^\star_{\BMS}\TP(R)$ for any smooth
    $\FF_p$-algebra $R$. Standard arguments using the K\"unneth isomorphism in
    crystalline cohomology and
    \'etale descent reduce us to checking that it is
    an equivalence for $R=\FF_p[x]$, the ring of functions on $\AA^1_{\FF_p}$.
    But, one sees in this case that since the crystalline cohomology is
    concentrated in (homological) degrees $0$ and $-1$, both filtrations satisfy
    $\F^n_\B\TP(R)\we\tau_{\geq 2n-1}\TP(R)\we\F^n_{\BMS}\TP(R)$, so we are
    done.
\end{proof}

\begin{remark}
    We expect that there is a common refinement of the Beilinson and BMS
    filtrations on $\TP(R)$ when $R$ is smooth over a perfect field $k$ of
    characteristic $p$. This
    would say that there is a filtration $\F^\star\TP(R)$ on $\TP(R)$ where
    each $\F^n\TP(R)$ is a bifiltered spectrum with one filtration reducing to
    the Nygaard filtration on the graded pieces and the other reducing to the
    de Rham--Witt filtration on the graded pieces.
\end{remark}
 
In conclusion, we will briefly discuss three spectral sequences computing $\TP(X)$ when $X$ is smooth and
proper over a perfect ring $k$ of characteristic $p$: the Hesselholt spectral sequence, the
Bhatt--Morrow--Scholze spectral sequence (or BMS spectral sequence for short),
and a new spectral sequence arising from the cyclotomic Postnikov tower of
$\THH(X)$.

Hesseholt's spectral sequence, which is given
in~\cite{hesselholt-tp}*{Theorem~6.8} when $k$ is a perfect field of
characteristic $p$, is obtained Zariski locally from the inverse
limit of the Postnikov filtrations of $\THH(\Oscr_X)^{tC_p^n}$. It has the form
\begin{equation}\label{ss:hesselholt}
    \E^2_{i,j}=\bigoplus_{m\in\ZZ}\lim_{n,F}\H^{-i}(X,W_n\Omega_X^{j+2m})\Rightarrow\TP_{i+j}(X).
\end{equation}
The differentials $d^r$ have bidegree $(-r,r-1)$.
Note that, as explained in~\cite{hesselholt-tp}*{Section~5}, the terms in this
spectral sequence arise also in the conjugate spectral sequence
\begin{equation}\label{ss:conjugate}
    \E_2^{s,t}=\lim_{n,F}\H^s(X,W_n\Omega_X^t)\Rightarrow\H^{s+t}_{\mathrm{crys}}(X/W(k))
\end{equation}
computing the crystalline cohomology of $X$ over $W(k)$. The point is that
$\lim_{n,F}\H^s(X,W_n\Omega_X^t)$ is isomorphic to
$\H^s(X,\Hscr^t_{\mathrm{crys}})$ by the Cartier isomorphism, where
$\Hscr^t_{\mathrm{crys}}$ is the Zariski sheafification of
$U\mapsto\H^t_{\crys}(U/W(k))$.

The BMS spectral sequence arises because of a quasisyntomic-local filtration 
on $\TP$. The filtration then has $n$th graded piece
$$\gr^n_{\BMS}\TP(X)\we\R\Gamma_\crys(X/W(k))\{n\}[2n].$$ The twist $\{n\}$ is used
in~\cite{bms2} to keep track of the twist of the action of Frobenius. In this
case, it means that the action of Frobenius is given by $p^{-n}$ times the
action of Frobenius on $\R\Gamma_\crys(X/W(k))$. The spectral sequence of the
filtration then takes the form
\begin{equation}\label{ss:bms}
    \E_2^{s,t}=\H^{s-t}_{\crys}(X/W(k))\{-t\}\Rightarrow\TP_{-s-t}(X)
\end{equation}
(see~\cite{bms2}*{Theorem~1.12}).

Our spectral sequence for $\TP$ is induced from the
cyclotomic Whitehead tower. Specifically, the cyclotomic Postnikov tower gives
a complete exhaustive decreasing multiplicative $\NN$-indexed filtration
$\F^t_\cyc\THH(X)=\tau_{\geq t}^\cyc\THH(X)$ of $\THH(X)$ with associated graded pieces
$\mathrm{gr}^t_{\cyc}\THH(X)\we\R\Gamma(X,\W\Omega^t)[t]$.

\begin{theorem}
    Let $X$ be a smooth and quasi-compact $k$-scheme where $k$ is a
    perfect field of characteristic $p$. There is a complete decreasing multiplicative
    $\NN$-indexed filtration
    $\F^\star_\cyc\TP(X)$ on $\TP(X)$ with
    associated gradeds $\gr^j_\cyc\TP(X)\we\left(\pi_j^\cyc\THH(X)[j]\right)^{tS^1}$. The associated
    multiplicative and conditionally convergent\footnote{The filtration might in general not be finite. But if the scheme is of finite dimension it is, so that the spectral sequence is strongly convergent.} spectral sequence takes the form
    \begin{equation}\label{ss:cyclotomic}
        \E^1_{i,j}=\pi_*\left(\R\Gamma(X,\W\Omega^j_X)\right)[t^\pm]\Rightarrow\TP_{i+j}(X).
    \end{equation}
    where the bidegree of $t$ is $(-2,0)$ and the elements in $\pi_i(\R\Gamma(X,\W\Omega^j_X))$ sit in bidegree $(i,j)$. 
\end{theorem}

\begin{proof}
    It follows immediately from Theorem~\ref{thm:derhamwitt} and
    Corollary~\ref{mt:tate} that we get a spectral sequence of the form
    \[
    \E^1_{i,j}=\pi_i\left(\R\Gamma(X,\W\Omega^j_X)^{tS^1}\right)\Rightarrow\TP_{i+j}(X).
    \]
     Now the action of $S^1$ on $\W\Omega^j_X\we\pi_j\TR(\Oscr_X)$ is
    trivial since there is just a single homotopy groups. Thus the action on the derived sections is trivial as well (in fact, canonically trivialized) and the Tate spectral sequence for
    $\R\Gamma(X,\W\Omega^j_X)^{tS^1}$ degenerates for all $j$. Moreover the trivialization is multiplicative which gives the result. 
\end{proof}

\begin{remark}
    Theorem~\ref{thm:drw} implies that the spectral sequence~\eqref{ss:cyclotomic}
    agrees with~\eqref{ss:bms} when $X$ is affine from the $\E_2$-page forward (up to reindexing).
    In general, they do not agree on the $\E_2$-page.
\end{remark}

\begin{remark}
    If $X$ is smooth and proper over a perfect field $k$ of characteristic $p$,
    then the Hesselholt, BMS, and cyclotomic spectral sequences for $\TP$
    degenerate rationally. To prove degeneration of the spectral sequences for $\TP$ in the smooth proper
    case, it is enough to check it for one spectral sequence since rationally all
    three spectral sequences start with the same ranks contributing to
    $\TP_n(X)$ (using rational degeneration of the Hodge and conjugate spectral
    sequences for crystalline cohomology due to~\cite{illusie-raynaud}).
    However, an argument of Scholze proves rational degeneration of the BMS
    spectral sequence for smooth {\em affine} schemes over perfect fields
    (see~\cite{elmanto-tc}). This
    comes from a canonical splitting due to the action Adams operations and
    thus extends to smooth proper schemes, giving the desired degeneration.
\end{remark}

\appendix

\section{Background on $t$-structures}\label{sec:tbackground}

For further background on $t$-structures~\cite{ha}*{Chapter~1} or the original
source~\cite{bbd}.
For prestable and especially Grothendieck prestable $\infty$-categories, see~\cite{sag}*{Appendix C}.

\subsection{Right complete $t$-structures}\label{sub:rightcomplete}

The purpose of this section is to introduce some terminology on $t$-structures
and especially right complete $t$-structures.

\begin{definition}\label{def:t}
    Let $\Cscr$ be a stable $\infty$-category. A {\bf $t$-structure} on $\Cscr$
    is a pair $(\Cscr_{\geq 0},\Cscr_{\leq 0})$ of full subcategories of
    $\Cscr$ such that
    \begin{enumerate}
        \item[{\rm (1)}] $\Cscr_{\geq 0}[1]\subseteq\Cscr_{\geq 0}$ and $\Cscr_{\leq 0}[-1]\subseteq\Cscr_{\leq 0}$;
        \item[{\rm (2)}] if $X\in\Cscr_{\geq 0}$ and $Y\in\Cscr_{\leq 0}$, then
            the mapping space $\Map_{\Cscr}(X,Y[-1])$ is contractible;
        \item[{\rm (3)}] for every $X\in\Cscr$ there is a fiber sequence
            $\tau_{\geq 0}X\rightarrow X\rightarrow\tau_{\leq -1}X$ where
            $\tau_{\geq 0}X\in\Cscr$ and $\tau_{\leq -1}X[1]\in\Cscr_{\leq 0}$.
    \end{enumerate}
\end{definition}

\begin{remarks}
    \begin{enumerate}
        \item[(a)] It is useful to let $\Cscr_{\geq n}=\Cscr_{\geq 0}[n]$ and
            $\Cscr_{\leq n}=\Cscr_{\leq 0}[n]$ for integers $n$.
        \item[(b)] A $t$-structure is determined by either $\Cscr_{\geq 0}$
            or $\Cscr_{\leq 0}$. In other words, if $(\Cscr_{\geq
            0},\Cscr_{\leq 0})$ is a $t$-structure on $\Cscr$, then
            $\Cscr_{\leq 0}$ is the full subcategory of objects $Y\in\Cscr$
            such that $\Map_\Cscr(X,Y[-1])\we 0$ for all $X\in\Cscr_{\geq 0}$;
            similarly, $\Cscr_{\geq 0}$ is the left orthogonal to $\Cscr_{\leq
            -1}$. Therefore, we will often think of a $t$-structure as a pair
            consisting of a stable $\infty$-category $\Cscr$ and a full
            subcategory $\Cscr_{\geq 0}$ such that $(\Cscr_{\geq 0},\Cscr_{\leq
            0})$ defines a $t$-structure in the sense of
            Definition~\ref{def:t}, where $\Cscr_{\leq 0}$ is the right
            orthogonal to $\Cscr_{\geq 1}$.
        \item[(c)] If $(\Cscr_{\geq 0},\Cscr_{\leq 0})$ is a $t$-structure on
            $\Cscr$, then $\Cscr_{\geq 0}$ is often called the {\bf aisle} and
            $\Cscr_{\leq 0}$ the co-aisle. In this case, $\Cscr_{\geq 0}$ is an
            example of a {\bf prestable $\infty$-category}, a notion studied
            in~\cite{sag}*{Appendix~C} that abstracts
            the properties of the $\infty$-category of connective spectra.
        \item[(d)] The truncations of condition (3) are functorial: the
            inclusion $\Cscr_{\geq 0}\subseteq\Cscr$ admits a right adjoint
            $\tau_{\geq 0}$, and $\Cscr_{\leq 0}\subseteq\Cscr$ admits a left
            adjoint $\tau_{\leq 0}$. For $X\in\Cscr$, $\tau_{\geq
            0}X\rightarrow X$ is the counit of the adjunction for $\Cscr_{\geq
            0}\subseteq\Cscr$ and $X\rightarrow\tau_{\leq 0}X$ is the unit for
            the adjunction for $\Cscr_{\leq 0}\subseteq\Cscr$.
        \item[(e)] The {\bf heart} of a $t$-structure is the full subcategory
            $\Cscr^\heart=\Cscr_{\geq 0}\cap\Cscr_{\leq 0}$ of $\Cscr$. It is
            an abelian category, by~\cite{bbd}. It also coincides with the full
            subcategory of $0$-truncated objects in $\Cscr_{\geq 0}$,
            by~\cite{ha}*{1.2.1.9}.
    \end{enumerate}
\end{remarks}

\begin{definition}
    Let $\Cscr,\Dscr$ be stable $\infty$-categories equipped with
    $t$-structures $(\Cscr_{\geq 0},\Cscr_{\leq 0})$ and $(\Dscr_{\geq
    0},\Dscr_{\leq 0})$, respectively. An exact functor
    $F\colon\Cscr\rightarrow\Dscr$ is {\bf right $t$-exact} if
    $F(X)\in\Dscr_{\geq 0}$ whenever $X\in\Cscr_{\geq 0}$. Similarly, the exact
    functor $F$ is {\bf left $t$-exact} if $F(X)\in\Dscr_{\leq 0}$ whenever
    $X\in\Cscr_{\leq 0}$. An exact functor $F$ is {\bf $t$-exact} if it is left
    and right $t$-exact.
\end{definition}

Now, we review some essentially well-known facts about right completions
of $t$-structures.

\begin{definition}
    Let $\Cscr$ be a stable $\infty$-category equipped with a $t$-structure
    $(\Cscr_{\geq 0},\Cscr_{\leq 0})$.
    \begin{itemize}
        \item   We say that $\Cscr$ is {\bf right separated} if the full
            subcategory of $\infty$-coconnective objects
            $$\bigcap_{n}\Cscr_{\leq n}$$ is contractible.
        \item We say that $\Cscr$ is {\bf right complete} if the natural
            map $$\Cscr\rightarrow\lim_n\Cscr_{\geq
            n}\we\lim\left(\cdots\rightarrow\Cscr_{\geq n}\xrightarrow{\tau_{\geq
            n+1}}\Cscr_{\geq n+1}\rightarrow\cdots\right)$$ is an equivalence.
        \end{itemize}
\end{definition}

\begin{remark}
    {\bf Left separated} and {\bf left complete} $t$-structures are defined in the
    analogous way.
\end{remark}

\begin{lemma}
    Let $\Cscr$ be a stable $\infty$-category with a $t$-structure
    $(\Cscr_{\geq 0},\Cscr_{\leq 0})$. Assume that $\Cscr$ admits countable
    coproducts and that $\Cscr_{\leq 0}\subseteq\Cscr$ is closed under countable
    coproducts. Then, $\Cscr$ is right complete if and only it is right
    separated.
\end{lemma}

\begin{proof}
    This is the right complete version of~\cite{ha}*{1.2.1.19}.
\end{proof}

\begin{definition}
    If $\Cscr$ is a stable $\infty$-category with a $t$-structure $(\Cscr_{\geq
    0},\Cscr_{\leq 0})$, then the $\infty$-category $$\lim\left(\cdots\rightarrow\Cscr_{\geq n}\xrightarrow{\tau_{\geq
    n+1}}\Cscr_{\geq n+1}\rightarrow\cdots\right)$$ is the {\bf right
    completion} of $\Cscr$.
\end{definition}

\begin{lemma}
    If $\Cscr$ is a stable $\infty$-category with a $t$-structure $(\Cscr_{\geq
    0},\Cscr_{\leq 0})$, then the
    right completion is naturally equivalent to $\Sp(\Cscr_{\geq 0})$, the
    $\infty$-category of spectrum objects in $\Cscr_{\geq 0}$.
\end{lemma}

\begin{proof}
    Write $\iota_{n+1}$ for the inclusion $\Cscr_{\geq n+1}\subseteq\Cscr_{\geq
    n}$. In other words, $\iota_{n+1}$ is the left adjoint to $\tau_{\geq n+1}$.
    We may rewrite the limit $$\lim\left(\cdots\rightarrow\Cscr_{\geq n}\xrightarrow{\tau_{\geq
    n+1}}\Cscr_{\geq
    n+1}\rightarrow\cdots\right)\we\lim\left(\cdots\rightarrow\Cscr_{\geq
    -2}\xrightarrow{\tau_{\geq -1}}\Cscr_{\geq
    -1}\xrightarrow{\tau_{\geq 0}}\Cscr_{\geq 0}\right)$$ as
    \begin{equation}\label{eq:rightlimit}
        \lim\left(\cdots\rightarrow\Cscr_{\geq 0}\xrightarrow{\Omega}\Cscr_{\geq
        0}\xrightarrow{\Omega}\Cscr_{\geq
        0}\right).
    \end{equation}
    If $\Dscr$ is an $\infty$-category with finite limits,
    Lurie defines $\Sp(\Dscr)$ in~\cite{ha}*{1.4.2.8} as the full
    subcategory of reduced, excisive functors
    $\Sscr^{\mathrm{fin}}_\ast\rightarrow\Cscr_{\geq 0}$, where
    $\Sscr^{\mathrm{fin}}_\ast$ is the full subcategory of pointed spaces
    on those spaces that can be built out of finite colimits from
    a point.\footnote{Thus, the idempotent completion of
    $\Sscr^{\mathrm{fin}}_\ast$ is $\Sscr^\omega_\ast$.}
    Lurie later shows, in~\cite{ha}*{1.4.2.24}, that $\Sp(\Dscr)$
    is equivalent to the limit
    $\cdots\rightarrow\Dscr\xrightarrow{\Omega}\Dscr\xrightarrow{\Omega}\Dscr$. Since
    $\Cscr_{\geq 0}$ has finite limits, we see that $\Sp(\Cscr_{\geq 0})$ is
    the right completion, as claimed.
\end{proof}

The right completion $\Sp(\Cscr_{\geq 0})$
can be identified with the full subcategory of $\Fun(\ZZ,\Cscr)$ consisting of
those functors $X(\star)\colon\ZZ\rightarrow\Cscr$ such that $X(m)\in\Cscr_{\geq m}$
for all $m$
and $X(m)\rightarrow X(n)$ induces an equivalence $\tau_{\geq n}X(m)\we X(n)$
for $m\leq n$. Let $\Sp(\Cscr_{\geq 0})_{\geq 0}\subseteq\Sp(\Cscr_{\geq 0})$ denote the full
subcategory of those objects $X(\star)$ where, additionally, $X(m)\in\Cscr_{\geq 0}$
for all $m$. Let $\Sp(\Cscr_{\geq 0})_{\leq 0}$ be the full subcategory of
those sequences $X(\star)$ in $\Sp(\Cscr_{\geq 0})$ where, additionally,
$X(m)\in\Cscr_{\leq 0}$ for all $m$.

\begin{lemma}
    Let $\Cscr$ be a stable $\infty$-category with a $t$-structure
    $(\Cscr_{\geq 0},\Cscr_{\leq 0})$.
    \begin{enumerate}
        \item[{\rm (a)}] The right completion $\Sp(\Cscr_{\geq 0})$ is stable.
        \item[{\rm (b)}] The full subcategories $(\Sp(\Cscr_{\geq 0})_{\geq
            0},\Sp(\Cscr_{\geq 0})_{\leq 0})$ define a right complete
            $t$-structure on $\Sp(\Cscr_{\geq 0})$.
        \item[{\rm (c)}] The natural functor $\Sp(\Cscr_{\geq
            0})\rightarrow\Cscr$ is right $t$-exact (and in particular exact).
        \item[{\rm (d)}] The functor in (c) induces an equivalence $\Cscr_{\geq
            0}\we\Sp(\Cscr)_{\geq 0}$.
    \end{enumerate}
\end{lemma}

\begin{proof}
    Part (a) follows from~\cite{ha}*{1.4.2.21} since $\Cscr_{\geq 0}$ has
    finite limits. The remainder is a right complete version
    of~\cite{ha}*{1.2.1.17}.
\end{proof}

\subsection{Compatibility with symmetric monoidal structures}\label{sub:compatible}

In this section, we discuss the interaction of $t$-structures and tensor
products.

\begin{definition}
    Let $\Cscr$ be a stable $\infty$-category equipped with a $t$-structure
    $(\Cscr_{\geq 0},\Cscr_{\leq 0})$ and a symmetric monoidal structure
    $\Cscr^\otimes$. We say that the $t$-structure is {\bf compatible} with the
    symmetric monoidal structure if the following conditions hold:
    \begin{enumerate}
        \item[(i)] the tensor product $\Cscr\times\Cscr\xrightarrow{-\otimes-}\Cscr$ is
            exact in each variable;
        \item[(ii)] the unit $\1_\Cscr$ is in $\Cscr_{\geq 0}$;
        \item[(iii)] $X\otimes Y\in\Cscr_{\geq 0}$ whenever $X,Y\in\Cscr_{\geq 0}$.
    \end{enumerate}
\end{definition}

\begin{example}
    The Postnikov $t$-structure on the $\infty$-category of spectra is
    compatible with the smash product symmetric monoidal structure. Indeed, the
    tensor product commutes with all colimits in each variable, the
    sphere spectrum is connective, and the smash product of connective spectra
    is connective.
\end{example}

In general, if $(\Cscr_{\geq 0},\Cscr_{\leq 0})$ is a $t$-structure compatible
with $\Cscr^\otimes$, then the $\infty$-category $\Cscr_{\geq 0}$ and the fully
faithful inclusion functor $\Cscr_{\geq 0}\rightarrow\Cscr$ inherit unique
$t$-structures. For example, $\Cscr_{\geq 0}^\otimes\subseteq\Cscr^\otimes$ is
the full subcategory spanned by the objects
$(X_1,\ldots,X_n)\in\Cscr^\otimes_{\langle n\rangle}\we\Cscr^n$ such that $X_i\in\Cscr_{\geq 0}$
for each $i$.

\begin{lemma}\label{lem:monoidalloc}
    Let $\Cscr$ be a stable $\infty$-category equipped with a symmetric
    monoidal structure $\Cscr^\otimes$ and a $t$-structure $(\Cscr_{\geq
    0},\Cscr_{\leq 0})$ that is compatible with $\Cscr^\otimes$. Then, there is
    a unique symmetric monoidal structure on $\Cscr^\heart$ and a unique
    symmetric monoidal structure on the functor $\pi_0\colon\Cscr_{\geq
    0}\rightarrow\Cscr^\heart$.
\end{lemma}

\begin{proof}
    The proof is the same as that of~\cite{nikolaus-scholze}*{Theorem~I.3.6}
    since $\Cscr_{\geq 0}\rightarrow\Cscr^\heart$ is a localization at the
    class $W$ of maps $X\rightarrow Y$ such that $\pi_0X\rightarrow\pi_0Y$ is
    an isomorphism in $\Cscr^\heart$. We just have to check that if $f\colon
    X\rightarrow Y$ is in the class $W$ and if $Z$ is an arbitrary object of
    $\Cscr_{\geq 0}$, then $X\otimes Z\xrightarrow{f\otimes\id_Z}Y\otimes Z$ is
    in $W$. This is a trivial consequence of the next lemma.
\end{proof}

\begin{lemma}\label{lem:tensorheart}
    In the situation of Lemma~\ref{lem:monoidalloc},
    for any pair $X,Y\in\Cscr_{\geq 0}$, the natural map $\pi_0(X\otimes Y)\rightarrow\pi_0(\pi_0 X\otimes \pi_0 Y)$ is an
    isomorphism in $\Cscr^\heart$.
\end{lemma}

\begin{proof}
    To see this, consider the $3\times 3$ diagram obtained by tensoring
    $\tau_{\geq 1} X\rightarrow X\rightarrow\pi_0 X$ with the corresponding
    sequence for $Y$. We get a commutative diagram
    $$\xymatrix{
        \tau_{\geq 1} X \otimes \tau_{\geq 1} Y\ar[r]\ar[d]  & X\otimes \tau_{\geq 1} Y \ar[r]\ar[d]&\pi_0 X\otimes\tau_{\geq 1} Y\ar[d]\\
        \tau_{\geq 1} X\otimes Y\ar[r]\ar[d]  & X\otimes Y\ar[r]\ar[d]&\pi_0X\otimes
    Y\ar[d]\save"2,2"."3,3"*[magenta]\frm<8pt>{.}\restore \\
    \tau_{\geq 1} X\otimes \pi_0 Y\ar[r]  & X\otimes \pi_0 Y \ar[r]&\pi_0 X\otimes\pi_0 Y,\\
    }$$
    where each row and column is a fiber sequence in $\Cscr$. Any term with at least one $\tau_{\geq
    1}$ is in $\Cscr_{\geq 1}$. Hence, applying $\pi_0$, we see using the
    bottom right four objects that $$\pi_0(X\otimes Y)\we\pi_0(\pi_0 X\otimes\pi_0 Y),$$
    as desired.
\end{proof}

\begin{definition}
    We say that a $t$-structure on $\Cscr$ is {\bf compatible with countable
    products} if $\Cscr$ admits countable products and $\Cscr_{\geq
    0}\subseteq\Cscr$ is closed under countable products. Similarly, we say that
    $\Cscr$ is {\bf compatible with countable coproducts} if $\Cscr$ admits countable
    coproducts and $\Cscr_{\leq 0}\subseteq\Cscr$ is closed under countable
    coproducts. We also say that a $t$-structure on $\Cscr$ is {\bf
    compatible with products} if $\Cscr$ admits products and $\Cscr_{\geq
    0}\subseteq\Cscr$ is closed under products.
\end{definition}

For the next proposition, let $\P$ be one of the following properties of a
$t$-structure: left separated, right separated, compatible with countable
products, or compatible with products.

\begin{proposition}\label{prop:permanence}
    Let $\Cscr$ be a presentably symmetric monoidal stable $\infty$-category
    equipped with a compatible accessible $t$-structure $(\Cscr_{\geq 0},\Cscr_{\leq 0})$.
    Let $A\in\Alg_{\EE_1}(\Cscr_{\geq 0})$.
    \begin{enumerate}
        \item[{\rm (1)}] Then, $\Mod_A(\Cscr)$ admits an accessible
            $t$-structure with $\Mod_A(\Cscr)_{\geq 0}\we\Mod_A(\Cscr_{\geq
            0})$.
        \item[{\rm (2)}] If the $t$-structure on $\Cscr$ satisfies property
            $\P$, then so does the $t$-structure on $\Mod_A(\Cscr)$.
        \item[{\rm (3)}] There is a natural equivalence
            $\Mod_A(\Cscr)^\heart\we\Mod_{\pi_0A}(\Cscr^\heart)$.
        \item[{\rm (4)}] If $A$ is in $\Alg_{\EE_\infty}(\Cscr_{\geq 0})$, then
            $\Mod_A(\Cscr)$ admits a presentably symmetric monoidal structure
            and the $t$-structure is compatible with the symmetric monoidal
            structure.
    \end{enumerate}
\end{proposition}

\begin{proof}
    For (1) it is enough by~\cite{ha}*{1.4.4.11} to note that $\Mod_A(\Cscr_{\geq
    0})\subseteq\Mod_A(\Cscr)$ are presentable (see~\cite{ha}*{4.2.3.7}), the
    inclusion preserves colimits, and
    that $\Mod_A(\Cscr)$ is closed under extensions in $\Mod_A(\Cscr)$.
    For (2), we will use that the forgetful functor
    $\Cscr\leftarrow\Mod_A(\Cscr)$ is conservative and $t$-exact. Indeed, right
    $t$-exactness follows by definition and left $t$-exactness follows from the
    fact that the forgetful functor is right adjoint to the right $t$-exact
    extension of scalars functor. Now, (2) for left or right separatedness
    follows from conservativity and $t$-exactness. For compatibility with
    (countable) products, we use conservativity, $t$-exactness, and preservation
    of limits. Part (3) is a general fact about symmetric monoidal localizations. Part (4)
    follows from~\cite{ha}*{4.5.2.1}.
\end{proof}

\begin{remark}\label{rem:completeness}
    By~\cite{ha}*{1.2.1.19}, compatibility with countable products and left
    separated together imply left complete. Similarly, compatibility with
    countable coproducts and right separated together imply right complete.
\end{remark}

In one way or another, we will typically be starting with a $t$-structure on
spectra with an $S^1$-action. This is a special case of a $t$-structure on
parametrized spectra, which we introduce below.

Let $X$ be a space and let $\Sp^X=\Fun(X,\Sp)$ be the $\infty$-category of
spectra parametrized over $X$. There is a natural $t$-structure on $\Sp^X$, the
{\bf Postnikov} $t$-structure, where $(\Sp^X)_{\geq 0}=\Fun(X,\Sp_{\geq 0})$
and $(\Sp^X)_{\leq 0}=\Fun(X,\Sp_{\leq 0})$. The heart $\Sp^{X,\heart}$ is
naturally equivalent to $\prod_{x\in\pi_0X}\Mod_{\ZZ[\pi_1(X,x)]}^\heart$, the
product over the connected components of $X$ of the abelian group of discrete
$\ZZ[\pi_1(X,x)]$-modules, where $\ZZ[\pi_1(X,x)]$ is the group algebra of
$\pi_1(X,x)$.

\begin{proposition}\label{prop:xsp}
    Let $X$ be a space. The Postnikov $t$-structure on $\Sp^X$ is
    \begin{enumerate}
        \item[{\rm (a)}] accessible,
        \item[{\rm (b)}] left and right complete,
        \item[{\rm (c)}] compatible with products and filtered colimits, and
        \item[{\rm (d)}] compatible with the pointwise symmetric monoidal structure
            \[
            \left(\Sp^{X}\right)^{\otimes}=\Fun(X,\Sp^\otimes) \times_{\Fun(X, \operatorname{Fin}_*)} \operatorname{Fin}_*
            \]
            on $\Sp^X$.
    \end{enumerate}
\end{proposition}

\begin{proof}
    Statement (d) follows immediately from the definitions. For the rest, note
    that each property is stable under products of stable $\infty$-categories. Thus, it is
    enough to check the case when $X$ is path-connected. Let $x\in X$ be a
    point. Then, $\Sp^X\xrightarrow{x^*}\Sp$ admits a left adjoint $x_!$.
    Moreover, $x_!\SS$ is a compact generator of $\Sp^X$~\cite{nikolaus-scholze}*{Theorem~I.4.1}.
    Thus, $\Sp^X\we\Mod_{\EndSp_{\Sp^X}(x_!\SS)}$ by Morita theory, where
    $\EndSp_{\Sp^X}(x_!\SS)$ denotes the endomorphism algebra spectrum of
    $x_!\SS$. Now, $\EndSp_{\Sp^X}(x_!\SS)\we\Map_\Sp(\SS,x^*x_!\SS)\we\SS[\Omega_xX]$, the
    spherical group algebra of the grouplike $\EE_1$-space $\Omega_xX$.
    Since $\SS[\Omega_xX]$ is a connective $\EE_1$-algebra,
    $\Mod_{\SS[\Omega_xX]}$ admits an accessible, left and right complete
    $t$-structure which is compatible with products and filtered colimits
    by~\cite{ha}*{Proposition~7.1.1.13}. Thus, it is enough to note that the
    two $t$-structures agree, which follows from the fact $M\in\Sp^X_{\geq 0}$
    if and only if $x^*M$ is in $\Sp_{\geq 0}$.
\end{proof}


\vspace{20pt}
\scriptsize
\noindent
Benjamin Antieau\\
University of Illinois at Chicago\\
Department of Mathematics, Statistics, and Computer Science\\
851 South Morgan Street, Chicago, IL 60607\\
USA\\
\texttt{benjamin.antieau@gmail.com}

\vspace{10pt}
\noindent
Thomas Nikolaus\\
Universit\"at M\"unster\\
FB Mathematik und Informatik\\ 
Einsteinstr. 62 D-48149 M\"unster\\
Germany\\
\texttt{nikolaus@uni-muenster.de}

\end{document}